\documentclass[11pt]{amsart}
\usepackage{amsmath,amssymb}
\usepackage{bbm}
\usepackage{graphicx,tikz,mathtools}
\usepackage[left=2.65cm,right=2.65cm,top=2.7cm,bottom=2.7cm]{geometry}
\usepackage[pdftex]{hyperref}
\usepackage{cite}
\usepackage{mathrsfs} 
\usepackage{caption}
\hypersetup{
	colorlinks=true,
	linkcolor=blue, 
	citecolor=blue,
	filecolor=blue,
	urlcolor=blue,
}

\makeatletter
\@namedef{subjclassname@2020}{\textup{2020} Mathematics Subject Classification}
\makeatother

\newtheorem{theorem}{Theorem}[section]
\newtheorem{proposition}[theorem]{Proposition}
\newtheorem{corollary}[theorem]{Corollary}

\newtheorem{lemma}[theorem]{Lemma}
\newtheorem*{conjecture*}{Conjecture}

\newcommand{\triple}[1]{{\left\vert\kern-0.25ex\left\vert\kern-0.25ex\left\vert #1
        \right\vert\kern-0.25ex\right\vert\kern-0.25ex\right\vert}}

\theoremstyle{definition}

\newtheorem{remark}[theorem]{Remark}

\newcommand{\al}{\alpha}

\newcommand{\ep}{\varepsilon}

\newcommand{\la}{\lambda}
\newcommand{\om}{\omega}

\newcommand{\te}{\theta}
\newcommand{\vp}{\varphi}

\newcommand{\De}{\Delta}
\newcommand{\Ga}{\Gamma}

\newcommand{\Om}{\Omega}

\newcommand{\opsi}{\overline{\psi}}
\newcommand{\of}{\overline{f}}

\newcommand{\tv}{\widetilde{v}}

\newcommand{\fL}{\mathfrak{L}}
\def\NN{\mathbb{N}}
\def\RR{\mathbb{R}}

\def\ZZ{\mathbb{Z}}

\def\TT{\mathbb{T}}

\def\bD{\mathbb{D}}

\newcommand\bsharp{b^{\sharp}}
\newcommand\Bsharp{B^{\flat}}

\newcommand{\dds}{\frac{{\rm d}}{{\rm d}s}}

\newcommand{\cB}{{\mathcal B}}
\newcommand{\cC}{{\mathcal C}}

\newcommand{\cE}{{\mathcal E}}
\newcommand{\cF}{{\mathcal F}}
\newcommand{\cG}{{\mathcal G}}

\newcommand{\cK}{{\mathcal K}}

\newcommand{\cL}{{\mathcal L}}

\newcommand{\cO}{{\mathbf O}}
\newcommand{\cV}{{\mathbf V}}

\newcommand{\cQ}{{\mathcal Q}}
\newcommand{\cR}{{\mathcal R}}

\newcommand{\cS}{{\mathcal S}}
\newcommand{\cT}{{\mathcal T}}
\newcommand{\cX}{{\mathbf X}}
\newcommand{\cY}{{\mathbf Y}}

\newcommand{\pd}{\partial}
\newcommand\minus\backslash

\newcommand\lan\langle
\newcommand\ran\rangle

\DeclareMathOperator\Div{div}

\DeclareMathOperator\dist{dist}

\renewcommand\leq\leqslant
\renewcommand\geq\geqslant

\newcommand\BOm{\overline\Om}

\newcommand\m{(m-1)(m-2)}
\newcommand\erre{R+a-4}
\newcommand{\Omg}{\Om^{{\rm far}}}
\newcommand\tOm{\widetilde{\Om}}
\newcommand\trho{\tilde{\rho}}

\numberwithin{equation}{section}

\newcommand\tcT{\widetilde\cT}
\newcommand{\tvp}{\widetilde\vp}
\newcommand\tu{\widetilde u}

\DeclareMathOperator{\ind}{ind}
\newcommand{\hcG}{\widehat\cG}

\begin{document}

\title[Smooth nonradial stationary Euler flows with compact support]{Smooth nonradial stationary Euler flows\\ on the plane  with compact support}

\author[A. Enciso]{Alberto Enciso}
\address{   
\newline
\textbf{{\small Alberto Enciso}} 
\vspace{0.15cm}
\newline \indent Instituto de Ciencias Matem\'aticas, Consejo Superior de Investigaciones Cient\'\i ficas, 28049 Madrid, Spain}
\email{aenciso@icmat.es}

 \author[A. J. Fern\'andez]{Antonio J.\ Fern\'andez}
 \address{ \vspace{-0.4cm}
\newline 
\textbf{{\small Antonio J. Fern\'andez}} 
\vspace{0.15cm}
\newline \indent Departamento de Matem\'aticas, Universidad Aut\'onoma de Madrid, 28049 Madrid, Spain}
 \email{antonioj.fernandez@uam.es}

 \author[D. Ruiz]{David Ruiz}
 \address{ \vspace{-0.4cm}
\newline 
\textbf{{\small David Ruiz}} 
\vspace{0.15cm}
\newline \indent  
 IMAG, Departamento de An\'alisis Matem\'atico, Universidad de Granada, 18071 Granada, Spain}
 \email{daruiz@ugr.es}



%
%
\begin{abstract}  \vspace{-0.2cm}
We prove the existence of nonradial classical solutions to the 2D incompressible Euler equations with compact support. More precisely, for any positive integer~$k$, we construct compactly supported stationary Euler flows of class $C^k(\RR^2)$ which are not locally radial. The proof uses a degree-theory-based bifurcation argument which hinges on three key ingredients: a novel approach to stationary Euler flows through elliptic equations with non-autonomous nonlinearities; a set of sharp regularity estimates for the linearized operator, which involves a potential that blows up as the inverse square of the distance to the boundary of the support; and overcoming a serious problem of loss of derivatives by the introduction of anisotropic weighted functional spaces between which the linearized operator is Fredholm.
\end{abstract}

\maketitle

\section{Introduction}

Let us consider stationary solutions to the incompressible Euler equations on the plane
\begin{equation}\label{E.Euler} 
v\cdot \nabla v+\nabla p=0\quad\textup{ and } \quad \Div v=0 \quad \textup{ in } \RR^2\,,
\end{equation}
which describe the steady flows of an ideal fluid. Since the velocity field $v$ is solenoidal, one can write it as the perpendicular gradient of the stream function $\psi$, i.e., $v=\nabla^\perp\psi:=(\pd_{x_2}\psi,-\pd_{x_1}\psi)$. In terms of the stream function, the stationary Euler equations can be equivalently written as
\begin{equation} \label{stream} \nabla^\perp\psi\cdot\nabla\Delta\psi=0\quad \textup{ in } \RR^2\,. 
\end{equation}

In this paper we are concerned with compactly supported stationary solutions to the Euler equations. In the three-dimensional case, the existence of such solutions was a long-standing open problem, and stationary Euler flows with compact support were obtained only recently~\cite{Gavrilov,CV}. In contrast, in $\RR^2$, the construction of compactly supported solutions to \eqref{E.Euler} is elementary: it suffices to pick any radially symmetric, compactly supported stream function. More generally, one can take~$\psi$ as a linear combination of radially symmetric functions with disjoint compact supports, possibly centered at distinct points. The corresponding stationary flows are then {\em locally radial}\/, and their support is a union of disjoint balls and annuli.

Although in the last few years there has been an emergence of rigidity results for steady two-dimensional fluids, the existence of nonradial classical solutions to the stationary Euler equations with compact support remains a well-known open problem. There do exist smooth nonradial stationary solutions with finite energy, which are not compactly supported, as a byproduct of the results  in \cite{Musso}. Moreover, in the context of wild solutions, which are only $L^\infty$, a wealth of compactly supported solutions can be constructed   using convex integration~\cite{Chof}. 

From the point of view of the regularity of nonradial compactly supported solutions, the best results to date only give Lipschitz velocities. Specifically, by means of a hard proof involving several clever observations and a Nash--Moser iteration scheme, G\'omez-Serrano, Park and Shi~\cite{JGS} constructed nonradial solutions of vortex patch type with compactly supported velocity. More precisely, the vorticity $-\Delta\psi$ is a linear combination of three indicator functions, so~$v$ is piecewise smooth but not~$C^1$. Another family of nonradial compactly supported solutions of the same regularity, which are not of vortex patch type, was subsequently obtained by the authors and Sicbaldi in~\cite{EFRS} as a byproduct of a result (somewhat related to the so-called Schiffer conjecture~\cite[Problem 80]{Yau}) on nontrivial Neumann eigenfunctions that are locally constant on the boundary. Here, the stream function~$\psi$ satisfies an equation of the form
\begin{equation}\label{E.semilinear}
\Delta\psi + f(\psi)=0\,,
\end{equation}
in the support of the velocity, and in fact the function $f$ is linear.


Our objective in this paper is to show that there are smooth nonradial stationary flows with compact support.

\begin{theorem}\label{T.main}
For any positive integer $k $, there exist compactly supported stationary Euler flows of class $C^k(\RR^2)$ that are not locally radial.
\end{theorem}

\subsection{Strategy of the proof}

Theorem~\ref{T.main} relies on a bifurcation argument: nonradial stationary flows with compact support branch out from a suitably chosen family of radially symmetric, compactly supported flows. 
These radial flows are described by a one-parameter family of radial stream functions $\psi_a$ which are supported on certain annuli $\Omega_a$ and which vanishes on~$\pd\Omega_a$ to a high order~$m\geq1$. In a nutshell, the way one aims to implement a bifurcation argument is by ensuring that, inside their support, $\psi_a$ satisfies certain differential equation. The gist of the  argument is to show that, for some value of the parameter~$a$, one can consider a smooth small nonradial deformation~$\psi$ of~$\psi_a$ which satisfies the same equation on a slightly deformed domain~$\tOm$. It is then easy to see that if~$\psi$ also vanishes on the boundary of the deformed domain  to order~$m$, then the vector field defined by $v:=\nabla^\perp\psi$ on the domain and $v:=0$ outside is of class~$C^{m-1}(\RR^2)$. The equation satisfied by~$\psi$ must therefore ensure that its perpendicular gradient $v:=\nabla^\perp \psi$ is a stationary solution to the Euler equations~\eqref{E.semilinear}.

For a bifurcation argument, it is known that one cannot directly use the Euler equation~\eqref{stream}, since its linerization is a completely unmanageable operator with an infinite-dimensional kernel. Vortex patch solutions are not~$C^1$, so it is not clear how one could adapt the strategy of~\cite{JGS}. Also, a variation of Gavrilov's construction can only give locally radial solutions of compact support \cite{Shara}.  One would naively think that the elliptic equation~\eqref{E.semilinear} should be the way to go, but in fact this is not true: in Theorem~\ref{T.nogoA} we show that any compactly supported stationary flow of class~$C^2$ whose stream function satisfies a semilinear equation of the form~\eqref{E.semilinear} must be locally radial. 

Hence, in this problem, even the choice of the equation one should consider is rather nontrivial. For us, the starting point of the paper is the construction of a {\em non-autonomous}\/ nonlinearity~$f_a$ enabling us to effectively use the equation
\begin{equation}\label{E.nonradialsemi}
\Delta\psi_a+f_a(|x|,\psi_a)=0	\,,
\end{equation}
to construct compactly supported solutions. To 
our best knowledge, this is the first time that non-autonomous elliptic equations have been used for a similar purpose. 

Still, passing from this rough idea to an actual proof is remarkably hard. This is because the above outline does not address the three essential difficulties that the problem entails:
\begin{enumerate}
\item The radial stream function~$\psi_a$ that we consider is a positive solution to an equation of the form~\eqref{E.nonradialsemi} on an annulus~$\Om_a$, which vanishes to $m$-th order on~$\pd\Om_a$. The deformation $\psi$ will satisfy the same equation on the deformed domain~$\tOm$ and tend to zero as $\trho^m$ on~$\pd\tOm$. However, even with our choice of the nonlinearity, if~$\psi$ is a nonradial solution to~\eqref{E.nonradialsemi} in~$\tOm$, $\nabla^\perp\psi$ does {\em not}\/ satisfy the stationary Euler equations: this is only true if $\psi$ is close to~$\psi_a$ in a certain sense (see Lemma~\ref{L.f}).\smallskip

	\item Suppose that the function $\psi$ vanishes to order~$m\geq3$ on~$\pd\tOm$, where $\tOm$ can be thought of as a slightly deformed annulus. For concreteness, we can think that $\psi=\trho^m u$, where $u$ is a smooth function that does not vanish on~$\pd\tOm$ and where~$\trho$ is a boundary defining function, that is, a positive function on~$\tOm$ that vanishes on~$\pd\tOm$ exactly to first order. Since $\Delta\psi$ goes like $\trho^{m-2}$ near~$\pd\tOm$, if $\psi$ satisfies a semilinear equation like~\eqref{E.nonradialsemi}, the nonlinearity $f(|x|,t)$ can only be H\"older continuous and must behave like $|t|^{1-\frac2 m}$ near $0$. Thus, the linearization of this equation, which one expects to encounter in any bifurcation argument, will be controlled by an operator of the form
$	L=-\Delta + \frac c{\trho^2}$	for some nonzero constant~$c$ (modulo terms that are less singular). The potential term is then critically singular (i.e., it scales like the Laplacian), so it cannot be treated as a perturbation of~$\Delta$: a new set of estimates is necessary.\smallskip

\item To control the deformation of the domain, one would naively parametrize the deformed domain, say in polar coordinates, as $\tOm:=\{a_-+b(\theta) < r< a_++B(\theta)\}$, where $b,B$ are small functions on the circle $\TT:=\RR/2\pi\ZZ$. However, this approach is known to lead to a serious loss of derivatives \cite{EFRS,Fall-Minlend-Weth-main} which one does not know how to compensate using a Nash--Moser iteration scheme. 
\end{enumerate}

Furthermore, the complexity of the problem resides on the fact that these difficulties are strongly interrelated. As we have already mentioned, the starting point of our approach is the construction of a non-autonomous nonlinearity $f_a(r,t)$ such that any solution to~\eqref{E.nonradialsemi} that is close enough to our radial solution~$\psi_a$ does define a stationary Euler flow. This has to be done carefully, as typical non-autonomous nonlinearities do not have this crucial property. Roughly speaking, the key idea is to construct the non-autonomous nonlinearity~$f_a$ and the radial function~$\psi_a$ such that $\psi_a$ solves an autonomous equation of the form $\Delta\psi_a + f_{a,p}(\psi_a)=0$ in a small neighborhood of each point~$p$. This local property suffices to show that the vector field $\nabla^\perp\psi$ satisfies the stationary Euler equation, and we can effectively control~$\psi$ in the arguments using the global equation~\eqref{E.nonradialsemi}. 

Near $t=0$, the non-autonomous nonlinearity $f_a(r,t)$ has the asymptotic behavior described in item~(i) above, and  the most singular part of the linearized operators we need to consider in the bifurcation argument is indeed like the aforementioned operator~$L$. In general, the theory of uniformly degenerate elliptic operators~\cite{Mazzeo,GrahamLee, Lee} (which is essentially a sophisticated PDE analog of the Frobenius theory for ODEs with regular singular points) is well suited to the study of operators such as $ L$. The key concept here is that of the {\em indicial roots}\/ of the operator, defined as the constants $\nu$ for which $L (\trho^\nu)=O(\trho^{\nu-1})$ in a neighborhood of a boundary component. Denoting by $\underline\nu,\overline\nu$ the smallest and largest indicial roots of~$L$, the rule of thumb for well behaved operators of this kind is that $\trho^2 L$ defines a Fredholm map between spaces of functions that are of order $\trho^\nu$ near~$\pd\tOm$, provided that~$\nu$ is not an indicial root and $\underline\nu<\nu<\overline\nu$. These function spaces are defined using the scale-natural H\"older or Sobolev norms
\begin{equation}\label{E.badnorm}
\begin{aligned}
\|\psi\|_{C^{j,\alpha}_\nu}&:= \sum_{l=0}^j\|\trho^{-\nu+l}\nabla^l \psi\|_{L^\infty(\tOm)}+ \sup_{x,x'\in\tOm} \trho(x)^{-\nu+j+\alpha}\frac{|\nabla^j \psi(x)-\nabla^j \psi(x')|}{|x-x'|^\alpha}\,,\\
\|\psi\|_{H^j_\nu}&:=\sum_{l=0}^j\|\trho^{-\nu+l-1}\nabla^l \psi\|_{L^2(\tOm)}\,.
\end{aligned}
\end{equation}

However, in our problem we find two major complications that prevent us from basing our approach on off-the-shelf estimates and spaces. Firstly, we find that our linearized operator has exactly two indicial roots, which are completely determined by the decay rate~$m$ as $\underline\nu= 2-m$ and $\overline\nu=m-1$, and we crucially need to control functions that are critical in that they behave like $\trho^{\overline\nu}$. Secondly, the fact that the function~$\trho$ vanishes on~$\pd\tOm$ makes the norms~\eqref{E.badnorm} strongly unsuitable to control the nonlinearities that arise in our problem. Hence we need to develop from scratch a sharp regularity theory for this kind of operators that is adapted to the situation at hand.

This has a major impact in the way we address the loss of derivatives. In the approach introduced in \cite{Fall-Minlend-Weth-main} and used in the papers~\cite{Fall-Minlend-Weth-main,EFRS} about analogs of the Schiffer problem, the loss of derivatives is overcome by the use of anisotropic H\"older spaces where the functions are one derivative smoother in the radial variable than they are in the angular variable, since in this setting the linearized map turns out to be a Fredholm operator of index $0$. In our case,  finding a functional setting where one can eliminate the loss of derivatives and where the linearized operators are Fredholm operators of index 0 turns out to be rather subtle, as illustrated by the fact that we eventually consider nonlinear maps from the Hilbert space
\[
 \cX^j:= \Big\{
 u\in H^{j}(\Om):   \rho\, \pd_R u \in H^{j} (\Om) \Big\}\,,
\]
into the Banach space
\[
 \cY^j:= \Big\{\rho u_1 +  u_2: u_1\in\cX^{j-2}\,,\; u_2\in H^{j-1}(\Om)\Big\}\,,
\]
both spaces being equipped with their natural norms.
Here $(R,\theta)$ are polar coordinates on the annulus $\Om:=\{1<R<7\}$, the weight $\rho(R,\theta):=\frac16(R-1)(7-R)$ vanishes on~$\pd\Om$, and $j$ is a large integer. The switch from H\"older to Sobolev spaces is not incidental. Roughly speaking, an essential ingredient to establish Fredholmness are sharp regularity estimates for the linearized map, which are standard when this map is uniformly elliptic with smooth coefficients as in~\cite{EFRS,Fall-Minlend-Weth-main} but certainly not in the present setting. In our context, $L^2$-based Sobolev norms are much better suited than H\"older norms to effectively capture the interplay between the singular potential and the high frequencies which underlies the regularity of solutions.

Once these issues have been settled, one can indeed prove the bifurcation result that translates into Theorem~\ref{T.main}. In view of the fairly delicate analytic setting in which we need to work, proving the technical transversality conditions necessary to use an implicit-theorem-based bifurcation theorem such as Crandall--Rabinowitz seems to be highly impractical. Nevertheless,  we can resort to Krasnoselskii's degree-theory-based global bifurcation theorem, where in fact the construction of the nonlinearity~$f_a$ and the radial solution~$\psi_a$ readily enables us to ensure that the crossing number condition of this theorem is satisfied, provided that we additionally restrict our functional setting to functions that are invariant under a certain discrete group of rotations.

\subsection{Related results} \label{Sub.RR}

There is an extensive literature on stationary solutions to the incompressible Euler equations on the plane. In particular, much is known about rigidity conditions, that is, hypotheses which determine the geometry of the possible solutions in certain cases. Among these so-called Liouville theorems, one can mention the results of Hamel and Nadirashvili, which ensure that smooth stationary Euler flows without stagnation points (i.e., with $\nabla\psi\neq0$)  on a bounded planar domain must inherit the symmetry of the domain under suitable assumptions~\cite{Hamel1,Hamel2,Hamel}, particularly in the case of disks and annuli. In the case of compactly supported stationary flows on~$\RR^2$ which we consider in this paper, sufficient conditions for (local) radial symmetry are the absence of stagnation points~\cite{R} or the fact that the vorticity~$-\Delta\psi$ does not change sign~\cite{gomez}. Other results about the rigidity and flexibility of planar stationary fluids in various geometric settings include \cite{ConstantinDrivas,CZ, Shv, Gui}~for smooth solutions to the Euler equations, \cite{Fraenkel, JGS,gomez,gomez21,gomez22}~for vortex patches and sheets, and \cite{Koch, YaoCarrillo}~for Navier--Stokes and for aggregration equations.


\subsection{Organization of the paper}

In Section~\ref{S.MainResult} we present the proof of Theorem \ref{T.main}. The proofs of most of the results stated there will be postponed to further sections. Specifically, Section~\ref{S.L.f} is devoted to the construction of the non-autonomous semilinear equation we will use to obtain stationary Euler flows. In Section \ref{S.FS} we elaborate on the functional setting introduced in Section \ref{S.MainResult}, and in Section~\ref{S.Ga} we analyze the basic properties of the nonlinear map~$\cG_a$ used to prove the main theorem. Sections ~\ref{S.Regularity} and \ref{S.keyest} are devoted to the proof of sharp regularity estimates in our weighted spaces for the linearization of the operator ~$\cG_a$. Finally, the analysis of the spectral properties of the linearized operator is performed in Section~\ref{S.bifurcation}. The paper concludes with an appendix where we show that stationary Euler flows of class~$C^2$ with compact support cannot be obtained via an autonomous semilinear equation for the stream function.

\section{Proof of the main result} \label{S.MainResult}

In this section we present the proof Theorem~\ref{T.main}. To streamline the presentation, we will state several key auxiliary results whose proofs are relegated to later sections. Throughout, we consider that $m \in \RR \setminus \frac12\ZZ$ is a fixed constant, $j \geq 4$ is a fixed integer, and $a \geq 4$ is a real constant. We also assume that $m \geq j+2$.

\subsection*{Step 1: Stationary Euler flows via non-autonomous nonlinearities}

The proof of Theorem \ref{T.main} starts with an elementary but crucial remark connecting the 2D stationary Euler equations to a non-autonomous semilinear elliptic PDE. The results of this section will be proved in Section \ref{S.L.f}.

For this, let $\chi: \RR \to \RR$ be a smooth function such that
$$
\chi = 0 \quad \textup{on } (-\infty,-1] \quad\textup{and}\quad \chi = 1 \quad \textup{on } [1,+\infty)\,.
$$
Suppose that $\psi \in C_c^{2}(\RR^2)$ is a solution to a non-autonomous semilinear equation of the form
\begin{equation}\label{E.psi}
\Delta \psi + f_a(|x|,\, \psi)  = 0\qquad \text{in }\RR^2\,,
\end{equation}
where 
$$
f_a(r,t) := \chi(r -a) f_-(t) + [1-\chi(r-a)] f_+(t)\,,
$$
for some continuous functions $f_\pm$. A straightforward calculation shows that the velocity field given by
\begin{equation} \label{E.velocityPressure}
v:= \nabla^{\perp} \psi 
\end{equation}
defines a compactly supported solution to \eqref{E.Euler} on the plane if and only if
\begin{equation} \label{E.refEuler}
 \chi'(|x|-a) \big[f_-(\psi)-f_+(\psi) \big]\, \nabla^{\perp}  \psi \cdot e_r = 0 \quad \textup{in } \RR^2\,,
\end{equation}
where $e_r$ is the unit radial vector. 
In this case, the pressure can be recovered from~$v$ via the formula
\[
p=-\Delta^{-1}\Div (v\cdot\nabla v)\,.
\]

The take away message here is that, if for some functions $f_\pm$ we are able to find a solution $\psi \in C_c^{k+1}(\RR^2)$ to~\eqref{E.psi} which is not locally radial and satisfies $f_-(\psi) = f_+(\psi)$ in $\bD_{a+1} \setminus \overline{\bD}_{a-1}$, then the vector field~$v$ defined in \eqref{E.velocityPressure} is a nonradial, compactly supported solution to \eqref{E.Euler} of class $C^{k}(\RR^2)$. In what follows we shall show how to implement this strategy to prove Theorem~\ref{T.main}.

Specifically, we will eventually construct nonradial solutions whose support is a nonradial small perturbation of an annulus that we can describe in polar coordinates $(r,\theta)\in\RR^+\times\TT$ as
\begin{equation*} 
\Omega_*:= \{(r,\te) \in (a_-^*,a_+^*) \times \TT\,\}\,.
\end{equation*}
Here $a^*$ is some large constant that we will choose later on and $\TT:=\RR/2\pi\ZZ$. In this formula and in what follows, for any real constant~$a>3$, we use the notation $a_\pm:= a\pm 3$. 

The nonlinearities that we will consider in Equation~\eqref{E.psi} are
\[
f_\pm(t) := |t|^{1-\frac{2}{m}}  g_\pm(t)\,,
\]
where $g_-,g_+$ are smooth functions that we will specify later on. This function depends smoothly on the parameter~$a \geq { 4}$, but we will not reflect this dependence notationally. We shall also use the notation
\begin{equation} \label{E.deff}
f_a(r,t):= \chi(r-a)\, |t|^{1-\frac{2}{m}} g_-(t) + [1-\chi(r-a)]\, |t|^{1-\frac{2}{m}} g_+(t) \,.
\end{equation}
Although the one-variable functions $\chi$ and $g_\pm$ will be smooth, the nonlinearity $f_a$ is only a H\"older continuous function on its second argument (with exponent $1-\tfrac2m$). 

The first key ingredient in the proof of the main theorem is the following lemma. To state it we need to make precise the rate at which functions vanish on the boundary of the annulus
\begin{equation} \label{E.DomainNonPerturbed}
\Omega_a:= \{(r,\te) \in (a_-,a_+) \times \TT\,\}\,.
\end{equation}
To this end, let us introduce the function
\[
\rho_a(r):=\frac16(a_+-r)(r-a_-)\,,
\]
which is a convenient choice of a radial function that is positive in~$\Om_a$ and vanishes linearly on~$\pd\Om_a$. Also, note that $|\rho'_a(a_\iota)|=1$, where here and in what follows $\iota\in\{-,+\}$.

\begin{lemma}\label{L.f}
There exist a constant $\ep>0$ and functions $\Psi_a, g_\pm\in C^\infty(\RR)$ such that:  
	\begin{enumerate}
		\item The function $\psi_a:=\rho_a^m  \, \Psi_a\in C^m(\BOm_a)$ is a radial solution to \eqref{E.psi} in~$\Om_a$.
		\item $\inf_{a_-<r<a_+}\Psi_a(r)>0$ and $g_\iota(a_\iota)\neq0$ for $\iota\in\{-,+\}$.
		\item If $\Om\subset\{a_--\ep < r<a_++\ep\}$ is a planar domain with $C^{2,\alpha}$~boundary and $\psi\in C^{2,\alpha}(\BOm)$ satisfies the semilinear equation~\eqref{E.psi} in~$\Om$ and the bound 
		\[
		\|\psi-\psi_a\|_{L^\infty(\Om_a)}+ \|\psi\|_{L^\infty(\Om\backslash\Om_a)}<\ep\,,
		\]
then the vector field $v:=\nabla^\perp\psi\in C^{1,\alpha}(\Om)$ satisfies the stationary Euler equation~\eqref{E.Euler} in~$\Om$ (with some $C^{2,\alpha}$ pressure).
	\end{enumerate}
\end{lemma}

\medbreak

\subsection*{Step 2: Setting up the problem}  Taking into account Step 1 and Lemma \ref{L.f}, we aim to find non radial solutions to \eqref{E.psi} bifurcating from $\psi_a$.  The technical results of this step are developed in Section \ref{S.FS}, where we deal with the functional setting, and Section \ref{S.Ga}. Given a constant $a \geq { 4}$ and functions $b,B\in C^{j-2}(\TT)$ bounded e.g. as
\[
\|b\|_{L^\infty(\TT)}+\|B\|_{L^\infty(\TT)}<\tfrac{1}{ 10}\,,
\]
we consider bounded domains defined in polar coordinates by
\begin{equation}\label{E.defOm}
\Om_a^{b,B}:=\{(r,\te)\in\RR^+\times\TT: a_-+b(\te)< r< a_++B(\te)\}\,.
\end{equation}

\noindent Then, as explained in Step 1, Theorem \ref{T.main} follows from the following result:

\begin{theorem} \label{T.main2} There exist a number $a^* >4$ and sequences $(a_n)_{n=1}^\infty \subset [4,+\infty)$, $(b_n)_{n=1}^\infty, (B_n)_{n=1}^\infty \subset C^{j-2}(\TT)$, and $(\psi_n)_{n=1}^\infty \subset C^{j-2}(\overline{\Omega}_{a_n}^{\,b_n,B_n})$ such that:
\begin{enumerate}
	\item $a_n\to a^*$ as $n\to\infty$.
	\item $b_n$ and $B_n$ are nonconstant functions which tend to~$0$ in the $C^{j-2}$-norm.
	\item $\psi_n$ are positive solutions to	
\begin{equation}\label{Ecuacion}
	\Delta \psi_n + f_{a_n}(|x|,\, \psi_n)  = 0\quad \text{ in }\Om_{a_n}^{\,b_n,B_n} \,,
\end{equation}
such that $\psi_n(x)  \dist(x, \partial \Omega_{a_n}^{b_n,B_n})^{-m}$ is uniformly bounded in $\overline{\Om}_{a_n}^{\,b_n,B_n}$. Moreover, if we extend $\psi_n$ by $0$ outside the domains $\Omega_{a_n}^{\,b_n,B_n}$, we have that $\psi_n \to \psi_a$ in  $C^{j-2}$.
\end{enumerate} 
\end{theorem}


The rest of this section is devoted to the proof of this result. It is convenient to transform the equations and to work on a fixed domain. For this, we map the fixed annulus  
$$
\Om := \Om_4 = \big\{(R,\te) \in (1,7) \times \TT \big\}
$$
into $\Omega_a^{b,B}$ through the diffeomorfism 
\begin{equation} \label{diffeo} 
\Phi_a^{b,B}: \Om \to \Om_a^{b,B}
\end{equation}
defined in polar coordinates as $\Om\ni(R,\te)\mapsto (r,\te)\in \Om_a^{b,B}$, with
	\begin{equation}\label{E.rchange}
	r(R,\te) := \frac16 \Big[ [a_++B^{\flat}(R,\te)](R-1)+[a_- +b^{\sharp}(R,\te)](7-R) \Big]\,.
	\end{equation}
Here, $b^{\sharp}$ and $B^{\flat}$ are defined as follows: we fix   an even cutoff function $\widetilde{\chi} \in C^{\infty}(\RR)$ such that
\begin{equation} \label{E.chibar}
\widetilde{\chi}(s) = 1 \quad \textup{if } |s| < \tfrac12\,, \quad \textup{ and } \quad \widetilde{\chi}(s) = 0 \quad \textup{ if } |s| > 1\,,
\end{equation}
and then set
\begin{equation} \label{E.TraceLifting}
b^{\sharp}(R,\te) := \sum_{n \in \ZZ} b_n e^{in\te} \widetilde{\chi}\big( (1+|n|)(R-1) \big)\,, \quad   B^{\flat} (R,\te) := \sum_{n \in \ZZ} B_n e^{in\te} \widetilde{\chi}\big( (1+|n|)(R-7) \big) \,,
\end{equation}
in terms of the Fourier coefficients
$$
b_n := \frac{1}{2\pi} \int_0^{2\pi} b(\te) e^{-in\te} d\te\,, \qquad  \qquad B_n:= \frac{1}{2\pi} \int_0^{2\pi} B(\te) e^{-in\te} d\te\,.
$$
Note that the diffeomorphism is as smooth (say in the scale of H\"older norms) as $b,B$ are. In general, throughout the paper we will regard the sharp and flat operators maps from functions on~$\TT$ to functions on~$\Om$. Let us point out that $b^{\sharp},B^{\flat}$ can be understood as trace liftings of the functions $b,B$, respectively. For future reference, let us record here the following bound. Since~$\Om$ does not contain the origin, for convenience we will define the $L^2(\Om)$-norm using the measure $dR d\theta$ on $(1,7)\times\TT$ unless specified otherwise, that is
\[
\|u\|_{L^2(\Om)}^2:=\int_\Om u^2\,dR \, d\theta\,. 
\]

\bigskip

The proof of this result will be given in Section \ref{S.FS}, where we actually provide a more precise statement.

\begin{lemma} \label{L.Liftings}
For $b^{\sharp}$ and $B^{\flat}$ as in \eqref{E.TraceLifting}, it follows that
$$
 \|b^{\sharp}\|_{H^j(\Om)} + \| \rho\, \pd_R b^{\sharp}\|_{H^j(\Om)} \leq C \|b\|_{H^{j-\frac12}(\TT)} \,, \qquad  \|B^{\flat}\|_{H^j(\Om)} + \| \rho\, \pd_R B^{\flat}\|_{H^j(\Om)} \leq C \|B\|_{H^{j-\frac12}(\TT)}\,.
$$
\end{lemma}

For later purposes, we also introduce the shorthand notation 
\[
\Phi_a := \Phi_a^{0,0}\,,
\]
and denote  by $\Phi_{a,1}$ the nontrivial component of this diffeomorphism, which only depends on the radial variable on~$\Om$. One can thus write  $\Phi_a(R,\te) = (\Phi_{a,1}(R),\te)$. 

Let us now use the function $\psi_a$ introduced in Lemma \ref{L.f} (i) to set
\begin{equation} \label{E.PsiPullBack}
\widetilde{\psi}_a = \psi_a \circ \Phi_{a,1}\,.
\end{equation}
With 
\[
\rho(R) := \rho_4(R) = \frac16(R-1)(7-R)\,,
\]
note that $\widetilde{\psi}_a$ is a radial function on~$\Om$ which vanishes like $\rho^m$ on~$\pd\Om$, and that
\begin{equation*} 
\widetilde{\Psi}_a(R) := \frac{\widetilde{\psi}_a(R)}{\rho(R)^m}
\end{equation*}
is a strictly positive function belonging to $C^\infty(\BOm)$.

Using this change of variables, one can rewrite Equation~\eqref{E.psi} in terms of the function 
$$
u := \psi \circ \Phi_a^{b,B} \,,
$$
which is defined on the fixed domain~$\Om$,
as
\begin{equation}\label{E.Pull-Back}
L_a^{b,B}u + f_a(\,|\cdot|,u \circ (\Phi_a^{b,B})^{-1}) \circ \Phi_a^{b,B} = 0 \quad \textup{in } \Om\,.
\end{equation}
Here the differential operator 
$$
L_a^{b,B} u :=  [\De (u\circ (\Phi_a^{b,B})^{-1})]\circ\Phi_a^{b,B} 
$$
is simply the Laplacian $\De$ written in the coordinates $(R,\te)$, and one should note that the regularity of~$u$ is determined by that of~$\psi,b$ and $B$.

A tedious but straightforward computation yields
\begin{align*}
& L_a^{b,B} u =  \frac{1}{(\pd_R r)^2} \pd_R^2 u - \frac{\pd_R^2 r}{(\pd_R r)^3} \pd_R u + \frac{1}{r} \frac{1}{\pd_R r }\pd_R u \\
& \quad  + \frac{1}{r^2} \left[ \pd_{\te}^2 u + \frac{(\pd_{\te} r)^2}{(\pd_R r)^2} \pd_R^2 u - 2 \frac{\pd_\te r}{\pd_R r }\pd_{R}\pd_{\te} u - \frac{\pd_R^2 r (\pd_\te r)^2 - 2 \pd_{R}r  \pd_{\te} r (\pd_{R}\pd_{\te}r)+(\pd_R r)^2 \pd_{\te}^2 r }{(\pd_R r)^3} \pd_R u \right]\,,
\end{align*}
where $r(R,\theta)$ is given by~\eqref{E.rchange}, so that
\begin{align*}
\pd_R r(R,\te) & = \frac{1}{6} \Big( 6 + \Bsharp(R,\te) - \bsharp(R,\te) + \pd_R \Bsharp(R,\te)(R-1)+ \pd_R \bsharp(R,\te)(7-R) \Big)\,, \\
\pd_R^2 r(R,\te) & = \frac16 \Big( 2  \pd_R \Bsharp(R,\te) - 2 \pd_R \bsharp(R,\te) + \pd_R^2 \Bsharp(R,\te)(R-1)+ \pd_R^2 \bsharp(R,\te)(7-R) \Big)\,, \\
\pd_\te r(R,\te) & = \frac16 \Big( \pd_{\te} \Bsharp(R,\te)(R-1) + \pd_{\te} \bsharp(R,\te)(7-R) \Big)\,,\\
\pd_\te^2 r(R,\te) & =  \frac16 \Big( \pd_{\te}^2 \Bsharp(R,\te)(R-1) + \pd_{\te}^2 \bsharp(R,\te)(7-R) \Big)\,,  \\
\pd_{R} \pd_{\te} r(R,\te) & = \frac16 \Big( \pd_{\te}\Bsharp(R,\te) - \pd_{\te} \bsharp(R,\te) + \pd_{R}\pd_{\te} \Bsharp(R,\te)(R-1)+ \pd_{R} \pd_{\te} \bsharp(R,\te)(7-R) \Big)\,.
\end{align*}
In particular, 
\begin{equation} \label{E.La00}
L_a^{0,0} u = \partial_R^2 u + \frac{1}{R+a-4} \partial_R u + \frac{1}{(R+a-4)^2} \partial_\theta^2 u \,.
\end{equation}


Let us introduce the functional setting that we will use. Even though it is inspired by the anisotropic H\"older spaces used in~\cite{Fall-Minlend-Weth-main,EFRS}, the setting we use needs to be substantially different because it crucially depends both on the mapping properties of the nonlinear function we consider and on the sharp regularity properties of its linearized operator, which we summarize in Theorem~\ref{T.regularity}. 

To incorporate these ingredients, let us introduce the scale of ``anisotropic'' Hilbert spaces
\begin{equation} \label{E.Xj}
\cX^j  := \big\{ w \in H^j(\Om) : \rho \, \partial_R  w \in H^j(\Om) \big\}\,,
\end{equation}
endowed with the norm
\[
\|w\|_{\cX^j} := \|w\|_{H^j(\Om)} + \| \rho\, \pd_R w\|_{H^j(\Om)}\,.
\]
In the proof of Theorem~\ref{T.main2}, $\cX^j$ (or rather an open subset of this space) is the domain of the nonlinear map~$\cG_a$ which we will consider in the bifurcation argument.
Throughout, we assume that $j\geq4$ is an integer, so in particular $\cX^j\subset C^2(\BOm)$. Let us also define the closed subset of the elements with zero boundary trace,
$$
\cX^j_{\rm{D}} := \big\{ w \in \cX^j : w|_{\pd \Om} = 0 \big\}\,.
$$
The codomain of the map~$\cG_a$ will be
\begin{equation} \label{E.Yj}
\cY^{j} := \rho \, \cX^{j-2} + H^{j-1}(\Om)\,.
\end{equation}
This is a Banach space, topologized by the canonical norm for the sum of two embedded Banach spaces:
\[
\|w\|_{\cY^{j}} := \inf\Big\{\|{w_1} \|_{\cX^{j-2}} + \|w_2\|_{H^{j-1}(\Om)} :  w = \rho w_1 + w_2  \textup{ with }  {w_1} \in \cX^{j-2},\ w_2 \in H^{j-1}(\Om)\Big\}\,.
\]

To analyze Equation~\eqref{E.Pull-Back}, we shall start by writing the unknown~$u:\Om\to\RR$ as
\begin{equation}\label{E.AnsatzPull-Back}
u := \widetilde{\psi}_a + \rho^{m-1}\, \Theta \,,
\end{equation}
in terms of the radial function~$\widetilde{\psi}_a(R)$ which we introduced in~\eqref{E.PsiPullBack} and another unknown function~$\Theta \in \cX^j_{\rm{D}}$.

Let us introduce the set
$$
\widetilde{\cO}^j := \Big\{ (\Theta,b,B) \in \cX_{\rm{D}}^j \times [H^{j-\frac12}(\TT)]^2  : { \Big\|\frac{\Theta}{\rho}\Big\|_{L^{\infty}(\Om)} } < \tfrac1{10}\, \inf_{\Om} \widetilde{\Psi}_a\  \textup{ and }\ \|b\|_{L^{\infty}(\TT)} + \|B\|_{L^{\infty}(\TT)} < \tfrac1{10}\Big\}\,.
$$
Note that  $\inf_{\Om} \widetilde{\Psi}_a>0$ by Lemma~\ref{L.f}, and that using Lemma \ref{L.Hardy} one can prove that $\widetilde{\cO}^j$ is open. Having at hand this set, we consider the operator arising when we multiply \eqref{E.Pull-Back} by $\rho^{2-m}$. Its specific definition and properties are gathered in the following lemma.  Here and in what follows $f_a'(r,t):=\pd_t f_a(r,t)$ denotes the derivative of the function~$f_a$ with respect to its second argument.

\begin{lemma}\label{L.mapping}
The following assertions hold true:
\begin{enumerate}
\item The function 
\begin{equation}\label{E.preGa}
\widetilde\cG_a(\Theta ,b,B):= \rho^{2-m} \Big[L_a^{b,B} (\widetilde{\psi}_a + \rho^{m-1} \Theta  ) +  f_a (\Phi_{a,1}^{b,B},\ \widetilde{\psi}_a + \rho^{m-1} \Theta  ) \Big]
\end{equation}
maps $\widetilde{\cO}^j\to \cY^j$. \smallbreak
\item The linear operator
$$
\phi\mapsto \rho^{2-m} \Big[ \big[ L_a^{0,0}(\rho^{m-1} \phi  ) + f_a'(\Phi_{a,1},\, \widetilde{\psi}_a)  \big)  \rho^{m-1} \phi   \Big] 
$$
maps $\cX^j \to \cY^j$. 
\end{enumerate}
\end{lemma}

A crucial aspect of this lemma is that the above linear operator maps $\cX^j \to \cY^j$, not only $\cX_{\rm{D}}^j \to \cY^j$ as one would have naively expected. This is because the equation that $\widetilde{\psi}_a$ satisfies leads to a nontrivial cancellation (see Remark \ref{R.cancellation}).

Using the functional setting we just presented, and inspired by \cite[Sect. 3]{EFRS} (see also \cite{Fall-Minlend-Weth-main}), we use functions $w\in\cX^j$ to parametrize $(\Theta,b,B)\in \widetilde{\cO}^j$. More precisely, we define
\begin{equation}\label{E.wto}
\begin{aligned}
 b_w(\te) & := - \big( m \widetilde{\Psi}_a(1) \big)^{-1} \,w(1,\te)\,,\\
 B_w(\te) & := \big( m \widetilde{\Psi}_a(7) \big)^{-1} \,w(7,\te)\,,\\
\Theta_w(R,\te) & := w(R,\te) + \frac{1}{6} \Big(m\rho'(R)\, \widetilde{\Psi}_a(R) + \rho(R) \widetilde{\Psi}_a'(R) \Big) \Big[\Bsharp_w(R,\te) (R-1) + \bsharp_w(R,\te)(7-R) \Big]
\end{aligned}
\end{equation}
for each function~$w$ in the open subset
$$
\cO^{j} := \left\{ w \in \cX^j: { \Big\| \frac{\Theta_w}{\rho} \Big\|_{L^{\infty}(\Om)} } < \frac1{10}\, \inf_{\Om} \widetilde{\Psi}_a \ \textup{ and } \ \|b_w\|_{L^{\infty}(\TT)} + \|B_w\|_{L^{\infty}(\TT)} < \frac1{10} \right\}\,.
$$
Here, $\bsharp_w$ and $\Bsharp_w$ are trace liftings of $b_w$ and $B_w$,  defined as in \eqref{E.TraceLifting} with $b := b_w$ and $B := B_w$, respectively. It is not hard to see that $\cO^{j}$ contains a small ball 
\[
\{w \in \cX^j:\|w\|_{\cX^j}<c_0\}\,,
\]
for some $c_0>0$, which is actually locally uniform for $a\geq4$. 

In view of~\eqref{E.Pull-Back} and~\eqref{E.preGa}, let us then define the map
\begin{equation} \label{E.Ga}
\cG_a : \cO^{j} \to  \cY^{j}\,, \qquad w \mapsto \rho^{2-m} \Big[ L_a^{b_w,B_w} \big( \widetilde{\psi}_a + \rho^{m-1} \Theta_w \big) + f_a \big(\Phi_{a,1}^{b_w,B_w}, \,\widetilde{\psi}_a + \rho^{m-1} \Theta_w \big) \Big]\,.
\end{equation}
If $w \in \cO^{j}$ with $j \geq k+3$ satisfies $\cG_a(w) = 0$ for some $a \geq 4$  and $m\geq j+2$, then $u := \widetilde{\psi}_a + \rho^{m-1} \Theta_w$ satisfies Equation~\eqref{E.Pull-Back} with $b := b_w$ and $B:= B_w$. Thus, our objective is to find small nontrivial zeros of the operator $\cG_{a_n}$ for some sequence $a_n$. 

\begin{remark} \label{R.ansatz}
Let us provide an intrinsic motivation for the definitions in \eqref{E.wto}. The underlying idea is easy to understand, as they are essentially obtained through a formal Taylor expansion of the solution one is trying to construct. Recall that we want to construct our solution by bifurcating from the function $\psi_a$ introduced in Lemma \ref{L.f}. The choice of the functions $b_w$, $B_w$ and $\Theta_w$ in terms of $w$ is precisely motivated by the first order expansion of 
$$
\widehat{\psi}_{a}^{\,sb,\,sB} := \psi_a \circ \Phi_a^{\,sb,\,sB}\,,
$$
at $s = 0$. Here, by abuse of notation, we are writing $\psi_{a}(r,\te) \equiv \psi_{a}(r)$. Observe that
$$
\widehat{\psi}_a^{\, 0,0} = \psi_{a} \circ \Phi_{a,1} = \widetilde{\psi}_a\,,
$$
and that
$$
 \dds\,  \widehat{\psi}_{a}^{\,sb, sB} \Big|_{s=0}   =  \Big[\Bsharp(R,\te) (R-1) + \bsharp(R,\te)(7-R) \Big] \big(\psi_{a}' \circ \Phi_{a,1} \big) = w_a^{b, B}\,,
$$
with 
$$
w_{a}^{b, B}(R,\te) := \frac{1}{6} \Big(m\rho'(R)\, \widetilde{\Psi}_a(R) + \rho(R) \widetilde{\Psi}_a'(R) \Big) \Big[\Bsharp(R,\te) (R-1) + \bsharp(R,\te)(7-R) \Big]\,.
$$
Hence, it is natural to look for a solution to \eqref{E.Pull-Back} of the form
$$
\widetilde{w} = \widetilde{\psi}_a + w_a^{b,B} + w\,,
$$
with $w$, $b$ and $B$ small. Our choice of $b_w$ and $B_w$ in terms of the function $w$ is then done to ensure that, for any $w \in \cX^j$, the function $w_a^{b_w,B_w} + w$ is in the space~$\cX^j_{\rm D}$. 
\end{remark}

\subsection*{Step 3: The linearized operator} 

To find nontrivial zeros of the operator $\cG_a$, we will use a local bifurcation argument. The Fr\'echet derivative of $\cG_a$ will play a major role in the argument. In this step we compute the Fr\'echet derivative of $\cG_a$ and study its properties.

In the first lemma, which will be proved in Section \ref{S.Ga}, we compute the differential of $\cG_a$ at $0$:

\begin{lemma} \label{L.DGa}
The map $\cG_a: \cO^{j} \to  \cY^{j}$ is of class $C^1$. Furthermore, $D\cG_a(0): \cX^j\to \cY^j$ is the linear operator given by
\begin{equation} \label{E.DGa}
D\cG_a(0)w = \rho^{2-m} \fL_a(\rho^{m-1} w)\,, 
\end{equation}
where
\begin{equation} \label{fL}
\fL_a(v):= L_a^{0,0}\,v + f_a'(\Phi_{a,1},\, \widetilde{\psi}_a) \, v\, .
\end{equation}
\end{lemma}

Our next objective is to understand the structure of~$D\cG_a(0)$. A first observation is that the function $f_a'(r,\widetilde{\psi}_a)$ appearing in Lemma~\ref{L.DGa} is radial, diverges as the inverse square of the distance to the boundary of~$\Om$, and is smooth in the interior. More precisely, we have the following result, which we will prove in Section~\ref{S.L.f}:

\begin{lemma}\label{L.f'}
	The function $R \mapsto \rho^2\,f_a'(\Phi_{a,1}, \widetilde{\psi}_a)$ is of class $C^\infty([1,7])$. Furthermore, 
	\[
	\rho^2 f_a'(\Phi_{a,1},\widetilde{\psi}_a)= -(m-1)(m-2)+ O(\rho)\,.
	\]
\end{lemma}

In view of Lemma~\ref{L.DGa} and Lemma~\ref{L.f'}, let us decompose the Fr\'echet derivative of~$\cG_a$ at~0 as
\[
D\cG_a(0)=\cT_a+\cK_a\,,
\]
where
\begin{align}
	\cT_aw&:= \rho\left[\rho^{1-m}  L_a^{0,0}(\rho^{m-1} w) -\frac{(m-1)(m-2)}{\rho^2}w\right]\,,\label{E.cTa}\\[1mm]
	\cK_aw&:= \rho\left[ f_a'(\Phi_{a,1},\widetilde{\psi}_a)+\frac{(m-1)(m-2)}{\rho^2}\right] w\,.\label{E.cKa}
\end{align}
One should think of $\cT_a$ as a positive-definite essentially self-adjoint operator of second order with a critically singular potential (in the sense that $\rho^{-2}$ scales like $\pd_R^2$). This intuition is made precise in the following theorem. In the proof of this result, which we present in Sections \ref{S.Regularity} and \ref{S.keyest}, the key step is to develop from scratch a sharp regularity theory for the operator~$\cT_a$.

\begin{theorem}\label{T.regularity}
The operator $\cT_a:\cX^j\to\cY^j$ is an isomorphism, and $\cK_a:\cX^j\to\cY^j$ is compact.
\end{theorem}

\subsection*{Step 4: The bifurcation argument} The proof of Theorem \ref{T.main2} will follow from Kranoselskii's local bifurcation theorem. The details of this step are gathered in Section \ref{S.bifurcation}. For the benefit of the reader, we start by providing a precise statement of Kranoselskii's theorem\footnote{This is in fact a slightly refined version of the Krasnoselskii bifurcation theorem, as  one typically assumes the existence of an isolated point where the differential $D_wG(0, a)$ is degenerate. However, the usual proof of the theorem, as presented in \cite[Theorem~II.3.2]{Kielhofer}, works equally well in this case; see e.g.~\cite[Remark 6.3]{RRS20}.}.

\begin{theorem}[Kranoselskii]\label{T.Kr}
	Let $U\subset Y$  be an open subset of a Banach space~$Y$, which we assume to contain~$0$, and let  $A\subset\mathbb{R}$ be an open interval. Let $\hcG:U\times A\rightarrow Y$ be a $C^{1}$ map such that:
	\begin{itemize}
		\item[(i)] $\hcG(0,a)=0$ for all $a\in A$.
		\item[(ii)] The map $w\mapsto \hcG(w,a)-w$ is compact for each $a\in A$.
		\item[(iii)] Denoting by $\ind(a) $ the index of $D_{w}\hcG(0,a)$ (that is, the sum of the algebraic multiplicities of all negative eigenvalues of $D_{w}\hcG(0,a)$), we assume that there exist values of the parameter $a_1<a_2$ in~$A$ such that:
		\begin{itemize}
			\item[(a)] The differential $D_{w}\hcG(0,a_i)$ is non degenerate for $i=1,2$.
			\item[(b)] $\ind(a_1)$ and $\ind(a_2)$ have different parity.
		\end{itemize}
\end{itemize}
Then there exists a bifurcation point $a_{*} \in (a_1,a_2)$, in the sense that $(0,a_{*})\in Y\times A$ is an accumulation point of the set of nontrivial solutions $\{(w,a)\in(Y\backslash\{0\})\times A: \hcG(a,w)=0\}$.
\end{theorem}


Our aim now is to rewrite the equation $\cG_a(w)=0$ in an equivalent way, which is well adapted to the use of Krasnoselskii's bifurcation theorem. On this purpose, we first note that the operator
$$
\fL_a: H_0^1(\Omega) \to H^{-1}(\Omega)\,, \quad v \mapsto L_a^{0,0} v + f_a'(\Phi_{a,1}, \widetilde{\psi}_a)\, v\,,
$$
defines a positive symmetric bilinear form $ \cB:H^1_0(\Omega) \times H^1_0(\Omega)  \to \RR$ as
\begin{align*}
\cB(v,w)& :=  -\int_{\Omega} \big( L_a^{0,0} v + f_a'(\Phi_{a,1}, \widetilde{\psi}_a) v \big) w\, (R+a-4)\, d R \, d \theta \\
& =  \int_{\Om} \left[ \pd_R v\, \pd_R w+ \frac{\pd_\theta v\,\pd_\theta w}{(\erre)^2}+\frac\m{\rho(R)^2} v w  - P_a(R) v w\right](\erre) \, dR \, d\theta\,.
\end{align*}
Here $P_a$ is the radial function
$$
P_a(R) := f_a'(\Phi_{a,1}, \widetilde{\psi}_a) + \frac{(m-1)(m-2) }{\rho(R)^2}\,.
$$
Note that this bilinear form and the operator $\fL_a:H_0^1(\Omega) \to H^{-1}(\Omega)$  are well defined thanks to the Hardy inequality. Moreover, by the properties of $f_a'$ proved in Lemma~\ref{L.f'},  $\rho(R)\, P_a(R)$ is a smooth function on~$\BOm$.


It is standard that there exists a sequence of eigenvalues of~$-\fL_a$ with finite multiplicity, which we denote by $\lambda_k \equiv \lambda_k(a)$, such that
$$ 
\lambda_1 < \lambda_2 \leq \lambda_3 \leq \cdots 
$$
Of course, $\la_k$ tends to infinity as $k \to \infty$. There exists a orthonormal basis of $L^2(\Om)$ consisting of eigenfunctions of~$-\fL_a$, which we denote by $\phi_k\in H_0^1(\Omega)$. A priori, the eigenvalue equation must be understood in a weak sense, that is,
$$ 
\cB(\phi_k, v) = \lambda_k \int_{\Omega} \phi_k\,  v\, dx \qquad \forall\ v \in H_0^1(\Omega)\,.
$$ 
By the regularity properties of the linearized operator (Proposition~\ref{P.regDGa}), the eigenfunctions $\phi_k$ actually belong to $\rho^{m-1}\, \cX^j$. Writing $\phi_k = \rho^{m-1} w_k$ with $w_k\in\cX^j$, one then has
\begin{equation} \label{equivalence} 
-\fL_a\, \phi_k = \lambda_k \phi_k \ \ \Longleftrightarrow \ \ -D\cG_a(0)\, w_k = \lambda_k \rho \, w_k\,. 
\end{equation}

In order to apply Kranoselskii's bifurcation theorem, we need to study the sign of the eigenvalues $\lambda_k$ and to find some degeneracies. On that purpose, we need to take an integer $\ell\geq3$ and to restrict our attention to the space of $\ZZ_\ell\,$-$\,$symmetric functions, i.e., functions that are invariant under the action of the isometry group of an $\ell$-sided regular polygon. To incorporate this restriction in our functional setting, we define 
\begin{equation} \label{symmetry}
\begin{aligned}
	H^j_\ell(\Om) & := \Big\{ w \in H^j(\Om) : w(R,\te) = w(R,-\te)\,, \ u(R,\te) = u\Big(R,\frac{2\pi}{\ell}\Big) \Big\}\,,  \\
	H^j_{0,\ell}(\Om) & :=  \big\{ w \in H^j_\ell(\Om) : w|_{\pd \Om} = 0 \big\}\,,\\
	\cX^j_\ell & := \Big\{ w \in \cX^j : w(R,\te) = w(R,-\te)\,, \ u(R,\te) = u\Big(R,\frac{2\pi}{\ell}\Big) \Big\}\,, \\
	\cX^j_{{\rm{D}},\ell} & :=  \big\{ w \in \cX^j_\ell : w|_{\pd \Om} = 0 \big\}\,, \\
	\cY^j_\ell & :=  \Big\{ w \in \cY^j : w(R,\te) = w(R,-\te)\,, \ u(R,\te) = u\Big(R,\frac{2\pi}{\ell}\Big) \Big\}\,.
\end{aligned}
\end{equation}
Of course, the eigenvalues depend on the symmetry chosen and on the parameter $a$, that is, $\lambda_k\equiv\lambda_{k}^{\ell}(a)$. For the sake of clarity, we shall omit this dependence notationally if there is no ambiguity. 

In the following proposition we establish the degeneracy property of eigenvalues that is required to apply the Kranoselskii's bifurcation theorem. More precisely, we have the following:

\begin{proposition}\label{eigen-crossing} Given $\ep>0$, there exist $a_1>a_0 \geq 4$, $\ell \in \NN$, $\ell \geq 3$, such that:
	\begin{enumerate}
		\item For any $ a \in [a_0, a_1]$, $\lambda_1^{\ell}(a)< \lambda_2^{\ell}(a) <- \ep < \lambda_3^{\ell}(a)$.
		
		\item $\lambda_3^{\ell}(a_0)>0$ and $\lambda_3^{\ell}(a_1)<0$.
		
		\item If $a \in [a_0, a_1]$ is such that $\lambda_3^{\ell}(a) \leq 0$, then $\lambda_3^{\ell}(a)$ is simple and the corresponding eigenfunction is $\phi_3(R,\te) = \phi(R) \cos (\ell \theta)$, where $\phi \geq C \rho^{m-1}$ for some $C > 0$. 
		\item For any $ a \in [a_0, a_1]$, $\lambda_4^{\ell}(a)>0$.

	\end{enumerate}
\end{proposition}

Having this result at hand , we can now reformulate the problem of solving the equation $\cG_a(w) = 0$ and conclude the proof of Theorem \ref{T.main2} using Krasnoselskii's bifurcation theorem (Theorem \ref{T.Kr}).

First of all, recall that 
\[
D\cG_a(0)=\cT_a+\cK_a\,,
\]
where $\cT_a: \cX^j \to \cY^j$ is an isomorphism and $\cK_a: \cX^j \to \cY^j$ is compact. In particular, $D\cG_a(0)$ is a Fredholm operator of index $0$. Since this class of operators is open, we can choose $\ep>0$ such that, for $a_1 > a_0 \geq 4$ as in Proposition \ref{eigen-crossing}, and
$$
 \cS:  \cO^{j} \times [a_0, a_1] \to \cY^j\,, \quad  (w,a) \mapsto \ep \rho \, w - \cG_a(w) \,,
$$
$\cS(\cdot,a)$ is also a Fredholm operator of index $0$ for all $a \in [a_0, a_1]$. We will  use the notation $\cS_a= \cS(\cdot,a)$. Next, observe that
$$
D \cS_a(0)w = \ep \rho \, w  - D \cG_a(0) w \,. 
$$
By the eigenvalue properties of Proposition \ref{eigen-crossing} (i) and the equivalence \eqref{equivalence}, $D\cS_a(0)$ is injective for all $a \in [a_0,a_1]$, so we conclude that it is an isomorphism. By the Inverse Function Theorem, we can take a smaller neighborhood of $0$ in $ \cX_{\ell}^j$ (still denoted by $\cO^j$) and a neighborhood $\cV^j$ of $0$ in  $\cY_{\ell}^j$ so that $\cS_a : \cO^j \to \cV^j$ is invertible for all $a \in [a_0, a_1]$. Here and in the rest of the section we fix $\ell \in \NN$, $\ell \geq 3$, as in Proposition \ref{eigen-crossing}. 

Now, we define
$$ 
\widehat{\cG}: \cV^j \times [a_0,a_1] \to  \cV^j, \quad (v,a) \mapsto   v - \ep \, \rho \, \cS_a^{-1}(v)\,,
$$
and stress that, since the image of $\cS_a^{-1}$ is included in $\cX_{\ell}^j$, which is compactly embedded in $\cY_{\ell}^j$ by Lemma \ref{L.CompactEmbedding}, the operator $\widehat{\cG}$ has the form of identity minus a compact operator. We also use the notation $\widehat{\cG}_a = \widehat{\cG}(\cdot,a)$. Clearly, $\widehat{\cG}(v,a)=0$ if and only if $\cG_a(w)=0$, where $ \cS_a(w)=v$.  Moreover, it follows that
$$
\begin{aligned}
D \widehat{\cG}_a(0)v = \mu v &\ \Longleftrightarrow\ (1- \mu) D \cS_a(0)w= \ep \rho \, w\,, \ \ \textup{with }\ D \cS_a(0)w =v\,, \\ &\ 
\Longleftrightarrow\ -D \cG_a(0)w= \frac{\ep \mu}{1-\mu} \rho \, w\\
& \ \Longleftrightarrow\ -\fL_a(\phi)= \frac{\ep \mu}{1-\mu}  \, \phi\,, \ \mbox{ with }\ \phi= \rho^{m-1}w\,.
\end{aligned}
$$
Therefore, the eigenvalues of $D\widehat{\cG}_a (0)$ are given by $\mu_k$, where
$$
\frac{\mu_k}{1-\mu_k} = \ep^{-1}\lambda_k\,.
$$
The same correspondence above shows that the eigefunctions $v$ do not belong to the range of $D \widehat{\cG}_a(0) - \mu \, {\rm Id}$. By Proposition \ref{eigen-crossing}  (i), we know that $\lambda_1\, \ep^{-1}<\lambda_2\, \ep^{-1}< -1$. As a consequence, $\mu_1>1$, $\mu_2>1$. On the other hand, $\lambda_k>0$, for all $k\geq 4$, which implies that $\mu_k \in (0,1)$, for all $k \geq 4$.  Finally, Proposition \ref{eigen-crossing} (i) also shows that $\ep^{-1} \lambda_3 > -1$, which implies that $\mu_3$ has the same sign than $\lambda_3$. Moreover, if $\lambda_3=0$ then $\mu_3=0$ and its algebraic multiplicity is equal to one.

Then, taking into account Proposition \ref{eigen-crossing} (ii) and (iii), we can apply Theorem \ref{T.Kr} and conclude that there exists a sequence $(v_n, a_n)_{n=1}^{\infty} \subset  \cV^j \times [a_0, a_1]$ such that 
$$ 
\widehat{\cG}(v_n,a_n)=0\,, \quad  a_n \to a^*\,, \quad \textup{and} \quad v_n \to 0  \mbox{ in }  \cY_{\ell}^j\,.
$$
As a consequence, we get a sequence $(w_n)_{n=1}^{\infty} \subset \cX_\ell^j$ such that
\begin{equation} \label{E.sequencewn}
\cG_{a_n}(w_n) =0\,,   \quad \textup{and} \quad \ w_n \to 0 \mbox{ in }  \cX_{\ell}^j \,, \ \mbox{ with }\  \cS_{a_n}(w_n)=v_n\,.
\end{equation}
In particular, it follows that $\lambda_3(a^*)=0$. 

\begin{proposition} \label{P.nonradial}
For all $n \in \NN$ sufficiently large, the functions $b_{w_n}$ and $B_{w_n}$ given in \eqref{E.wto} with $w:= w_n$ are nonconstant. 
\end{proposition}

In short, we have obtained a solution to the problem
\begin{equation} \label{final}
L_a^{b_{n},B_{n}} \big( \widetilde{\psi}_a + \rho^{m-1} \Theta_{w_n} \big) + f_a \big(\Phi_{a,1}^{b_{n},B_{n}}, \,\widetilde{\psi}_a + \rho^{m-1} \Theta_{w_n} \big)=0 \quad \textup{in } \Om\,,
\end{equation}
where we are using the shortened notation $b_n := b_{w_n}$ and $B_n:= B_{w_n}$ for $b_{w_n}$ and $B_{w_n}$ as in \eqref{E.wto}.
We now compose with $\Phi_{a_n}^{b_n, B_n}$ to obtain positive solutions to 
$$
\Delta \psi_n + f_{a_n} (|x|, \psi_n(x))=0 \quad \mbox{in }\Omega_{a_n}^{b_n,B_n}\,.
$$
We know that this (positive) solutions are of the form
$$ \psi_n(x) = d_n(x)^m \, \Psi_n(x)\,,$$
where $d_n(x):= \dist(x, \partial \Omega_{a_n}^{b_n,B_n})$ and $\Psi_n \in H^j(\Omega_{a_n}^{b_n,B_n})$. This shows the validity of Theorem \ref{T.main2}, and of Theorem \ref{T.main} too. Indeed, if we set 
$$
\psi_n^*(x) = \left \{
\begin{aligned}
\ & \psi_n(x) \quad && \textup{if } x \in \Om_{a_n}^{b_n,B_n}\,,\\
& 0 && \textup{in } x \not \in  \Om_{a_n}^{b_n,B_n}\,,
\end{aligned}
\right.
$$
we have that, for all $n \in \NN$ sufficiently large and all $k \leq j-3$, 
$$
v_n := \nabla^{\perp} \psi_n^*\,,
$$
is a compactly supported stationary Euler flow of class $C^k(\RR^2)$ which is not locally radial. 
 

\section{Stationary Euler flows via \\ elliptic equations with non-autonomous nonlinearities}
\label{S.L.f}

In this section we analyze the radial solutions to the elliptic equation with a non-autonomous nonlinearity which we will consider throughout the paper.  Lemmas \ref{L.f} and \ref{L.f'} immediately follow from the following result, where we actually find a rather explicit expression for $f_{\pm}$.

\begin{lemma}\label{L.f2}
	There exist constants $a_0>4$, $\ep>0$, such that, for any $a>a_0$, there are positive radial functions $\Psi_a \in C^{\infty}(\overline{\Omega}_a)$, $g_\pm\in C^\infty(\RR^+)$ such that:  
	\begin{enumerate}
		\item The function $\psi_a:=\rho_a^m\, \Psi_a$ is a solution to \eqref{E.psi} in $\Om_a$ with 
		\[
		f_\pm(t) := -m (m-1)|t|^{1-\frac{2}{m}}  g_\pm(t).
		\]
		\item $\inf_{a_-<r<a_+}\Psi_a(r)>0$ and $g_\iota(a_\iota) =1$ for $\iota\in\{-,+\}$.
		\item If $\Om\subset\{a_--\ep < r<a_++\ep\}$ is a planar domain with $C^{2,\alpha}$~boundary and $\psi\in C^{2,\alpha}(\BOm)$ satisfies the semilinear equation~\eqref{E.psi} in~$\Om$ and the bound 
		\[
		\|\psi-\psi_a\|_{L^\infty(\Om_a)}+ \|\psi\|_{L^\infty(\Om\backslash\Om_a)}<\ep\,,
		\]
		then the vector field $v:=\nabla^\perp\psi\in C^{1,\alpha}(\Om)$ satisfies the stationary Euler equation~\eqref{E.Euler} in~$\Om$ (with some $C^{2,\alpha}$ pressure).
		\item $\Psi_a(\cdot +a-4) \to \overline{\Psi}$ in the $C^{\ell}$ sense, for any fixed $\ell \in \NN$, and $f_\pm \to \overline{f}$ uniformly as $a \to + \infty$. Here,
$$
\overline{\Psi}:=\rho^m\, \overline{\psi}, \quad \textup{ and }  \quad  \overline{\psi}''(r) + \overline{f}(\overline{\psi}(r))=0\,,\ \ r \in (1,7)\,.
$$
Moreover,
$$
f'_{\pm}(\psi_a(\cdot+a-4)) \mp \frac{m-1}{a_{\pm}} \rho^{-1} - \overline{f}'(\overline{\psi}) \to 0\,, \ \ \mbox{ as } a \to  + \infty\,,
$$ 
uniformly.
	\end{enumerate}
\end{lemma}

\begin{proof}

Let $\tau_a \in C^{\infty}(\BOm_a)$ be the unique solution to
\begin{equation}\label{E.Torsion}
\De \tau_a +1 =0 \quad \textup{in } \Om_a\,, \qquad \tau_a = 0 \quad \textup{on } \pd \Om_a\,.
\end{equation}
The solution is positive and radially symmetric, so we will simply write $\tau_a(r)$. Note that, as function of $r$, it satisfies the ODE
\begin{equation}\label{E.taua}
\tau_a''+\frac{\tau_a'}{r} + 1=0\,,\qquad \tau_a(a_-)=\tau_a(a_+)=0\,,
\end{equation}
and has an explicit expression, namely
\begin{equation} \label{explicit} \tau_a(r)= \frac{(a_+^2-a_-^2) \log(r) - (r^2-a_-^2) \log(a_+) - (a_+^2-r^2) \log(a_-)}{4 \log (a_+) - 4 \log(a_-)}\,.
\end{equation}
By the maximum principle,  $\tau_a:[a_-,a_+]\to[0,\infty)$ is nonnegative and only vanishes at the endpoints. Moreover, as a consequence of the classical moving plane method (see \cite{gnn}), its unique maximum is attained at a unique point $m_a\in(a_-,a)$. Also, it is easy to show that, as $a\to\infty$,  we have that $\tau_a(\cdot + a-4 ) \to \overline{\tau}$ in $C^\ell([1,7])$ for any fixed $\ell \geq 0$. Here, $\overline{\tau}$ is a solution to
$$
\overline{\tau}'' + 1 = 0 \quad \textup{in } (1,7)\,, \qquad \overline{\tau}(1)=\overline{\tau}(7) = 0\,.
$$
Clearly, $\overline{\tau}(r)= \tfrac12(9-(r-4)^2)$. 

We now truncate $\tau_a$ appropriately. Let us define the distance function $\tilde{\rho}_a: \Omega_a \to \RR$, which is also radially symmetric, as
$
\tilde{\rho}_a(r) := \min \{ r-a_-, a_+-r\}\,.
$
Then, we take a cut-off $\widetilde{\chi} : \RR \to \RR$ as in \eqref{E.chibar}, and define
\begin{equation} \label{psi_a} 
\psi_a(r):= \widetilde{\chi} (\tilde{\rho}_a(r)) \tilde{\rho}_a(r)^m  + (1- \widetilde{\chi}(\tilde{\rho}_a(r))) \tau_a(r) \,. 
\end{equation}
Note that $\psi_a \in C^{\infty}(\overline{\Om}_a)$, and that $\psi_a(r)=\tilde{\rho}_a(r)^m$, if $\tilde{\rho}_a(r) <1/2$, but $\psi_a(r)= \tau_a(r)$, if $\tilde{\rho}_a(r) >1$.

\noindent \textbf{Claim:} \textit{There exists $a_0 > 4$ such that, for all $a > a_0$:}
\begin{equation} \label{this}
\begin{aligned}
\psi_a(r)&  > 3\ \Longrightarrow\ \psi_a(r) = \tau_a(r)\,, \\
\psi_a(r)& < 2^{-m} \ \Longrightarrow\ \psi_a(r) = \tilde{\rho}_a(r)^{m}\,.
\end{aligned}
\end{equation}

\noindent \textit{Proof of the claim:} To prove the first assertion, we just have to prove that $\tilde{\rho}_a(r) \geq 1$ whenever $\psi_a(r) > 3$. We argue by contradiction and assume that $\psi_a(r) > 3$ and  that $\tilde{\rho}_a(r) < 1$.  Then, observe that $\overline\tau(r) \geq 5/2$ if and only if $r \in [2,6]$, and so that, for $a$ sufficiently large, $\tilde{\rho}_a(r)<1$ implies $\tau_a(r) < 3$. Hence, taking into account the definition of $\psi_a$, namely \eqref{psi_a}, we get that
\begin{equation} \label{def psia}
\psi_a(r)= \widetilde{\chi} (\tilde{\rho}_a(r)) \tilde{\rho}_a(r)^m  + (1- \widetilde{\chi}(\tilde{\rho}_a(r))) \tau_a(r) \leq \widetilde{\chi} (\tilde{\rho}_a(r)) + 3(1- \widetilde{\chi}(\tilde{\rho}_a(r)))  \leq 3\,,
\end{equation}
reaching a contradiction.

To prove the second assertion, we show that $\tilde{\rho}_a(r)\leq 1/2$ whenever $\psi_a(r) < 2^{-m}$. We argue again by contradiction and assume that $\psi_a(r) < 2^{-m}$ and that $\tilde{\rho}_a(r) > 1/2$. Then, observe that $\overline\tau(r) \leq 11/8$ if and only if $r \in (3/2,13/2)$, and so that, for $a$ sufficiently large, $\tau_a(r) > 1$ whenever $\tilde{\rho}_a(r)>1/2$. Hence, taking into account the definition of $\psi_a$, i.e. \eqref{psi_a}, we get that 
$$
\psi_a(r)= \widetilde{\chi} (\tilde{\rho}_a(r)) \tilde{\rho}_a(r)^m  + (1- \widetilde{\chi}(\tilde{\rho}_a(r))) \tau_a(r) \geq \widetilde{\chi} (\tilde{\rho}_a(r)) 2^{-m}  + (1- \widetilde{\chi}(\tilde{\rho}_a(r)))  \geq  2^{-m}\,,
$$
reaching again a contradiction. The claim is proved.

Next, we study the monotonicity properties of $\psi_a$. Clearly, see \eqref{psi_a}, $\psi_a$ is increasing in $(a_-, \,a_-+1/2)$ and in $(a_-+1,\, m_a)$. Moreover, we can show that
\begin{equation} \label{E.monotonicityPsia}
\psi_a'(r) > \ep >0\,, \quad  \mbox{ for all } r \in (a_-+ 1/2,\, a_-+1)\,,
\end{equation}
with $\ep > 0$ independent of $a \geq a_0$. Indeed, for all $r (a_-+ 1/2,\, a_-+1)$, it follows that
$$ 
\psi_a'(r)=  \widetilde{\chi}' (r-a_-) (r-a_-)^m  + \widetilde{\chi} (r-a_-) m (r-a_-)^{m-1} + (1- \widetilde{\chi}(r-a_-)) \tau_a'(r) - \widetilde{\chi}'(r-a_-) \tau_a(r)\,.
$$
Moreover, observe that, for $a$ sufficiently large,
$$
\widetilde{\chi} (r-a_-) m (r-a_-)^{m-1} + (1- \widetilde{\chi}(r-a_-)) \tau_a'(r) > \min \{\tau_a'(r): r \in (a_-+ 1/2, a_-+1)\}>0\,,
$$
Thus, since $\tilde{\chi}'(r-a_-)$ is negative for $r \in (a_-+1/2,\, a_-+1)$, to prove \eqref{E.monotonicityPsia}, we only need to show that
$$ 
 \tau_a(r) -(r-a_-)^m \geq 0\,, \quad \mbox{ for all } r \in (a_-+ 1/2, a_-+1)\,.
$$
However, this inequality is immediate, at least, for large values of $a$.

Similarly, it is immediate to see that $\psi_a$ is decreasing in $(m_a,a_+-1)$ and in $(a_+-1/2,\, a_+)$. Moreover, arguing as in the proof of \eqref{E.monotonicityPsia}, we get that
$$
\psi_a'(r) < -\ep <0\,, \quad \mbox{ for all } r \in (a_+ - 1, a_+ -1/2)\,,
$$
with $\ep > 0$ independent of $a \geq a_0$.

These monotocicity properties allow us to determine functions $f_-(t)$ and $f_+ (t)$ such that 
$$ 
\psi_a''(r) + \frac{\psi_a'(r)}{r} + f_{\pm}(\psi_a(r))=0\,, \quad \mbox{ if}\ \pm (r-m_a)>0\,.
$$
Moreover, by the first implication in \eqref{this}, 
$$ 
f_{\pm}(t)=1\,, \quad \mbox{ if } t \geq 3\,.
$$

Let us now compute the explicit expression of $f_{\pm}(t)$ for $t\in [0, 2^{-m}]$. We consider $f_-(t)$, the other case being completely analogous. For $r \in (a_-, a_-+1/2)$, we have that $ \psi_a(r)= (r-a_-)^m$, and so that
$$
\psi_a''(r) + \frac{\psi_a'(r)}{r} = m(m-1) (r-a_-)^{m-2} + \frac{m}{r} (r-a_-)^{m-1}\,.
 $$
Hence, setting $r=: t^{1/m}+a_-$, we obtain that
\begin{equation} \label{fmenos} f_{-}(t)= -m(m-1)\, t^{\frac{m-2}{m}} - \frac{m}{\, t^{1/m}+a_-} t^{\frac{m-1}{m}}\,, \quad \textup{ for } t \in [0, 2^{-m}]\,.
\end{equation}
Similarly, we have that
\begin{equation} \label{fmas}f_{+}(t)= -m(m-1) t^{\frac{m-2}{m}} - \frac{m}{\, t^{1/m}- a_+} t^{\frac{m-1}{m}}\,,  \quad \textup{ for }t \in [0, 2^{-m}]\,.
\end{equation}

\begin{figure}[h]
	\centering 
	\begin{minipage}[c]{120mm}
		\centering
		\resizebox{120mm}{80mm}{\includegraphics{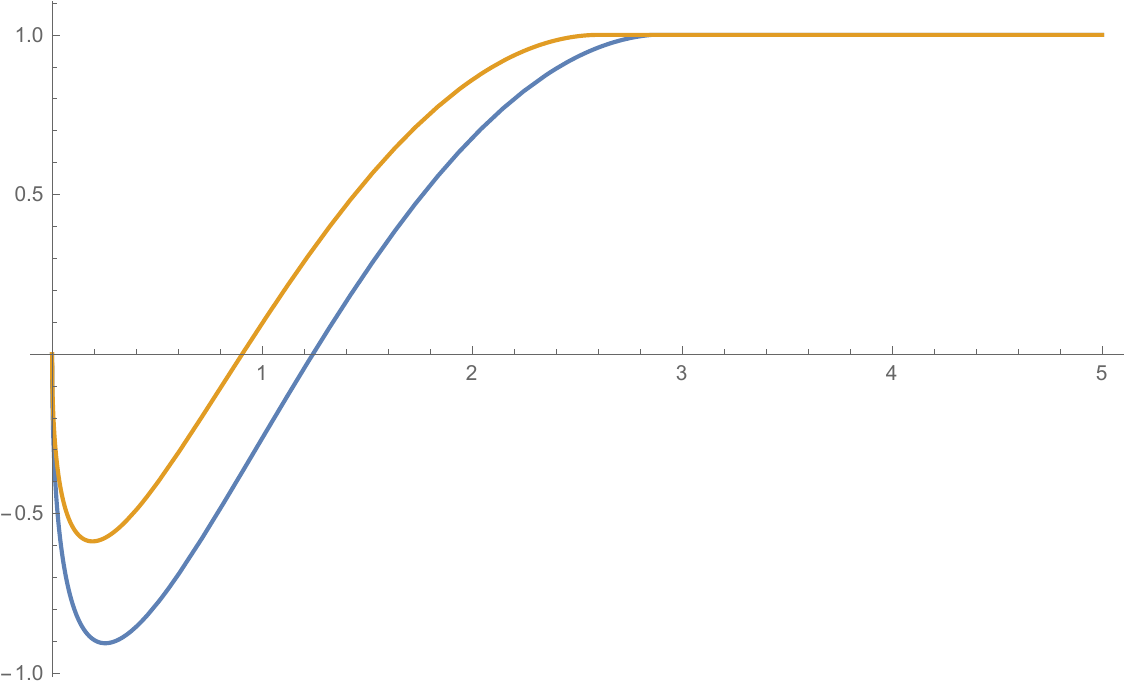}}
	\end{minipage}
	\caption{The graphs of the functions $f_-$ (blue) and $f_+$ (orange).}
\end{figure}

These explicit expressions of $f_{\pm}$ immediately implies that
\begin{align} \label{fmenosprima} f_{-}'(t) & = -(m-1)(m-2)\,t^{-2/m} - \frac{m-1}{t^{1/m} + a_{-}} t^{-1/m} + \frac{1}{(t^{1/m} + a_{-})^2}\,, \quad \textup{ for } t \in [0, 2^{-m}]\,, \\ \label{fmasprima}
 f_{+}'(t) & = -(m-1)(m-2)t^{-2/m} - \frac{m-1}{t^{1/m} - a_{+}} t^{-1/m} + \frac{1}{(t^{1/m} - a_{+})^2}\,,\quad \textup{ for } t \in [0, 2^{-m}]\,. 
\end{align}
Finally, we point out for later use that there exists $C > 0$ independent of $a \geq a_0$ such that
$$
|f'(t)| \leq C\,, \quad  \mbox{ for } t \in (2^{-m}, 3)\,.
$$

Observe now that, since $\overline{\tau}(r) \geq 4$ if $r \in (3,5)$, for $a$ sufficiently large, $\psi_a(r) > 3$ whenever $r \in [a-1,a+1]$. Hence, using again the first implication in \eqref{this}, we get that
\begin{equation} \label{semilineal2}
\begin{aligned}
 \Delta \psi_a(x) +  f_+(\psi_a(x))&  =0\,, \quad \mbox{ for } |x| > a-1\,, \\ \Delta \psi_a(x) +  f_-(\psi_a(x)) & =0\,, \quad \mbox{ for } |x| < a+1\,,
\end{aligned}
\end{equation} 
and so that $\psi_a$ is a solution of the problem \eqref{E.psi}. 

\begin{figure}[h]
	\centering 
	\begin{minipage}[c]{120mm}
		\centering
		\resizebox{120mm}{80mm}{\includegraphics{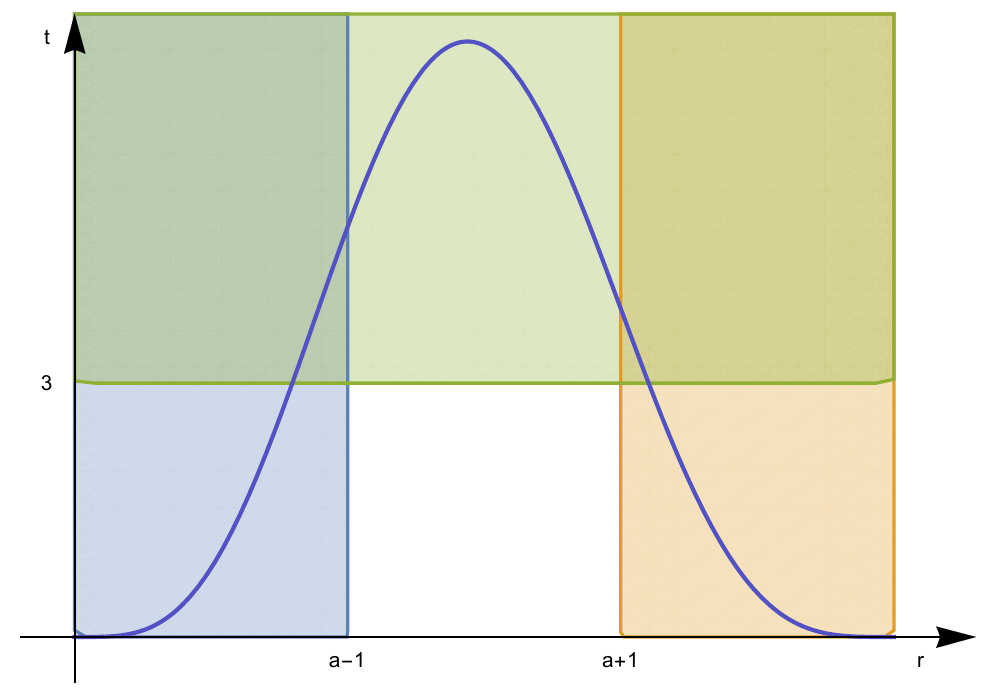}}
	\end{minipage}
	\caption{In this picture we illustrate the definition of $f(r,t)$: \\	
$\circ$   On the left blue rectangle, $f(r,t)= f_-(t)$. \\
$\circ$		On the right orange rectangle, $f(r,t)= f_+(t)$. \\
$\circ$	On the upper green rectangle, $f(r,t)=f_+(t)=f_-(t)=1$. \\
$\circ$ $\partial_r f(r,t) \neq 0$ only in the lower white rectangle, which does not intersect the graph of $\psi_a$, here represented in blue.}
\end{figure}

At this point we only need to verify (iii) and (iv). We start by proving (iii). Assume that $\Omega$ and $\psi$ are as in the statement of the lemma. We are going to show that \eqref{semilineal2} is also satisfied by $\psi$. By the definition of $\widetilde{\chi}$, we only need to check these identities in the annulus $(r,\te) \in [a-1, a+1] \times \TT$. However, since $\| \psi - \psi_a \|_{L^{\infty}} < \ep$, we immediately infer  that 
$$
\psi(r,\te) >3 \quad \mbox{ for } r \in (r,\te) \in [a-1, a+1] \times \TT\,,
$$
and so we can conclude by recalling that $f_{\pm}(t)=1 $ for all $t \geq 3$. 

We now deal with (iv), namely we interested in the asymptotics as $a$ tends to infinity. Clearly, we have that $\psi_a(\cdot +a-4) \to \overline{\psi}$ in the $C^{\ell}$ sense for all $\ell \in \NN$. Here, $\overline{\psi}$ is defined as
\begin{equation} \label{psibar} 
\overline{\psi}(r)= \widetilde{\chi} (\tilde{\rho}(r)) \tilde{\rho}(r)^m  + (1- \widetilde{\chi}(\tilde{\rho}((r))) \overline{\tau}(r)\,, \quad  \textup{ with } \tilde{\rho}(r):= \min\{r-1,7-r\}\,,
\end{equation}
and it is a solution to 
$$ 
\overline{\psi}''(r) + \overline{f}(\overline{\psi}(r))=0\,, \ \ r \in (1,7)\,,
$$
with
\begin{equation} \label{fbar} \overline{f}(t)=1 \ \mbox{ if } t \geq 3\,, \quad \textup{ and } \quad    \overline{f}(t)= -m(m-1)\, t^{\frac{m-2}{m}}\ \mbox{ if } t \in [0, 2^{-m}]\,.
\end{equation}

Also, note that, by the $C^{\ell}$ convergence of $\psi_a(\cdot +a-4)$ to $\overline{\psi}$ as $a$ to infinity, it follows that
$$ 
f_{\pm} \to \overline{f}\,, \quad \mbox{ in the $C^1$ sense, when restricted to } [1/2^m, 3]\,.
$$
Hence, $ f_{\pm} $ converges uniformly to $\overline{f}$, and moreover, the following uniform convergence holds:
$$
f'_{\pm}(t) \mp \frac{m-1}{a_{\pm}} t^{-1/m} - \overline{f}'(t) \to 0\,, \ \ \mbox{ as } a \to  + \infty\,.
$$ 
\end{proof}

\begin{remark} Such example cannot be built for a semilinear problem of the form 
$$
\Delta \psi + f(\psi)=0\,,
$$ 
for any continuous function $f$. Indeed, if $\psi$ is a nonconstant radial solution to this problem, it follows that
$$ 
\psi''(r) + \psi'(r)/r +f (\psi(r))=0\,.
$$
Multiplying by $\psi'(r)$ and integrating we then obtain that
$$ 
\Big ( \frac 1 2 \psi'(r)^2 + F(\psi(r))  \Big )'= -\psi'(r)^2/r \leq 0\,,
$$ 
with equality holding only on the critical points of $\psi$. Here, $F$ denotes a primitive of $f$. Hence,  in this case, the function $P(r) := \frac 1 2 \, \psi'(r)^2 + F(\psi(r)) $ is strictly decreasing in $r$. 
 
If we now choose $0 < a_- < a_+ < \infty$ and impose that $\psi(a_+)=\psi'(a_+)= \psi(a_-)=0$, we would have that
$$ 
\frac 1 2 \, \psi'(a_-)^2 + F(0)= \frac 1 2 \psi'(a_-)^2 + F(\psi(a_-))= P(a_-) > P(a_+) = \frac 1 2 \psi'(a_+)^2 + F(\psi(a_+)) = F(0)\,,
$$ 
and so that $\psi'(a_-) \neq 0$.

\end{remark}

\section{The functional setting} \label{S.FS}

In this section, we establish some basic properties of the spaces $\cX^j$, $\cX^j_{\rm{D}}$, and $\cY^j$ which we introduced in Section \ref{S.MainResult}. 

More precisely, given an open set $\Om' \subseteq \Om := \{(R,\te) \in (1,7) \times \TT\}$ with smooth boundary, we set
\begin{equation} \label{E.XjOm'}
\cX^j(\Om'):= \Big\{ w \in H^j(\Om'): \rho\, \pd_R w \in H^j(\Om')\Big\}\,,
\end{equation}
endowed with the norm
\begin{equation} \label{E.normXjOm'}
\|w\|_{\cX^j(\Om')} := \|w\|_{H^j(\Om')} + \|\rho \pd_R w \|_{H^j(\Om')} \,.
\end{equation}
We recall that the function $ \rho(R) := \tfrac16(R-1)(7-R)$ is positive in~$\pd\Om$ and vanishes on~$\pd\Om$ to first order. Here, $j \geq 2$ is an integer.

We also define the closed subset of elements with zero boundary trace, 
$$
\cX^j_{\rm{D}}(\Om') := \big\{ w \in \cX^j(\Om') : w|_{\pd \Om'} = 0 \big\}
$$
and the Banach space
\begin{equation} \label{E.YjOm'}
\cY^j(\Om') := \rho\, \cX^{j-2}(\Om') + H^{j-1}(\Om')\,,
\end{equation}
endowed with the canonical norm for the sum of two embedded Banach spaces:
\[
\|w\|_{\cY^{j}(\Om')} := \inf\Big\{\|{w_1} \|_{\cX^{j-2}(\Om')} + \|w_2\|_{H^{j-1}(\Om')} :  w = \rho w_1 + w_2  \textup{ with }  {w_1} \in \cX^{j-2}(\Om'),\ w_2 \in H^{j-1}(\Om')\Big\}\,.
\]
When the underlying set is the whole annulus~$\Om$, we omit this fact notationally and  write  $\cX^j := \cX^j(\Om)$, $\cX^j_{\rm{D}} := \cX^j_{\rm{D}}(\Om)$, and $\cY^j := \cY^j(\Om)$. 

Similarly, let us define the following spaces of~$\ZZ_\ell$-symmetric functions, with $\ell\geq3$:
\begin{align*}
\cX^j_\ell(\Om') & := \Big\{ w \in \cX^j(\Om') : w(R,\te) = w(R,-\te)\,, \ u(R,\te) = u\Big(R,\frac{2\pi}{\ell}\Big) \Big\}\,, \\
\cX^j_{{\rm{D}},\ell}(\Om') & :=  \big\{ w \in \cX^j_\ell(\Om') : w|_{\pd \Om'} = 0 \big\}\,, \\
\cY^j_\ell(\Om') & :=  \Big\{ w \in \cY^j(\Om') : w(R,\te) = w(R,-\te)\,, \ u(R,\te) = u\Big(R,\frac{2\pi}{\ell}\Big) \Big\}\,.
\end{align*}
We will omit the set notationally when it is the whole annulus~$\Om$.

\medbreak

Since it will be useful later on, we now prove the following bound, which in particular shows the validity of Lemma \ref{L.Liftings}. Here and in what follows, we use the standard notation $A \lesssim B$ to denote $A \leq C B$ for some harmless constant $ C$.


\begin{lemma} \label{L.LiftingsFS}
For any nonnegative integers $j_1,j_2, l,k$ with $j_1+j_2-k\geq1$, 
$$
\|b^{\sharp}\|_{L^2(\Om)}+ \|\rho^{k+l}\, \pd_\theta^{j_1} \pd^{j_2+l}_{R}\, b^{\sharp}\|_{L^2(\Om)} \leq C \|b\|_{H^{j_1+j_2-k-\frac12}(\TT)}\,.
$$
The function $B^{\flat}$ satisfies an analogous estimate.
\end{lemma}

\begin{proof} 
We prove the result for $b^\sharp$; the case of $B^\flat$ is analogous. By Parseval's identity, 
\begin{align*}
& \frac{1}{2\pi} \int_{\Om} \rho^{2k+2l}\, [\pd_\theta^{j_1} \pd^{j_2+ l}_{\rho}\, b^{\sharp}(R,\te) ]^2 dRd\theta \\
&  \quad \lesssim  \sum_{n \in \ZZ} |b_n|^2 n^{2j_1} (1+|n|)^{2j_2+2l} \int_1^2 \rho^{2l+2k}  \big[\widetilde{\chi}^{(j_2+l)} \big( (1+|n|)(R-1)\big)\big]^2 dR \\
&  \quad \lesssim  \sum_{n \in \ZZ} |b_n|^2 n^{2j_1} (1+|n|)^{2j_2+2l} \int_1^2 (R-1)^{2l+2k}  \big[\widetilde{\chi}^{(j_2+l)} \big( (1+|n|)(R-1)\big)\big]^2 dR \\
& \quad =   \sum_{n \in \ZZ} |b_n|^2 n^{2j_1} (1+|n|)^{2j_2-2k-1} \int_0^\infty  s^{2l+2k} \big[\widetilde{\chi}^{(j_2+l)} (s)\big]^2 ds \\
& \quad \lesssim \sum_{n \in \ZZ} |b_n|^2 n^{2j_1} (1+|n|)^{2j_2-2k-1} \lesssim \|b\|_{H^{j_1+j_2-k-\frac12}(\TT)}\,.
\end{align*}
The estimate for $\|b^\sharp\|_{L^2(\Om)}$ follows similarly.
\end{proof}

Next, we present a Hardy inequality that will be crucial in what follows:

\begin{lemma} \label{L.Hardy}
For any $j \geq 0$,
\begin{equation}
\left\|\frac  w\rho\right \|_{H^j(\Om)} \lesssim  \|w\|_{H^j(\Om)}+ \|\pd_R w\|_{H^j(\Om)}
\end{equation}
provided that $w|_{\pd\Om}=0$. In particular, the multiplication operator $w\mapsto w/\rho$ maps $\cX^j_{\rm{D}}\to \cX^{j-1}$ for all $j\geq { 2}$. 
\end{lemma}

The proof of this result relies on a higher order Hardy-type inequality in one dimension that we state here for completeness. Here and in what follows, we will say that a function $\om$ defined on an interval $(a,b)\subset\RR$ is a {\em weight function}\/ if it is measurable and positive almost everywhere.

\begin{theorem} {\rm{(\hspace{-0.003cm}\cite[Theorem 4.3]{Kufner}).}} \label{T.Kufner43}
Let us fix two weight functions $\om_-,\om_+$ on~$(0,1)$. The inequality
\begin{equation}\label{E.Kufner43}
\int_0^1 |f(t)|^2 \omega_-(t) dt  \lesssim  \int_0^1 |f^{(k)}(t)|^2 \omega_+(t)  dt\,,
\end{equation}
holds for all functions $f$ such that $f(0) = \cdots = f^{(k-1)}(0)=0$ if and only if the weights functions satisfy
\begin{align*}
 \sup_{0 < t < 1}\left[ \left( \int_t^1 (\tau-t)^{2(k-1)} \omega_-(\tau) d\tau \right) \left( \int_0^t \frac{d\tau}{\omega_+(\tau)} \right) +  \left( \int_t^1 \omega_-(\tau) d\tau \right) \left( \int_0^t \frac{(t-\tau)^{2(k-1)}} {\omega_+(\tau)} d\tau \right)\right]<\infty\,.
\end{align*}
\end{theorem}

\begin{proof}[Proof of Lemma \ref{L.Hardy}]
Let $w \in \cX^j_{\rm{D}}$, so that $w(1,\te) = w(7,\te)= 0$. Let us write this function as
$$
w =:  \rho \,  E  +  W \,,
$$
where $ E $ is defined using boundary data and the sharp and flat operations introduced in~\eqref{E.TraceLifting} as
\[
  E (R,\te):= \frac1{\rho(R)}\sum_{k=1}^{j-1} \frac1{k!} \Big[(R-1)^{k}\,(\pd_R^{k}w(1,\cdot))^{\sharp}\,(R,\te) + (R-7)^{k}\,(\pd_R^{k}w(7,\cdot))^{\flat}\,(R,\te)\Big]\,.
\]
This is a sort of Taylor expansion, without any order zero terms because $w|_{\pd\Om}=0$, and chosen so that
\begin{equation}\label{E.bdryW}
\pd_R^k W (1,\cdot)=\pd_R^k W (7,\cdot)=0\qquad \text{for } 0\leq  k\leq j-1.
\end{equation}
and so that the term~$E$ has good bounds in Sobolev spaces. Note that the sharp term in~$E$ is supported on $R\in[1,2]$, while the flat one is supported on~$[6,7]$. Furthermore, note that the functions $(R-1)/\rho$ and $({7-R})/\rho$ are smooth on $[1,2]$ and $[6,7]$, respectively. 

Therefore, arguing as in the proof of Lemma \ref{L.Liftings}, one can readily check that the ``sharp'' terms are bounded as
$$
\big\|\pd_\theta^{j_1}\pd_R^{j_2}\big[\rho^{-1}(R-1)^{k}\,(\pd_R^{k}w(1,\cdot))^{\sharp}\big]\big\|_{L^2(\Om)}\lesssim \|\pd_R^k w(1,\cdot)\|_{H^{j_1+j_2-k+\frac12}(\TT)}\,,
$$
and similarly the ``flat'' terms.
By the trace inequality, we then obtain
\begin{equation} \label{E.HardyConclusion1}
\| E \|_{H^j(\Om)} \lesssim  \|\pd_Rw\|_{H^j(\Om)} \,.
\end{equation}

Let us now estimate~$W$. With $j_1+j_2\leq j$, let us first write
\begin{multline*}
\int_{\Om} \left( \pd_R^{j_1} \pd_\te^{j_2}  \frac{W (R,\te) }{\rho}   \right)^2 dR d\te= \int_{(1,2)\times\TT} \left( \pd_R^{j_1} \pd_\te^{j_2}  \frac{W (R,\te) }{\rho}   \right)^2\\
 + \int_{(2,6)\times\TT} \left( \pd_R^{j_1} \pd_\te^{j_2}  \frac{W (R,\te) }{\rho}   \right)^2 + \int_{(6,7)\times\TT} \left( \pd_R^{j_1} \pd_\te^{j_2}  \frac{W (R,\te) }{\rho}   \right)^2\,.
\end{multline*}
The central region is immediate: as $\rho>0$ on $[2,6]$ and $j_1+j_2\leq j$,
\[
\int_{(2,6)\times\TT} \left( \pd_R^{j_1} \pd_\te^{j_2}  \frac{W (R,\te) }{\rho}   \right)^2 dR d\te \lesssim \|W\|_{H^{j}(\Om)}^2\,.
\]
The remaining two terms can be dealt with similarly, so let us consider the first one.

To use the boundary information information~\eqref{E.bdryW}, let us take the Fourier transform in the $\te$-variable, so that
$$
\frac{1}{\rho}\,  W (R,\te) = \sum_{n \in \ZZ} \frac{ W _n(R)}{\rho} e^{in\te} \quad \textup{ with } \quad  W _n(R):= \frac{1}{2\pi} \int_0^{2\pi}  W (R,\te) e^{-in\te} d \te\,.
$$
Then, we argue as in the proof of Lemma \ref{L.Liftings}: for all nonnegative integers with $j_1+j_2 \leq j$, we have 
\begin{equation} \label{E.Hardy1}
\begin{aligned}
& \int_{[1,2]\times\TT} \left( \pd_R^{j_1} \pd_\te^{j_2}  \frac{W (R,\te) }{\rho}   \right)^2 dR d\te \\
& \quad = 2\pi \sum_{n \in \ZZ} n^{2j_2} \int_1^2 \left| \sum_{m=0}^{j_1} \frac{j_1!}{m! (j_1-m)!} W _n^{(m)}(R)\,  \pd_{R}^{j_1-m} (\rho^{-1})  \right|^2 dR \\
& \quad \lesssim \sum_{n \in \ZZ} n^{2j_2} \sum_{m=0}^{j_1} \int_1^2 \frac{| W _n^{(m)}(R)|^2}{(R-1)^{2(1+j_1-m)}}  dR\,.
\end{aligned}
\end{equation}
Here we have used the Leibniz formula and the fact that $|\pd_R^k(\rho^{-1})|\lesssim (R-1)^{-k-1}$ for $R\in(1,2)$.

Thanks to~\eqref{E.bdryW}, we can now use Theorem \ref{T.Kufner43} with $\tilde R:= R-1\in(0,1)$,  $\omega_-(\tilde R) := \tilde R^{-2(1+j_1-m)}$ and $\omega_+(\tilde R):= 1$ to estimate
\begin{equation} \label{E.Hardy2}
\int_1^2 \frac{| W _n^{(m)}(R)|^2}{(R-1)^{2(1+j_1-m)}}  dR\lesssim \int_1^2 | W _n^{(j_1+1)}(R)|^2 dR \qquad \textup{ for all } 0 \leq m \leq j-1\,.
\end{equation}
Note that the implicit constant is independent of $n$. 
Then, combining \eqref{E.Hardy1} and \eqref{E.Hardy2}, we obtain
\begin{equation*}
\begin{aligned}
& \int_{[1,2]\times\TT} \left( \pd_R^{j_1} \pd_\te^{j_2}  \frac{W (R,\te)}{\rho}    \right)^2 dR d\te \\
& \quad \lesssim \sum_{n \in \ZZ} n^{2j_2} \int_1^2 | W _n^{(j_1+1)}(R)|^2 dR \lesssim \|\pd_R^{j_1} \pd_{\te}^{j_2} \pd_R  W  \|_{L^2(\bD)}^2\lesssim \|\pd_RW\|_{H^j(\Om)}^2\,.
\end{aligned} 
\end{equation*}

\noindent The integral over $[6,7]\times\TT$ is completely analogous. Putting the estimates together, the lemma follows.
\end{proof}

Finally, having Lemma \ref{L.Hardy} at hand, we prove the compactness of the embedding $\cX^j \hookrightarrow \cY^j$:

\begin{lemma} \label{L.CompactEmbedding}
For all $j \geq 2$, the embedding $\cX^j  \hookrightarrow \cY^j$ is compact. 
\end{lemma}

\begin{proof}
First, observe  that the embedding $\cX^j \hookrightarrow H^{j-1}(\Om)$ is obviously compact because so is $H^j(\Om) \hookrightarrow H^{j-1}(\Om)$. On the other hand, note that $\cX^j \cap \rho\, \cX^{j-2} \subset \cX^j_{\rm{D}}$ and multiplication by $\rho^{-1}$ maps $\cX^j_{\rm{D}}\to \cX^{j-1}$ by Lemma~\ref{L.Hardy}. Since the embedding $\cX^{j-2}\hookrightarrow  \cX^{j-1}$ is easily seen to be compact, we conclude that the embedding $\cX^j_{\rm{D}} \hookrightarrow \rho\, \cX^{j-2}$ is also compact. This implies that $\cX^j \hookrightarrow \cY^j$ is compact. 
\end{proof}

\section{The nonlinear map $\cG_a$ and its Fr\'echet derivative}
\label{S.Ga}

In this section we shall analyze the mapping properties of the map $\cG_a$ introduced in \eqref{E.Ga} and we shall compute  its Fr\'echet derivative, proving Lemmas~\ref{L.mapping} and \ref{L.DGa}. 

Having Lemmas \ref{L.Liftings} and \ref{L.Hardy} in hand, it is not hard to prove Lemma \ref{L.mapping}:
 
\begin{proof}[Proof of Lemma \ref{L.mapping}] (i) Let $(w,b,B) \in \widetilde{\cO}^j$ be fixed but arbitrary. We define
\begin{multline*}
 {\mathcal F}_1(w,b,B):=  \, \frac{1}{(\pd_R r)^2} \Big( \rho \pd_R^2 w + 2(m-1) \rho' \pd_R w + (m-1)(\rho \rho'' + (m-2) (\rho')^2 ) \frac{w}{\rho} \Big) \\
 + \frac{1}{r\pd_R r} \Big( \rho \pd_R w + (m-1) \rho' w \Big)-2 \frac{\pd_\te r}{r^2 \pd_R r} \Big( \rho \pd_R \pd_\te w + (m-1) \rho' \pd_\te w \Big) \\
 + \frac{(\pd_\te r)^2}{r^2 (\pd_R r)^2}  \Big( \rho \pd_R^2 w + 2(m-1) \rho' \pd_R w + (m-1)(\rho \rho'' + (m-2) (\rho')^2 ) \frac{w}{\rho} \Big) \,,
\end{multline*}
and
\begin{multline*}
 {\mathcal F}_2(w,b,B) :=  \rho \bigg[ - \frac{\pd_R^2 r}{(\pd_R r)^3} \Big( \pd_R w + (m-1) \rho' \frac{w}{\rho} \Big) + \frac{1}{r^2} \pd_\te^2 w \\
  \quad - \frac{\pd_R^2 r (\pd_\te r)^2 - 2 \pd_{R}r  \pd_{\te} r (\pd_{R}\pd_{\te}r)+(\pd_R r)^2 \pd_{\te}^2 r }{r^2(\pd_R r)^3} \Big( \pd_R w + (m-1) \rho' \frac{w}{\rho} \Big) \bigg]\,,
\end{multline*}
noting that
\[
\rho^{2-m} L_a^{b,B} [\widetilde{\psi}_a + \rho^{m-1} w ] = {\rho^{2-m}} L_a^{b,B}[\widetilde{\psi}_a] + {\mathcal F}_1(w,b,B) + {\mathcal F}_2(w,b,B)\,.
\]
Having  Lemma \ref{L.Hardy} at hand, and taking into account the definition of $\widetilde{\cO}^j$ and the fact that $H^{j-1}(\Om) \subset \cX^{j-2}$, it is easy to check that 
\[
\rho^{2-m} L_a^{b,B}[\widetilde{\psi}_a] + {\mathcal F}_1(w,b,B) + f_a(r,\ \widetilde{\psi}_a + \rho^{m-1}w ) \in  H^{j-1}(\Om)
\]
and that 
$$
{\mathcal F}_2(w,b,B) \in \rho\, \cX^{j-2}\,.
$$

(ii) Let $w \in \cX^j$ be fixed but arbitrary. First, note that
$$
\rho^{2-m} \pd_R^2 (\rho^{m-1} w) = \rho \pd_R^2 w + 2(m-1) \rho' \pd_R w + (m-1)\rho'' w + (m-1)(m-2) (\rho')^2\, \frac{w}{\rho}\,.
$$
Thus, by Lemma \ref{L.f'}, 
$$
\rho^{2-m} \Big[ \pd_R^2 (\rho^{m-1} w) +  f_a'(\Phi_{a,1}, \widetilde{\psi}_a) \rho^{m-1} w \Big]=  \rho \pd_R^2 w + 2(m-1) \rho' \pd_R w + \big((m-1)\rho'' + O(1)\big) w \,.
$$
Having this expansion  at hand, we set 
\begin{align*}
\widetilde{\cF}_1(w)&:= \rho^{2-m} \pd_R^2 (\rho^{m-1} w) + \rho f_a'(\Phi_{a,1}, \widetilde{\psi}_a)  w + \frac{\rho \pd_R w + (m-1) \rho' w}{R+a-4}  \,,\\
\widetilde{\cF}_2(w)&:=  \frac{\rho }{(R+a-4)^2} \pd_\te^2 w \,.
\end{align*}
Then
$$
\rho^{2-m} \big( \big( L_a^{0,0} + f_a'(\Phi_{a,1},\, \widetilde{\psi}_a)  \big) [ \rho^{m-1} w ] \big) = \widetilde{\cF}_1(w) + \widetilde{\cF}_2(w)\,,
$$
and it is easy to check that $\widetilde{\cF}_1(w) \in  H^{j-1}(\Om)$ and  $\widetilde{\cF}_2(w) \in \rho\, \cX^{j-2}$.
\end{proof}

\begin{remark} \label{R.cancellation}
In the proof of (ii), we are using that the term
$
(m-1)(m-2) w/\rho 
$
cancels out. This allows us to extend the definition of the operator from $\cX^j_{\rm{D}}$ to $\cX^j$. 
\end{remark}

\begin{proof}[Proof of Lemma \ref{L.DGa}]

Let us start by noting that the linear map $w\mapsto (b_w,B_w,\Theta_w)$ defined by~\eqref{E.wto} is continuous $\cO^j\mapsto\widetilde\cO^j$. Furthermore,
\[
v\mapsto \rho^{2-m}L_a^{b,B}(\rho^m v)
\]
is a second order differential operator whose coefficients depend smoothly on~$b^\sharp,B^\flat$ (provided they satisfy the smallness assumption included in the definition of~$\widetilde\cO^j$). It is then clear that 
\[
\cG^1(w):= \rho^{2-m}L_a^{b_w,B_w}(\rho^m\widetilde\Psi_a)
\]
is a $C^1$ map $\cO^j\to \cY^j$. 

To analyze the map $\cG^2(w):= \rho^{2-m}L_a^{b_w,B_w}(\rho^{m-1}\Theta_w)$, let us write it as
\[
\cG^2(w)=\rho L_a^{b_w,B_w}\Theta_w-[\rho^{2-m}L_a^{b_w,B_w},\rho^{m-1}]\Theta_w\,.
\]
In the commutator, the only dangerous term that appears is when the second order derivative $\pd_R^2$ appearing in $L_a^{b_w,B_w}$ hits $\rho^{m-1}$. This leads to a term of the form
$
{\Theta_w}\,\widehat\cG^2(b_w,B_w)/\rho,
$
where {$\widehat\cG^2(b_w,B_w)$} is harmless (and depends on~$w$ in a differentiable fashion). Since $\Theta_w\in\cX^j_{{\rm D}}$, this term is controlled by the Hardy inequality of Lemma~\ref{L.Hardy}, so $\cG^2:\cO^j\to \cY^j$ is also $C^1$.

Let us now consider 
\[
\cG^3(w):= \rho^{2-m}f_a (\Phi_{a,1}^{b_w,B_w},\ \widetilde{\psi}_a + \rho^{m-1} \Theta_w  )\,.
\]
If $\widehat\chi(R)$ is a smooth cutoff function which is identically zero in a neighborhood of~1 and~7, it is clear that {$w\mapsto  \widehat\chi\, \cG^3(w)$} is~$C^1$. To study what happens near, say, $R=1$, note that Lemma~\ref{L.f} ensures that for $R\in[1,2]$ one can write
\[
\cG^3(w)= \left(\widetilde{\Psi}_a +\frac{\Theta_w}\rho\right)^{1-\frac2m}g_-(\widetilde{\psi}_a + \rho^{m-1} \Theta_w)\,,
\]
where $g_-$ is smooth. In view of the smallness assumption of~$\cO^j$, it is clear that $\cG^3$ is also~$C^1$. Since $\cG_a=\cG^1+\cG^2+\cG^3$, we conclude that $\cG_a:\cO^j\to \cY^j$ is continuously differentiable.

Let us now prove the formula \eqref{E.DGa} for $D\cG_a(0)$. To this end, let us write $\Theta_w$ as
$
\Theta_w = w + \widetilde\Theta_w,
$
where
$$
\widetilde\Theta_w(R,\te) := \frac{1}{6} \Big(m\rho'(R)\, \widetilde{\Psi}_a(R) + \rho(R) \widetilde{\Psi}_a'(R) \Big) \Big[\Bsharp_w(R,\te) (R-1) + \bsharp_w(R,\te)(7-R) \Big]\,. 
$$
Then, by direct computations, we get that
\begin{equation} \label{E.DGa1}
	\begin{aligned}
	\frac{\rm d}{{\rm d}s} & \left(  \rho^{2-m} \Big[ L_a^{sb_w,sB_w} \big( \widetilde{\psi}_a + \rho^{m-1} s \Theta_w \big) + f_a \big(\Phi_{a,1}^{sb_w,s B_w}, \, \rho^m \, \widetilde{\Psi}_a + \rho^{m-1} s\Theta_w \big) \Big] \right) \bigg|_{s = 0}  \\
	& = \rho^{2-m} \bigg( L_a^{0,0} \big[\rho^{m-1} \big( w + \widetilde\Theta_w \big) \big] +  \frac{\rm d}{{\rm d}s} L_a^{sb_w, sB_w} \Big|_{s=0} \big[ \widetilde{\psi}_a \big] \\
	& \quad + \frac16 \partial_r f_a(\Phi_{a,1},\widetilde{\psi}_a) \big( b_w(7-\cdot) + B_w(\cdot-1) \big) + \rho^{m-1} f_a'(\Phi_{a,1}, \rho^m \, \widetilde{\Psi}_a ) ( w + \widetilde\Theta_w \big)  \bigg)\,.
\end{aligned}
\end{equation}

On the other hand, we write $\psi_a(r,\te) \equiv \psi_a(r)$ with some abuse of notation and define $\widetilde{\psi}_a^{b,B} := \psi_a \circ \Phi_a^{b,B}$, where we recall that $\psi_a \in C^\infty(\BOm_a)$ is the solution to \eqref{E.psi} in $\Om_a$ constructed in Lemma \ref{L.f}. Note that 
$$
L_a^{b,B}\, \widetilde{\psi}_a^{b,B} + f_a(\,|\cdot|,\widetilde{\psi}_a^{b,B} \circ (\Phi_a^{b,B})^{-1}) \circ \Phi_a^{b,B} = 0 \quad \textup{in } \Om\,.
$$
Thus, substituting $(b,B) = (sb_w, sB_w)$ and differentiating the resulting identity, we find 
\begin{equation*}
\begin{aligned}
0 & = 	\frac{\rm d}{{\rm d}s} L_a^{sb_w, sB_w} \Big|_{s=0} \big[ \widetilde{\psi}_a \big] + L_a^{0,0} \left[	\frac{\rm d}{{\rm d}s} \widetilde{\psi}_a^{sb_w, sB_w} \Big|_{s=0} \right] \\
& \quad  +  \frac16\,  \partial_r f_a(\Phi_{a,1},\widetilde{\psi}_a) \big( b_w(7-\cdot) + B_w(\cdot-1) \big) + f_a'(\Phi_{a,1}, \rho^m \, \widetilde{\Psi}_a ) \, \frac{\rm d}{{\rm d}s} \widetilde{\psi}_a^{sb_w,sB_w} \Big|_{s=0}\,.
\end{aligned}
\end{equation*}
Also, note that
$$
 \frac{\rm d}{{\rm d}s} \widetilde{\Psi}_a^{sb_w,sB_w} \Big|_{s=0} = \rho^{m-1}\, \widetilde\Theta_w\,.
$$
Hence,
\begin{equation} \label{E.DGa2}
\begin{aligned}
0 & = 	\rho^{2-m} \bigg(\frac{\rm d}{{\rm d}s} L_a^{sb_w, sB_w} \Big|_{s=0} \big[ \widetilde{\psi}_a \big] + L_a^{0,0} \left[\rho^{m-1}\, \widetilde\Theta_w\right] \\
& \quad  +  \frac16\,  \partial_r f_a(\Phi_{a,1},\widetilde{\psi}_a) \big( b_w(7-\cdot) + B_w(\cdot-1) \big) +  \rho^{m-1}\, f_a'(\Phi_{a,1}, \rho^m \, \widetilde{\Psi}_a ) \, \widetilde\Theta_w \bigg) \,.
\end{aligned}
\end{equation}
Since $L_a^{0,0}$ is a linear operator, \eqref{E.DGa}  immediately follows by combining \eqref{E.DGa1} and \eqref{E.DGa2}.
\end{proof}

\section{Fredholmness and invertibility properties of the linearized operator} \label{S.Regularity}

This section is devoted to the proof of Theorem \ref{T.regularity}. Using Lemma \ref{L.CompactEmbedding}, it is easy to show the compactness of the operator~$\cK_a$ defined in~\eqref{E.cKa}. The hardest part is to show that $\cT_a$ (introduced in~\eqref{E.cTa}) is an isomorphism. We split the proof into two parts. In the first part, we shall prove an easy existence and uniqueness result on low regularity spaces for the linear elliptic equation defined by the operator~$\cT_a$. In the second part, which is considerably harder, we prove sharp regularity estimates for the solutions. Theorem \ref{T.regularity} then follows. In particular, $D\cG_a(0)$ is a Fredholm operator of index 0. The proof of a key estimate (Lemma \ref{L.keyest}) is presented in Section~\ref{S.keyest}.
 
\begin{lemma}\label{L.compactness}
For all $j\geq2$, the operator $\cK_a:\cX^j \to \cY^j$ is compact.
\end{lemma}

\begin{proof}
By Lemma \ref{L.CompactEmbedding} we know that the embedding $\cX^j \hookrightarrow \cY^j$ is compact. Also, note that $\rho\, \pd_R$ maps $ \cX^j\to H^j(\Om)$. It is apparent that $H^j(\Om)$ is compactly embedded in~$\cY^j$, so $\rho\, \pd_R:\cX^j\to \cY^j$ is compact.

Now, observe that, by Lemma~\ref{L.f'}, the operator~$\cK_a$ is of the form 
$$
\cK_a w= K_1\,\rho\,\pd_Rw+ K_2\, w\,,
$$
where $K_1,K_2:\BOm\to\RR$ are smooth radial functions. The compactness of $\cK_a$ then follows.
\end{proof}

\subsection{Existence and uniqueness in a low regularity space}

Let us consider the equation 
\begin{equation}\label{E.TawF}
\cT_a w = F	
\end{equation}
which one can rewrite using~\eqref{E.cTa} as
\begin{equation}\label{E.eqwcTa}
	L_a^{0,0}(\rho^{m-1}w)-\frac{\m}{\rho^2}\rho^{m-1}w= \rho^{m-2}F\,.
\end{equation}
Let us recall that the differential operator $L_a^{0,0}$ was introduced in~\eqref{E.La00}. The following result is then elementary:

\begin{proposition}\label{P.lowreg}
	For each $F\in L^2(\Om)$, there exists a unique $w\in \rho^{1-m} H^1_0(\Om)$ such that $\cT_a w= F$. Furthermore,
	\[
	\|\rho^{m-1} w\|_{H^1(\Om)}\lesssim \|F\|_{L^2(\Om)}\,.
	\]
\end{proposition}

\begin{proof}
	Let us write $W:= \rho^{m-1}w\in H^1_0(\Om)$, so that~\eqref{E.cTa} reads as
	\begin{equation}\label{E.eqcTa}
	L_a^{0,0}W-\frac{\m}{\rho^2}W= \rho^{m-2}F\,.
\end{equation}
	Noting that $(\erre) \, dRd\theta$ is a positive measure on~$\Om$ because $a\geq 4$ and $R\in (1,7)$, let us set	\begin{equation}\label{E.Q}
	\cB_0(U,V):=\int_{\Om} \left[ \pd_RU\, \pd_RV+ \frac{\pd_\theta U\,\pd_\theta V}{(\erre)^2}+\frac\m{\rho(R)^2}UV\right]\, (\erre) \, dRd\theta\,.
	\end{equation}
	This is the bilinear form associated with the equation~\eqref{E.eqcTa}, as can be seen from the fact that
\begin{align*}
	 \cB_0(U,V) = -\int_\Om \left[ L_a^{0,0}U - \frac{(m-1)(m-2)}{\rho^2} U \right] V\, (\erre)\, dR\,d\theta 
	\end{align*}
if $U,V\in H^1_0(\Om)\cap H^2(\Om)$.
	
Note that $\cB_0: H^1_0(\Om)\times H^1_0(\Om)\to \RR$ is well-defined by Hardy's inequality, and obviously 
$$
\cB_0 (U,U)\geq \|U\|_{H^1(\Om)}^2\,.
$$
By Lax--Milgram (or by the Riesz representation theorem, as the quadratic form is symmetric), we conclude that for each $F\in L^2(\Om)$ there exists a unique $W\in H^1_0(\Om)$ such that
\[
\cB_0(W,V)=\int_{\Om} \rho^{m-2}FV\, (\erre)\, dRd\theta\qquad \text{for all } V\in H^1_0(\Om)\,,
\]
which is moreover bounded as
\[
\|W\|_{H^1(\Om)}\lesssim \|\rho^{m-2}F\|_{H^{-1}(\Om)}\lesssim \|F\|_{L^2(\Om)}\,.
\]
The proposition then follows. 	
\end{proof}

Although this elementary result yields solution to the equation $\cT_a w =F$, it is not obvious a priori that $w\in \cX^j$ when $F\in \cY^j$, so we cannot conclude from this result that $\cT_a$ is an isomorphism between these spaces. What we do obtain from this result and from the fact that $\cX^j\subset \rho^{1-m}H^1_0(\Om)$ is that the kernel of $\cT_a$ is trivial:

\begin{corollary}\label{C.injective}
The map $\cT_a: \cX^j\to \cY^j$ is injective.
\end{corollary}

\subsection{Regularity estimates}

Our objective now is to show that $w\in \cX^j$ whenever $F\in \cY^j$. Away from $\pd\Om$, the elliptic operator in~\eqref{E.eqwcTa} is uniformly elliptic, so  regularity considerations are elementary. Near the boundary, the situation is much more involved and requires further analysis.

To make this precise, for each small $\ep >0$, let us define the sets
\[
\Om_\ep^-:=(1,1+\ep)\times\TT\,,\qquad \Om_\ep^+:=(7-\ep,7)\times\TT\,,
\]
and well as the complement
\[
\Omg_\ep:=(1+\ep,7-\ep)\times\TT\,.
\]
Since $\ep/2<\rho<2$ in $\Omg_\ep$, an elementary elliptic estimate for~\eqref{E.eqwcTa} yields 
\begin{align*}
	\|w\|_{H^j(\Omg_\ep)}&\leq  C_\ep\|\rho^{m-2}F\|_{H^{j-2}(\Omg_{\ep/2})}+ C_\ep\|w\|_{H^1(\Omg_{\ep/2})} \\
	&\leq  C_\ep\|F\|_{H^{j-2}(\Omg_{\ep/2})}+ C_\ep\|\rho^{m-1}w\|_{H^1(\Omg_{\ep/2})}\,.
\end{align*}
The last quantity can be controlled in terms of $\|F\|_{L^2(\Om)}$ by Proposition~\ref{P.lowreg}, and thus we obtain the following we conclude
\begin{equation*}
	\|w\|_{H^j(\Omg_\ep)}\leq C_\ep \|F\|_{H^{j-2}(\Om)}\,.
\end{equation*}
Here and in what follows, we denote by $C_\ep$ constants (which may vary from line to line) which are not uniformly bounded as $\ep\to0$. Constants that are uniform in~$\ep$ are regarded as harmless, so we typically omit them using the symbol $\lesssim$ as before.

After differentiating~\eqref{E.eqwcTa} with respect to~$R$, an analogous argument shows that
\begin{equation*}
	\|\pd_Rw\|_{H^j(\Omg_\ep)}\leq C_\ep \|F\|_{H^{j-2}(\Om)}+ C_\ep \|\pd_RF\|_{H^{j-2}(\Om)}\,.
\end{equation*}
Since
\[
\|w\|_{\cX^j(\Omg_{\ep})}\lesssim \|w\|_{H^j(\Omg_{\ep})}+ \|\pd_Rw\|_{H^j(\Omg_{\ep})} \,,
\]
replacing $\ep$ by $\ep/2$ for later convenience, we thus arrive at the following estimate:

\begin{lemma}\label{L.Omgepbound}
For any small $\ep>0$,
\begin{equation*}
	\|w\|_{\cX^j(\Omg_{\ep/2})} \leq C_\ep \|F\|_{\cY^j}\,.
\end{equation*}
\end{lemma}

Hence, we only need to estimate $w$ on $\Om_{\ep/2}^\pm$ (which is of course where the problem is nontrivial). For concreteness, let us consider only the region $\Om_{\ep/2}^-$, since the case of $\Om_{\ep/2}^+$ is completely analogous. 

Consider the function
\begin{equation}\label{E.defwep}
w^\ep(R,\theta):=w(R,\theta)\, \chi_\ep(R)\,,
\end{equation}
where the smooth cutoff function is defined in terms of the function $\widetilde\chi$ that we introduced in~\eqref{E.chibar} as
\[ 
\chi_\ep(R):=\widetilde\chi((R-1)/\ep)\,.
\]
By construction, $w^\ep$ coincides with $w$ on $\Om_{\ep/2}^-$ and is supported in $\Om_\ep^-$. 

Now, note that 
\[
\rho^{2-m}L_a^{0,0}(\rho^{m-1}w^\ep)= \rho^{2-m}L_a^{0,0}(\rho^{m-1}w)\, \chi_\ep + F_\ep
\]
with
\[
F_\ep:= \rho^{2-m}\left[2\pd_R\chi_\ep \, \pd_R(\rho^{m-1}w) + \left(\pd_R^2\chi_\ep + \frac{\pd_R\chi_\ep}{\erre} \right)\rho^{m-1} w\, \right]\,.
\]
Since $F_\ep$ is supported on $\Om^-_{\ep}\backslash\Om^-_{\ep/2}$, it follows from Lemma \ref{L.Omgepbound} that
\begin{equation}\label{E.Fep}
\|F_\ep\|_{\cY^j}\leq \|F_\ep\|_{H^{j-1}(\Om)}\leq C_\ep  \|w_\ep\|_{H^j(\Om^-_{\ep}\backslash\Om^-_{\ep/2})}\leq C_\ep \|F\|_{\cY^j}\,.	
\end{equation}
Therefore, the function~$w^\ep$ satisfies the equation 
\begin{equation}\label{E.wepFep}
	\cT_a w^\ep= \cF_\ep\,,
\end{equation}
where the function $\cF_\ep:= F\chi_\ep +F_\ep$ is bounded as
\begin{equation}\label{E.Gbound}
\|\cF_\ep\|_{\cY^j}\leq C_\ep\|F\|_{\cY^j}\,.
\end{equation}

To analyze this equation, let us introduce a new radial variable~$z$ as
\[
R(z):= 1+ \ep z
\]
and denote by $v^\ep(z,\theta)$ and $G_\ep(z,\theta)$ the expression of $w^\ep$ and $\cF_\ep$ in the new variables, that is,
\begin{equation}\label{E.defvepGep}
v^\ep(z,\theta):= w^\ep(R(z),\theta)\,,\qquad G_\ep(z,\theta):= \cF_\ep(R(z),\theta)\,.
\end{equation}
It is clear that the support of $v^\ep$ and $G_\ep$ is contained in $[0,1]\times\TT$, although of course we can regard them as functions $(0,\infty)\times\TT\to\RR$. Denoting by $\tcT_a$ the expression of the differential operator $\cT_a$ in these coordinates, one has
\[
\tcT_av^\ep(z,\theta):=(\cT_a w^\ep)(R(z),\theta)\,.
\]

We need a more explicit formula for the differential operator $\tcT_a$. A straightforward computation shows that one can decompose
\[
\tcT_a v=\frac1\ep(\cL v  + \ep \cE v)\,,
\]
where\footnote{If one estimates $w$ on $\Om_{\ep/2}^+$ instead of on~$\Om_{\ep/2}^-$, setting $R(z):=7-\ep z$, one arrives at operators of the same form but with $a_+$ instead of $a_-$.} 
\begin{align*}
	\cL v&:=z\pd_z^2 v +\frac{\ep^2}{a_-^2}z\,\pd_\theta^2v +(2m-2)\pd_z v\,,\\
	\cE v&:=E_1(\ep z)\,z^2\pd_z^2 v +E_2(\ep z)\,{\ep^2}z^2\,\pd_\theta^2v +E_3(\ep z)\,z\pd_z v+ E_4(\ep z)v\,. 
\end{align*}
Here, $E_j$ are certain analytic functions on $[0,1]$ whose explicit expression we will not need, and we recall that $a_-:= a-3$. We omit the dependence of these operators on~$a$ for notational convenience.

\begin{remark}\label{R.cKa}
	For future reference, let us record here an analogous formula for the operator~$\cK_a$, defined in~\eqref{E.cKa}, namely
	\[
	\widetilde\cK_av^\ep(z,\theta):=(\cK_a w^\ep)(R(z),\theta)= E_5(\ep z)v
	\]
	for some analytic function on~$[0,1]$.
\end{remark}

The operator~$\cL$ is the scale-critical part of the operator, which controls the regularity properties of~$\cT_a$ near the endpoint, where it becomes degenerate. To derive sharp estimates, let us now consider the corresponding equation
\begin{equation}\label{E.LvG}
	\cL v=G\qquad \text{in } (0,1)\times\TT\,.
\end{equation}
To present the key a priori estimate we need, let us define the $\ep$-dependent norms
\begin{equation} \label{E.normsDisk}
\begin{aligned}
\|v\|_{\cX^j_\ep} & := \sum_{j_2=0}^{j-1} \sum_{j_1=0}^{j-j_2} \Big( \|\pd_z^{j_1} \pd_\te^{j_2} v\|_{L^2((0,1)\times \TT)} + \|z \pd_z^{j_1+1} \pd_\te^{j_2} v\|_{L^2((0,1)\times \TT)} \Big) \\ & \hspace{5.15cm} + \ep \Big( \|\pd_\te^j v\|_{L^2((0,1)\times \TT)} + \|z\pd_z \pd_\te^j v\|_{L^2((0,1)\times \TT)} \Big)\,, \\
\|G\|_{\cY^j_\ep} &:=\inf\Big\{ \|G_1\|_{H^{j-2}((0,1)\times \TT)} + \|z\pd_z G_1\|_{H^{j-2}((0,1)\times \TT)} + \|G_2\|_{H^{j-1}((0,1)\times \TT)}: G= \ep z\,G_1+G_2 \Big\}\,,
\end{aligned}
\end{equation}
For each fixed $\ep>0$, it is clear that these norms are equivalent to the $\ep$-independent norms that we have used so far, so in particular
\begin{equation}\label{E.equivnorm}
	\|w^\ep\|_{\cX^j}\leq C_\ep \|v^\ep\|_{\cX^j_\ep}\,,\qquad \|\cF_\ep\|_{\cY^j_\ep}\leq C_\ep\|\cF_\ep\|_{\cY^j(\Om)}\leq C_\ep\|F\|_{\cY^j}\,.
\end{equation}

We are now ready to state the basic estimate we need. Note that the implicit constants in the statements are independent of~$\ep$, and that they are uniform in $a_-$ in compact subsets of $[1,+\infty)$.

\begin{lemma}\label{L.keyest}
	Suppose that $G\in \cY^j_\ep$. There is a unique solution $v\in z^{1-m} H^1_0((0,1)\times\TT)$ to the equation~\eqref{E.LvG}, which is furthermore bounded as
	\[
	\|v\|_{\cX^j_\ep}\lesssim \|G\|_{\cY^j_\ep}\,.
	\]
\end{lemma}

Assuming Lemma~\ref{L.keyest}, whose proof is postponed to Section \ref{S.keyest}, we are now ready to prove the main result of this section. Together with Lemma~\ref{L.compactness}, this proves Theorem~\ref{T.regularity}:

\begin{theorem}\label{T.iso}
	The operator $\cT_a : \cX^j\to\cY^j$ is an isomorphism.
\end{theorem}

Indeed, this result is a fairly direct consequence of Lemma~\ref{L.keyest} and of the following estimate for the error, which is obtained essentially by inspection:

\begin{lemma}\label{L.cE}
	The operator $\cE:\cX^j_\ep \to\cY^j_\ep$ satisfies
	\[
	\|\cE v\|_{\cY^j_\ep}\lesssim \|v\|_{\cX^j_\ep}\,.
	\]
\end{lemma}

\begin{proof}
We decompose the operator $\cE$ as $\cE  =  \ep z \, \cE_1 + \cE_2$ with
$$
\cE_1 v:= \ep z \pd^2_\te\,  v \quad \textup{ and } \quad \cE_2 v:=  z^2 \pd_z^2 v + z \pd_z v  + v\,.
$$
Taking into account the definition of $\cX^j_\ep$, it is then straightforward to check that 
$$
\|\cE_1v\|_{H^{j-2}((0,1)\times \TT)} + \|z\pd_z \cE_1v\|_{H^{j-2}((0,1)\times \TT)}\lesssim \|v\|_{\cX^j_\ep} \quad \textup{ and } \quad \|\cE_2 v\|_{H^{j-1}((0,1) \times \TT)}\lesssim \|v\|_{\cX^j_\ep}\,,
$$
and so that
	\[
	\|\cE v\|_{\cY^j_\ep}\lesssim \|v\|_{\cX^j_\ep}\,.
	\]
The error estimate then follows.	
\end{proof}

\begin{proof}[Proof of Theorem~\ref{T.iso}]
Since $\cT_a: \cX^j \to \cY^j$ is a linear continuous operator, by the open mapping theorem, it suffices to show that it is a bijective map. Corollary~\ref{C.injective} ensures that it is injective, so we need to show it is onto. To this end, take any $F\in\cY^j$. Proposition~\ref{P.lowreg} ensures that there is a unique $w\in \rho^{1-m} H^1_0(\Om)$ satisfying $\cT_a w= F$. Furthermore, for any $\ep>0$, Lemma~\ref{L.Omgepbound} ensures 
	\[
	\|w-w_1^\ep-w_7^\ep\|_{\cX^j}\leq C_\ep\|F\|_{\cY^j}\,.
	\]
with $w_1^\ep:=w^\ep$ given by~\eqref{E.defwep} and $w_7^\ep$ defined analogously, namely
$$
w_7^\ep(R,\te):= w(R,\te) \widetilde{\chi}((7-R)/\ep)\,.
$$ 

Therefore, it only remains to show that there is some $\ep>0$, independent of~$F$, for which the a priori estimate
\begin{equation*} 
\|w_1^\ep\|_{\cX^j} + \|w_7^\ep\|_{\cX^j}\leq C_\ep\|F\|_{\cY^j}
\end{equation*}
holds. Let us prove that
\begin{equation}\label{E.wepfinal}
\|w_1^\ep\|_{\cX^j} \leq C_\ep\|F\|_{\cY^j}\,,
\end{equation}
as the bound
\begin{equation*}
\|w_7^\ep\|_{\cX^j} \leq C_\ep\|F\|_{\cY^j}\,,
\end{equation*}
is completely analogous.

Let us define~$v^\ep,G_\ep$ as in~\eqref{E.defvepGep}. Since $v^\ep\in z^{1-m} H^1_0((0,1)\times\TT)$ and 
	\begin{equation}\label{E.cLvepG}
	\cL v^\ep+\ep \cE v^\ep =\ep G_\ep\,,		
	\end{equation}
	Lemma~\ref{L.keyest} ensures that
	\[
	\|v^\ep\|_{\cX^j_\ep}\lesssim \|\cL v^\ep\|_{\cY^j_\ep}\leq \ep \| G_\ep \|_{\cY^j_\ep} + \ep \|\cE v^\ep \|_{\cY^j_\ep} \lesssim  \ep \| G_\ep \|_{\cY^j_\ep} + \ep \| v^\ep \|_{\cX^j_\ep} \,.
	\]
	To pass to the last inequality we have used Lemma~\ref{L.cE}. Since the implicit constant is independent of~$\ep$, we can take $\ep$ small enough (but independent of~$F$) such that
	\[
	\|v^\ep\|_{\cX^j_\ep}\lesssim \ep  \| G_\ep \|_{\cY^j_\ep}\,.
	\]
The bound~\eqref{E.wepfinal} then follows from \eqref{E.normsDisk} and ~\eqref{E.equivnorm}.
\end{proof}

\begin{proof}[Proof of Theorem \ref{T.regularity}]
The result immediately follows from Lemma \ref{L.compactness} and Theorem \ref{T.iso}.
\end{proof}

A straightforward consequence of our analysis of the operator~$\cT_a$ is an analogous regularity result for the full linearized operator $D\cG_a(0)$. While the regularity theory carries over verbatim to this setting, one should note that the existence and uniqueness part certain does not, as we will crucially use in the bifurcation argument later on. 

\begin{proposition}\label{P.regDGa}
Let $w\in \rho^{1-m} H^1_0(\Om)$ satisfy the equation
\[
D\cG_a(0)w = F
\]
for some $F\in \cY^j$. Then $w\in\cX^j$ and moreover
$$
\|w\|_{\cX^j} \lesssim \|F\|_{\cY^j} + \|\rho^{m-1} w\|_{H^1(\Om)}\,.
$$
\end{proposition}

\begin{proof}
Let us start by recalling that $D\cG_a(0)= \cT_a+\cK_a$ by~\eqref{E.cTa}-\eqref{E.cKa} and that $\cK_a w$ is the multiplication of~$w$ by a smooth function. In particular, just as in the case of Lemma~\ref{L.Omgepbound}, standard elliptic regularity estimates show that
	\[
\|w\|_{\cX^j(\Omg_{\ep/2})} \leq C_\ep \|F\|_{\cY^j}+ C_\ep\| \rho^{m-1} w\|_{H^1(\Om)}\,,
	\]
the only difference being that in this case one must keep the second summand on the right hand side because one cannot control it using an analog of Proposition~\ref{P.lowreg}. 

As in~\eqref{E.defwep}, we introduce a cutoff, and the resulting function satisfies the equation
\[
D\cG_a(0)w^\ep = F_\ep
\]
for some function $F_\ep$ bounded as
\[
\|F_\ep\|_{\cY^j}\leq C_\ep  \|F\|_{\cY^j}+ C_\ep \| \rho^{m-1} v\|_{H^1(\Om)}\,.
\]	
Let us now rescale the radial variable, defining $v^\ep$, $G_\ep$ as in~\eqref{E.defvepGep}. By Remark~\ref{R.cKa}, we have the equation
\[
\frac1\ep \cL v^\ep + \cE v^\ep + \widetilde\cK_a v^\ep = G_\ep
\]
on $(0,1)\times\TT$, 
where $\widetilde\cK_a$ satisfies the same bounds as~$\cE$. Therefore,
arguing as in~\eqref{E.cLvepG}, one finds that 
\[
\|v^\ep\|_{\cX^j_\ep}\lesssim \ep \|G_\ep\|_{\cY^j_\ep}\,.
\]
Thus, we arrive at the a priori estimate
\[
\|w^\ep\|_{\cX^j}\leq C_\ep \|F\|_{\cY^j}+ C_\ep \| \rho^{m-1} v\|_{H^1(\Om)}\,,
\]
which ensures that $w \in \cX^j$.
\end{proof}

\section{Analysis of the operator $\cL$}
\label{S.keyest}

In this section we shall prove the key estimate for the equation $\cL v= G$ presented in Lemma~\ref{L.keyest}. For clarity, we will divide the proof into two parts. In the first one we shall  estimate the solution to certain ODE which depends on a parameter~$\al\geq0$. The point of these estimates is that one needs to capture the sharp dependence of the constants on the parameter~$\al$. In the second part of this section, we shall see that this ODE arises from the PDE $\cL v =G$ after taking the Fourier transform in the angular variable, with the parameter~$\al$ being essentially the Fourier frequency multiplied by the small scale parameter~$\ep$ introduced in Lemma~\ref{L.keyest}. Sharp ODE bounds then capture the interplay between large frequencies and the effect of the singularity on the boundary of the annular domain~$\Om$, and translate into sharp estimates for~$\cL$.

\subsection{Auxiliary ODE estimates} \label{S.keyestODE}

A straightforward variation on Proposition~\ref{P.lowreg} using the quadratic form
\[
\cQ_0(U):= \int_0^1\left[ (\pd_z U)^2 +\left( \al^2+\frac\m{z^2}  \right)U^2 \right]dz
\]
shows that for each $f\in L^2((0,1))$ (or, more generally, in $z^{2-m} H^{-1}((0,1))$), and each $\al\geq0$ there exists a unique solution $\vp\in z^{1-m} H^1_0((0,1))$ to the ODE
\begin{equation} \label{E.ODEvp}
z\pd_z^2 \vp + (2m-2) \pd_z \vp - \al^2 z\vp = f\,.
\end{equation}
Our objective is to derive a priori estimates for~$\vp$ with the sharp dependence on the parameter~$\al$, under the assumption that $f \in H^\ell((0,1))$ for some integer $\ell\geq0$.

We shall start off by decomposing $f \in H^\ell((0,1))$ as
$$
f(z) =: \sum_{k=0}^{\ell-1} \frac{z^k}{k!}f^{(k)}(0) \widetilde{\chi}((1+\al)z) + \cR_f(z)\,,
$$
where the cutoff function $\widetilde{\chi}$ was defined in \eqref{E.chibar}. Obviously, if $\ell = 0$, the first term is absent and we do not need to decompose~$f$. With this expression for~$f$, we can write~$\vp$ as
\begin{equation*}
\vp := \sum_{k=0}^{\ell-1} \frac{1}{k!} \vp_k + \vp_\cR\,,
\end{equation*}
in terms of the only solutions $\vp_k,\vp_\cR \in z^{1-m} H_0^1((0,1))$  to the equations
\begin{align} \label{E.ODEvpk}
z \pd_z^2 \vp_k +(2m-2) \pd_z \vp_k - \al^2 z \vp_k &= z^{k} f^{(k)}(0) \widetilde{\chi}((1+\al)z)\,,  \\
\label{E.ODEvpR}
z \pd_z^2 \vp_\cR +(2m-2) \pd_z \vp_\cR - \al^2 z \vp_\cR &= \cR_f\,.
\end{align}

A particular case that we are particularly interested in is when $f(z) := z g(z)$ with $g \in H^{\ell}((0,1))$. If we now plug in the decomposition
$$
g(z) =: \sum_{k=0}^{\ell-1} \frac{z^k}{k!}g^{(k)}(0) \widetilde{\chi}((1+\al)z) + \cR_g(z)\,,
$$
one then has
\begin{equation*}
\vp := \sum_{k=0}^{\ell-1} \frac{1}{k!} \tvp_k + \tvp_\cR
\end{equation*}
in terms of the only solutions $\tvp_k,\tvp_\cR \in z^{1-m}H_0^1((0,1))$ to the equations
\begin{align} \label{E.ODEvpkg}
z \pd_z^2 \tvp_k +(2m-2) \pd_z \tvp_k - \al^2 z \tvp_k &= z^{k+1} g^{(k)}(0) \widetilde{\chi}((1+\al)z)\,,  \\
z \pd_z^2 \tvp_\cR +(2m-2) \pd_z \tvp_\cR - \al^2 z \tvp_\cR &= z\cR_g\,. \label{E.ODEvpRg}
\end{align}

Our main estimates are the following:

\begin{lemma} \label{L.keyLemmaODE}
Suppose that $f,g \in H^{\ell}((0,1))$ and $z g \in H^{\ell+1}((0,1))$ with $\ell\geq0$ and fix some $\alpha \geq 0$. The functions $\vp_k,\vp_\cR$ defined by~\eqref{E.ODEvpk}-\eqref{E.ODEvpR} satisfy the $L^2$--estimates
\begin{itemize}
\item[(i.a)] $\|\vp_k\|_{L^2((0,1))} \lesssim (1+\al)^{-k-\frac32} \, |f^{(k)}(0)|\,,$ 
\item[(i.b)] $\|\vp_\cR\|_{L^2((0,1))} + \|z \vp_\cR'\|_{L^2((0,1))}  \lesssim (1+\al)^{-1} \|\cR_f\|_{L^2((0,1))}\,,$
\end{itemize}
and the $H^{\ell+2}$--estimates
\begin{itemize}
\item[(ii.a)]   $\|\vp_k^{(\ell+1)}\|_{L^2((0,1))}+ \|z\vp_k^{(\ell+2)}\|_{L^2((0,1))} \lesssim (1+\al)^{\ell-k-\frac12}\, | f^{(k)}(0)|  \quad \textup{for all } 0 \leq k \leq \ell-1\,,$ 
\item[(ii.b)] $ \displaystyle\|\vp_\cR^{(\ell+1)}\|_{L^2((0,1))} + \|z \vp_\cR^{(\ell+2)}\|_{L^2((0,1))} \lesssim  \sum_{\nu=0}^{\ell} (1+\al)^{\ell-\nu}\, \|\cR_f^{(\nu)}\|_{L^2((0,1))}\,.$ 
\end{itemize}
Similarly, the functions $\tvp_k,\tvp_\cR$ defined by~\eqref{E.ODEvpkg}-\eqref{E.ODEvpRg} satisfy the $L^2$--estimates

\begin{itemize}
\item[(iii.a)] $\|\tvp_k\|_{L^2((0,1))} \lesssim (1+\al)^{-k-\frac52} \, |g^{(k)}(0)|\,,$ 
\item[(iii.b)] $\|\tvp_\cR\|_{L^2((0,1))} + \|z\tvp_\cR'\|_{L^2((0,1))} \lesssim (1+\al)^{-2}\, \|\cR_g\|_{L^2((0,1))}\,,$
\end{itemize}
and the $H^{\ell+3}$--estimates
\begin{itemize}
\item[(iv.a)]  $\|\tvp_k^{(\ell+2)}\|_{L^2((0,1))} + \|z\tvp_k^{(\ell+3)}\|_{L^2((0,1))} \lesssim (1+\al)^{\ell-k-\frac12} \,|g^{(k)}(0)|  \quad \textup{for all } 0 \leq k \leq \ell-1\,,$ 
\item[(iv.b)] $\displaystyle \|\tvp_\cR^{(\ell+2)}\|_{L^2((0,1))} + \|z  \tvp_\cR^{(\ell+3)}\|_{L^2((0,1))} \lesssim \sum_{\nu=0}^{\ell} (1+\al)^{\ell-\nu} \Big( \|\cR_g^{(\nu)}\|_{L^2((0,1))} + \|z\cR_g^{(\nu+1)} \|_{L^2((0,1))} \Big)\,.$
\end{itemize}
All the implicit constants are independent of $\al\in[0,\infty)$.
\end{lemma}

To prove this lemma, we will need the following classical Hardy inequality. We recall that by a weight function, we refer to a function that is measurable and positive almost everywhere.

\begin{theorem}{\rm{(\hspace{-0.003cm}\cite[Theorem 1.1]{Kufner}).}} \label{T.Kufner}
Let $-\infty \leq a < b \leq + \infty$ and let $\om_+,\om_-$ be weight functions on the interval~$(a,b)$. The inequality
	\begin{equation} \label{E.Kufner}
\int_a^b \left( \int_a^t f(\tau) d\tau \right)^2 \omega_-(t) dt  \leq \mathcal{C}  \int_a^b f(t)^2 \omega_+(t) dt \,,
	\end{equation}
holds for all functions $f \geq 0$ if and only if
\begin{equation} \label{E.C'}
	\cC':= \sup_{a<t<b} \left(\int_a^t \omega_+(\tau)^{-1} d\tau\right)^{-1} \left(\int_a^t \omega_-(\tau) \left( \int_a^\tau \omega_+(s)^{-1} ds \right)^{2} d\tau \right)  <\infty\,.
\end{equation}
Moreover, the constant $\mathcal{C} $ in \eqref{E.Kufner} satisfies $\cC' \leq \mathcal{C} \leq 4\cC'$.
\end{theorem}

\begin{remark} \label{R.Kufner}
By duality, \eqref{E.Kufner} is equivalent to the estimate
	\[
	\int_a^b \left( \int_t^b g(\tau)d\tau \right)^{2} \omega_+(t)^{-1} dt \leq \mathcal{C}  \int_a^b g(t)^{2} \omega_-(t)^{-1} dt \,,
	\]
for all $g \geq 0$, with the same constant $\mathcal{C}$. In the proof we shall also use this dual version of Hardy's inequality.
\end{remark}

\begin{proof}[Proof of Lemma \ref{L.keyLemmaODE}]
In the estimates, we need to consider the cases $\al\in[0,1]$ and $\al>1$ separately. Let us start by assuming $\al>1$, since the former case (which follows using similar arguments) is simpler in the sense that there is no need to track the dependence on the large parameter~$\al$. We also assume that $\ell\geq1$, as the case $\ell=0$ is completely analogous but does not require bounds for the functions $\vp_k,\tvp_k$.

To capture the effect on the large parameter~$\al$, we introduce the change of variables $t := \al z$, and write the functions in this variable as 
\[
u_k(t) := \vp_k(z(t))\,,\qquad u_W(t) := \vp_\cR(z(t))\,,\qquad W_f(t) := \cR_f(z(t))\,.
\]
Thus $u_k,u_W\in t^{1-m} H^1_0((0,\alpha))$ satisfy
\begin{align} \label{E.ODEuk}
t \pd_t^2 u_k(t) +(2m-2) \pd_t u_k(t) - t u_k(t) &= \al^{-k-1}t^k f^{(k)}(0) \widetilde{\chi}\Big(\Big(1+\frac{1}{\al}\Big)\,t\Big) \,, \\
\label{E.ODEuR}
t \pd_t^2 u_W(t)+(2m-2) \pd_t u_W(t) - t u_W(t) &= \al^{-1}\, W_f(t)\,.
\end{align}
Since the right hand sides are compactly supported in $[0,\al)$, without loss of generality we can consider these ODEs in the whole half-line $(0,\infty)$ instead of just in~$(0,\alpha)$. 

Note that two linearly independent solutions for the Bessel-type homogeneous equation 
$$
t \pd_t^2 h +(2m-2) \pd_t h - t h=0\qquad \text{in }(0,\infty)\,,
$$
are\begin{equation} \label{E.SolBesselHomogeneous}
h_1(t) := t^{\frac{3}{2}-m} I_{m-\frac32}(t) \quad \textup{ and } \quad h_2(t) := t^{\frac32-m} K_{m-\frac32}(t) \,, 
\end{equation}
where $I_m$ and $K_m$ denote respectively the modified Bessel functions of the first and second kind and of order~$m$. 
These solutions have been chosen so that $h_1(t)$ tends to a nonzero constant at~0 and grows exponentially fast for $t\gg1$, while $h_2(t)$ diverges like $t^{3-2m}$ at 0 but tends to zero exponentially fast at infinity. More precisely, for each~$\nu\geq0$ one has
\[
|h_1^{(\nu)}(t)|\lesssim \widetilde h_1(t)\,,\qquad |h_2^{(\nu)}(t)|\lesssim \widetilde h_{2,\nu}(t)
\]
with
\begin{equation}\label{E.tildeh}
\widetilde h_1(t):= \mathbbm 1_{(0,1)}(t) + t^{1-m} e^{t-1}\mathbbm 1_{(1,\infty)}(t)\,,\qquad  \widetilde h_{2,\nu}(t):= t^{3-2m-\nu}\mathbbm 1_{(0,1)}(t) + t^{1-m} e^{1-t}\mathbbm 1_{(1,\infty)}(t)\,.
\end{equation}
Furthermore, $h_1(t)$ and $t^{2m-3} h_2(t)$ (respectively, $t^{m-1} e^{-t} h_1(t)$ and $t^{m-1} e^{t} h_2(t)$) are smooth functions of~$t\in[0,\infty)$ (respectively, of $t^{-1}\in[0,\infty)$) which do not vanish at the closed endpoint.

From the variation of parameters formula and the asymptotic behavior of~$h_j$, we infer that the functions $u_k,u_W$ admit the representation formulas
\begin{align} \label{E.Repuk}
u_k(t) &= - \al^{-k-1} f^{(k)}(0) \bigg(  h_2(t) \int_0^t s^{2m+k-3} h_1(s)\widetilde{\chi}\Big(\Big(1 +\frac{1}{\al}\Big)s\Big) \,ds \\ & \qquad\qquad\qquad\qquad\qquad\qquad\qquad\qquad+ h_1(t) \int_t^{\infty} s^{2m+k-3} h_2(s) \widetilde{\chi}\Big(\Big(1+\frac{1}{\al}\Big)s\Big) \,ds \bigg)\,,\notag\\
u_W(t) &= -\frac{1}{\al} \left( h_2(t) \int_0^t s^{2m-3} h_1(s)   W_f(s) \,ds  + h_1(t) \int_t^{\infty} s^{2m-3} h_2(s)  W_f(s) \,ds \right)\,.\label{E.RepuW}
\end{align}

Let us  start proving (i.a) and (ii.a) using this representation formula. We start with~$u_k$, which only depends on~$f$ in an extremely simple way. Indeed,
\begin{equation} \label{E.uk-Ukf}
u_{k}(t) =  - \al^{-k-1} f^{(k)}(0)  \, U_{k}(t)\,,
\end{equation}
where the smooth function~$U_k$ is independent of~$f$. Using the bound $|h_j|\leq\widetilde h_j$ and the fact that $\widetilde{\chi}(r) = 0$ for all $r \geq 1$ in the above representation formula, one easily sees that
\begin{equation} \label{Uk-bound}
\int_{0}^{\infty} (1+t^2) (U_{k}^{(p)}(t))^2 dt \lesssim 1
\end{equation}
for any nonnegative integer~$p$. 
Thus, we arrive at the $L^2$-estimate
\begin{align*}
\|\vp_k\|_{L^2((0,1))}^2 = \frac{1}{\al} \|u_k\|_{L^2((0,\al))}^2 = \frac{|f^{(k)}(0)|^2}{\al^{2k+3}} \|U_k\|_{L^2((0,\al))}^2 \leq \frac{|f^{(k)}(0)|^2}{\al^{2k+3}} \|U_k\|_{L^2(\RR_+)}^2 \lesssim \frac{|f^{(k)}(0)|^2}{\al^{2k+3}}\,,
\end{align*}
and and the $H^{\ell+1}$-estimate
\begin{align*}
\|\vp_k^{(\ell+1)}\|_{L^2((0,1))}^2 & = \al^{2(\ell+1)-1}\|u_k^{(\ell+1)}\|_{L^2((0,\al))}^2 \\
& = \al^{2(\ell-k)-1}|f^{(k)}(0)|^2 \|U_k^{(\ell+1)}\|_{L^2((0,\al))}^2 \lesssim \al^{2(\ell-k)-1}|f^{(k)}(0)|^2\,.
\end{align*} 
Similarly, we infer that that
\begin{align*}
\|z \vp_k^{(\ell+2)} \|_{L^2((0,1))}^2 = \al^{2(\ell+1)-1} \|z u_k^{(\ell+2)}\|_{L^2((0,\al))}^2  \lesssim \al^{2(\ell-k)-1}|f^{(k)}(0)|^2\,.
\end{align*}
Combining these three estimates we get (i.a) and (ii.a). 

The proofs of (iii.a) and (iv.a) for~$\tvp_k$ are similar. Indeed, setting 
\[
\tu_k(t) := \tvp_k(z(t))
\]
and arguing as above, we see that
\begin{equation*}
\tu_{k}(t) =  - \al^{-k-2} g^{(k)}(0)  \, \widetilde{U}_{k}(t)\,,
\end{equation*}
with 
\begin{equation*}
\int_{0}^{\infty} (1+t^2) (\widetilde U_{k}^{(p)}(t))^2 dt \lesssim 1
\end{equation*}
for each integer $p\geq0$. Hence (iii.a) and (iv.a) follows using the same reasoning.

We now deal with the regularity of $\vp_\cR$. We start by proving (i.b) and (ii.b). The starting point is a formula for the $p$-th derivative of~$u_W$. Thanks to the representation formula~\eqref{E.RepuW}, fairly direct computations yield the following pointwise estimate:

\begin{lemma}\label{L.pointwise}
	For each integer $p\geq0$, there exists a constant $C > 0$ independent of~$\al$ such that
\begin{equation} \label{E.lDerivativeuRW}
\begin{aligned}
& \al \, |u_{W}^{(p)}(t)| \leq \left|\,  h_2^{(p)}(t) \int_0^t s^{2m-3} h_1(s) \, W_f(s) \,ds\, \right| \\
& \quad + \left|\, h_1^{(p)}(t) \int_t^{\infty} s^{2m-3} h_2(s) \, W_f(s) \,ds \, \right| + C\, p (p-1)
\sum_{\nu = 0}^{p-2} \left(\frac{1}{t^{\,p-1-\nu}}+\frac1t\right) |W_f^{(\nu)}(t)|\,.
\end{aligned}
\end{equation}
\end{lemma}

\begin{proof}
For $p=0,1$, the result is immediate consequence of~\eqref{E.RepuW}. Note that the factor $p(p-1)$ accounts for the fact that the third term in the right hand side of~\eqref{E.lDerivativeuRW} does not appear in these estimates.

Differentiating~\eqref{E.RepuW} twice, one obtains the formula
\begin{align*}
	-\al u_W''(t)&=  h_2''(t) \int_0^t s^{2m-3} h_1(s)   W_f(s) \,ds \\
	& \quad\, + h_1''(t) \int_t^{\infty} s^{2m-3} h_2(s)  W_f(s) \,ds + (h_2'h_1-h_1'h_2)t^{2m-3}W_f(t)\\
	&=h_2''(t) \int_0^t s^{2m-3} h_1(s)   W_f(s) \,ds  + h_1''(t) \int_t^{\infty} s^{2m-3} h_2(s)  W_f(s) \,ds - \frac{W_f(t)}{t}\,.
\end{align*}
To pass to the last line, we have used that the Wronskian of the two solutions is $h_2'(t)h_1(t)-h_1'(t)h_2(t)= -t^{2-2m}$. This yields the formula in the statement in the case $p=2$.

For $p\geq2$, the result follows by differentiating this formula and noting that, for any nonnegative integers $\mu,\nu$, the following pointwise estimate holds
\[
t^{2m-3}\left|h_1^{(\mu)}(t)h_2^{(\nu)}(t)-h_1^{(\nu)}(t)h_2^{(\mu)}(t)\right|\lesssim\begin{cases}
	t^{-\mu-\nu} &\text{if } t<1\,,\\
	t^{-1}&\text{if } t>1\,.
\end{cases}
\]
by the asymptotic formulas for $h_j$ that we established above. Of course, the commutator is zero if $\mu=\nu$.
\end{proof}

The $L^2$ and $H^{\ell+2}$ estimates for~$\vp_\cR$ follow from~\eqref{E.lDerivativeuRW} using suitable Hardy-type estimates. For conciseness, we shall only present the proof of the $H^{\ell+2}$-estimate in detail, as the simpler $L^2$-estimate
$$
\|\vp_\cR\|_{L^2((0,1))} +  \|z \vp_\cR'\|_{L^2((0,1))} \lesssim \frac{1}{\al} \|\cR_f\|_{L^2((0,1))}\,,
$$
follows using the same argument.

Let us start with an estimate for~$\vp_\cR^{(\ell+1)}$. Combining \eqref{E.RepuW} with \eqref{E.lDerivativeuRW}, we get
\begin{align*}
& \|\vp_\cR^{(\ell+1)}\|_{L^2((0,1))}^2 = \al^{2(\ell+1)-1} \|u_W^{(\ell+1)}\|_{L^2((0,\al))}^2 \\
& \quad  \lesssim \al^{2\ell-1} \bigg[ \int_0^\al \left| h_2^{(\ell+1)}(t) \int_0^t s^{2m-3} h_1(s) \, W_f(s) \,ds \right|^2 dt \\
& \hspace{2cm} + \int_0^\al \left|  h_1^{(\ell+1)}(t) \int_t^{\infty} s^{2m-3} h_2(s) \, W_f(s) \,ds \, \right|^2 dt \\
& \hspace{2cm} + \ell(\ell+1) \sum_{\nu = 0}^{\ell-1} \bigg( \int_0^\al  \frac{|W_f^{(\nu)}(t)|^2}{t^{2(\,\ell-\nu)}} dt + \int_0^\al \frac{|W_f^{(\nu)}(t)|^2}{t^2} dt \bigg) \bigg]\,.
\end{align*}
Now let us estimate each term of the right hand side separately. First, observe that
$$
\al^{2\ell-1} \int_0^\al  \frac{|W_f^{(\nu)}(t)|^2}{t^{2(\,\ell-\nu)}} dt  = \int_0^1 \frac{|\cR_f^{(\nu)}(z)|^2}{z^{2(\ell-\nu)}} dz\,.
$$
Moreover, observe that $\cR_f^{(p)}(0)=0$ for all $0 \leq p \leq \ell-1$. Thus, by Theorem \ref{T.Kufner43} applied with $\omega_{-}(z) := z^{-2(\ell-\nu)}$ and $\omega_+(z) := 1$, it follows that
$$
\int_0^1 \frac{|\cR_f^{(\nu)}(z)|^2}{z^{2(\ell-\nu)}} dz \lesssim \int_0^1 |\cR_f^{(\ell)}(z)|^2 dz\,, \quad \textup{ for all } 0 \leq \nu \leq \ell-1\,.
$$
Hence, we have that
\begin{equation} \label{E.ODEr1}
 \ell(\ell+1)\,  \al^{2\ell-1}  \sum_{\nu = 0}^{\ell-1} \int_0^\al  \frac{|W_f^{(\nu)}(t)|^2}{t^{2(\,\ell-\nu)}} dt \lesssim \int_0^1 |\cR_f^{(\ell)}(z)|^2 dz\,.
\end{equation}

Likewise, if follows that
$$
\al^{2\ell-1} \int_0^\al \frac{|W_f^{(\nu)}(t)|^2}{t^2} = \al^{2(\ell-\nu-1)} \int_0^1 \frac{|\cR_f^{(\nu)}(z)|^2}{z^2}dz\,.
$$
Moreover, by Theorem \ref{T.Kufner43} applied with $\omega_-(z) := z^{-2}$ and $\omega_+(z) := 1$ (i.e. by the classical Hardy inequality),  we get that
$$
 \int_0^1 \frac{|\cR_f^{(\nu)}(z)|^2}{z^2}dz \lesssim \int_0^1 |\cR_f^{(\nu+1)}(z)|^2 dz\,.
$$
Thus, we infer that
\begin{equation} \label{E.ODEr1bis}
 \ell(\ell+1)\, \al^{2\ell-1}  \sum_{\nu = 0}^{\ell-1} \int_0^\al  \frac{|W_f^{(\nu)}(t)|^2}{t^{2}} dt \lesssim \sum_{\nu=1}^{\ell} \al^{2(\ell-\nu)}  \int_0^1 |\cR_f^{(\nu)}(z)|^2 dz
\end{equation}

Next, to estimate the first term, we apply Theorem \ref{T.Kufner} with $\omega_{-}(t) := (\widetilde h_{2,\ell+1}(t))^2$ and $\omega_+(t) := t^{6-4m}\widetilde h_1(t)^{-2}(1+t^{-2\ell})$, where we recall that $\widetilde h_1$ and $\widetilde{h}_{2,\ell+1}$ were introduced in~\eqref{E.tildeh}. Indeed, taking into account the asymptotics
$$
\begin{aligned}
& \omega_-(t) \sim t^{6-4m-2(\ell+1)}\,, \quad & \omega_+(t) \sim t^{-2\ell+6-4m}\,, \qquad &  \textup{as } t \to 0^+\,, \\
& \omega_-(t) \sim t^{2-2m} e^{-2t}\,, \quad & \omega_+(t) \sim t^{4-2m} e^{-2t}\,, \qquad & \textup{as } t \to \infty\,,
\end{aligned}
$$
one can check that
\begin{align*}
\left(\int_0^t \omega_+(\tau)^{-1} d\tau\right)^{-1} \left(\int_0^t \omega_-(\tau) \left( \int_0^\tau \omega_+(s)^{-1} ds \right)^{2} d\tau \right) & \lesssim 1\,, \quad \textup{as } t \to 0^+\,, \\
\left(\int_0^t \omega_+(\tau)^{-1} d\tau\right)^{-1} \left(\int_0^t \omega_-(\tau) \left( \int_0^\tau \omega_+(s)^{-1} ds \right)^{2} d\tau \right) & \lesssim \frac{1}{t^2}\,, \quad \textup{as } t \to \infty\,.
\end{align*} 
Note that in the last estimate we are using that, for $\tau \gg 1$,
$$
\int_1^\tau s^{\al} e^{\beta s} ds \sim \tau^{\al} e^{\beta \tau}\,, \quad \textup{ for all } \al, \beta \in \RR\,.
$$
Thus, we infer that \eqref{E.C'} holds, and so that
\begin{equation} \label{E.ODEr2}
\begin{aligned}
\alpha^{2\ell-1}\int_0^\al & \left| h_2^{(\ell+1)}(t) \int_0^t s^{2m-3} h_1(s) \, W_f(s) \,ds \right|^2 dt \\
& \lesssim  \al^{2\ell-1} \int_0^\al \left(\frac{h_1(t)}{\widetilde h_1(t)}\right)^2 \left( \frac{|W_f(t)|^2}{t^{2\ell}} + |W_f(t)|^2 \right) dt \\
& \lesssim \al^{2\ell-1} \int_0^\al  \left( \frac{|W_f(t)|^2}{t^{2\ell}} + |W_f(t)|^2 \right)  dt =  \int_0^1  \frac{|\cR_f(z)|^2}{z^{2\ell}} dz + \al^{2\ell} \int_0^1 |\cR_f(z)|^2 dz \,.
\end{aligned}
\end{equation}

Finally, applying now the dual Hardy inequality in Remark \ref{R.Kufner} with $\omega_-(t) := t^{4m-6}\, \widetilde h_{2,0}(t)^2$ and $\omega_+(t) := (\widetilde h_1(t))^{-2}$ we estimate the second term. Observe that we now have the asymptotics
$$
\begin{aligned}
& \omega_-(t) \sim 1\,, \quad && \omega_+(t) \sim 1\,, \qquad &  \textup{as } t \to 0^+\,, \\
& \omega_-(t) \sim t^{2m-4} e^{-2t}\,, \quad && \omega_+(t) \sim t^{2m-2} e^{-2t}\,, \qquad & \textup{as } t \to \infty\,,
\end{aligned}
$$
and thus
\begin{align*}
\left(\int_0^t \omega_+(\tau)^{-1} d\tau\right)^{-1} \left(\int_0^t \omega_-(\tau) \left( \int_0^\tau \omega_+(s)^{-1} ds \right)^{2} d\tau \right) & \lesssim t^2\,, \quad \textup{as } t \to 0^+\,, \\
\left(\int_0^t \omega_+(\tau)^{-1} d\tau\right)^{-1} \left(\int_0^t \omega_-(\tau) \left( \int_0^\tau \omega_+(s)^{-1} ds \right)^{2} d\tau \right) & \lesssim \frac{1}{t^2}\,, \quad \textup{as } t \to \infty\,.
\end{align*} 
Hence, we get that
\begin{equation} \label{E.ODEr3}
\al^{2\ell-1} \int_0^{\al} \left| h_1^{(\ell+1)}(t) \int_t^\infty s^{2m-3} h_2(s) W_f(s) ds \right|^2 dt \lesssim \al^{2\ell} \int_0^1 |\cR_f(z)|^2 dz\,.
\end{equation}
Combining \eqref{E.ODEr1}--\eqref{E.ODEr3}, we thus find
\begin{equation} \label{E.ODEr4}
\|\vp_\cR^{(\ell+1)}\|_{L^2((0,1))} \lesssim \sum_{\nu = 0}^{\ell} \al^{\ell-\nu}\, \|\cR_f^{(\nu)}\|_{L^2((0,1))}\,.
\end{equation}

Let us now estimate $z \vp_\cR^{(\ell+2)}$. By \eqref{E.lDerivativeuRW},  
\begin{align*}
& \|z \vp_\cR^{(\ell+2)}\|_{L^2((0,1))}^2 = \al^{2(\ell+1)-1} \|t u_W^{(\ell+2)}\|_{L^2((0,\al))}^2 \\
& \quad  \lesssim \al^{2\ell-1} \bigg[ \int_0^\al \left|t h_2^{(\ell+2)}(t) \int_0^t s^{2m-3} h_1(s) \, W_f(s) \,ds \right|^2 dt \\
& \hspace{2cm} + \int_0^\al \left|  t h_1^{(\ell+2)}(t) \int_t^{\infty} s^{2m-3} h_2(s) \, W_f(s) \,ds \, \right|^2 dt  \\
& \hspace{2cm} + (\ell+1)(\ell+2) \sum_{\nu = 0}^{\ell} \bigg( \int_0^\al  \frac{|W_f^{(\nu)}(t)|^2}{t^{2(\,\ell-\nu)}} dt + \int_0^\al |W_f^{(\nu)}(t)|^2 dt \bigg) \bigg]\,.
\end{align*}
Hence, we can argue exactly as in the proof of~\eqref{E.ODEr4} to estimate each term of the right hand side, thereby showing that
\begin{equation} \label{E.ODEr5}
\|z\vp_\cR^{(\ell+2)}\|_{L^2((0,1))} \lesssim \sum_{\nu=0}^{\ell} \al^{\ell-\nu }\, \|\cR_f^{(\nu)}\|_{L^2((0,1))}\,.
\end{equation}
Combining \eqref{E.ODEr4} and \eqref{E.ODEr5} we obtain (ii.b).

To prove the estimates (iii.b) and (iv.b), one argues just as in the case of (i.b) and (ii.b), discussed above in detail, after replacing $f$ by $zg$. In fact, writing
\[
\widetilde u_W(t) := \tvp_\cR(z(t))\,,\qquad W_g(t) := \cR_g(z(t))\,,
\]
and using that $z(t) \cR_g(z(t))= \frac t \al W_g(t)$, one can write a representation formula for $\widetilde{u}_W$:
\begin{equation} \label{E.RepFormulauWg}
\widetilde u_W(t) = -\frac{1}{\al^2} \left( h_2(t) \int_0^t s^{2m-2} h_1(s)   W_g(s) \,ds + h_1(t) \int_t^{\infty} s^{2m-2} h_2(s)  W_g(s) \,ds \right)\,. 
\end{equation}
Lemma \ref{L.pointwise} thus yields the pointwise bound
\begin{equation*}  
\begin{aligned}
\al^2 \, & |\widetilde{u}_{W}^{(p)}(t)| \leq \left|\,  h_2^{(p)}(t) \int_0^t s^{2m-2} h_1(s) \, W_g(s) \,ds\, \right| + \left|\, h_1^{(p)}(t) \int_t^{\infty} s^{2m-2} h_2(s) \, W_g(s) \,ds \, \right| \\
& + C\, p (p-1) \bigg(
\sum_{\nu = 0}^{p-2} \Big( \frac{1}{t^{p-2-\nu}}+1 \Big) |W_g^{(\nu)}(t)| + \sum_{\nu=0}^{p-3} \frac{1}{t} \, |W_g^{(\nu)}(t)| + \sum_{\nu=1}^{p-2} \sum_{\mu=0}^{\nu-1} \frac{1}{t^{p-1-\nu}} |W_g^{(\mu)}(t)| \bigg)\,.
\end{aligned}
\end{equation*}
so one can argue essentially as in the proof of (i.b) and (ii.b) to obtain the bounds (iii.b) and (iv.b).
\end{proof} 

\subsection{PDE estimates}

We shall next show how we can use Lemma~\ref{L.keyLemmaODE} to derive the key regularity estimates for the only solution $v \in z^{1-m}H_0^1((0,1)\times \TT)$ to the equation
\begin{equation}\label{E.LvGbis}
	\cL v = G\,,
\end{equation}
which we stated as Lemma~\ref{L.keyest} in the previos section. In fact, in the following lemma we state a somewhat more detailed set of a priori estimates; Lemma \ref{L.keyest} stems from them simply by keeping track of the factors of~$\ep$ in the definition of the norms (Equation~\eqref{E.normsDisk}).

Equivalently, in terms of the Fourier components
\begin{equation} \label{E.FourierVn}
v_n(z):=\frac1{2\pi}\int_0^{2\pi} v(z,\theta)\, e^{-in\theta}\,d\theta\,, \qquad G_n(z):=\frac1{2\pi}\int_0^{2\pi} G(z,\theta)\, e^{-in\theta}\,d\theta\,,
\end{equation}
one has the system of ODEs
\begin{equation}\label{E.LnvG}
	\cL_n v_n=G_n\qquad \text{in } (0,1)\,,
\end{equation}
with
\begin{equation*}
	\cL_n u:= z\pd_z^2 u -\frac{\ep^2n^2}{a_-^2}z\,u +(2m-2)\pd_z u\,.
\end{equation*}

\begin{lemma} \label{L.keyestReformulated}
For any integer $j\geq1$, the unique solution $v \in z^{1-m}H_0^1((0,1)\times \TT)$  to \eqref{E.LvG} is bounded as
\begin{align} \label{E.keyestR1}
& \sum_{j_2=0}^{j-1} \sum_{j_1=0}^{j-j_2} \Big( \|\pd_z^{j_1} \pd_\te^{j_2} v \|_{L^2((0,1) \times \TT)} + \|z \pd_z^{j_1+1} \pd_\te^{j_2} v \|_{L^2((0,1) \times \TT)} \Big)  \lesssim \|G\|_{H^{j-1}((0,1) \times \TT)} \,, \\ \label{E.keyestR2}
& \|\pd_\te^{j} v\|_{L^2((0,1) \times \TT)} + \|z\pd_z \pd_\te^j v\|_{L^2((0,1) \times \TT)}   \lesssim \frac{1}{\ep}\, \|\pd_\te^{j-1} G\|_{L^2((0,1) \times \TT)} \,.
\end{align}
Additionally, when $G(z) = \ep z H(z)$,  and $j \geq 2$, one has 
\begin{align} \label{E.keyestR3}
& \sum_{j_2=0}^{j-1} \sum_{j_1=0}^{j-j_2} \Big( \|\pd_z^{j_1} \pd_\te^{j_2} v \|_{L^2((0,1) \times \TT)}  + \|z \pd_z^{j_1+1} \pd_\te^{j_2} v \|_{L^2((0,1) \times \TT)} \Big) \\ \nonumber
  & \hspace{4.8cm} \lesssim \|H\|_{H^{j-2}((0,1) \times \TT)} + \|z \pd_z H\|_{H^{j-2}((0,1) \times \TT)} \,, \\ \label{E.keyestR4}
& \|\pd_\te^{j} v\|_{L^2((0,1) \times \TT)} + \|z\pd_z \pd_\te^j v\|_{L^2((0,1) \times \TT)}  \lesssim \frac{1}{\ep}\, \|\pd_\te^{j-2} H\|_{L^2((0,1) \times \TT)}\,.
\end{align}
The implicit constants are independent of $\ep$, and the dependence on $a_-$ is uniform in compact subsets of $[1,+\infty)$.
\end{lemma}

\begin{proof} Let us divide the proof into two steps for the sake of clarity. \medbreak

\noindent \textbf{Step 1: Reduction to ODE estimates.} Taking the Fourier transform in the $\te$-variable, let us write the function $G \in H^{j-1}((0,1)\times \TT)$ as
$$
G(z,\te) = \sum_{n \in \ZZ} G_n(z) e^{in\te}\,, \qquad G_n(z) := \frac{1}{2\pi} \int_0^{2\pi} G(z,\te) e^{-in\te} d\te  \,.
$$
The unique solution~$v\in z^{1-m}H_0^1((0,1)\times \TT)$ to~\eqref{E.LvGbis} is therefore
\begin{equation}\label{E.seriesv}
	v(z,\te) = \sum_{n \in \ZZ} v_n(z) e^{in\te} \,,
\end{equation}
where each Fourier component is the only solution $v_n\in z^{1-m}H_0^1((0,1))$ to the ODE
\begin{equation} \label{E.LnvnGn}
\cL_n v_n = G_n \qquad \textup{in } (0,1)\,,
\end{equation}
where
\begin{equation*}
	\cL_n u:= z\pd_z^2 u  +(2m-2)\pd_z u-\al_n^2 z\,u\,.
\end{equation*}
Note that this is the ordinary differential operator~\eqref{E.ODEvp} that we studied in Subsection \ref{S.keyestODE}, and that the parameter
\[
\al_n:=  \frac{\ep |n|}{ a_-}\in[0,\infty)\,,
\]
now depends on the Fourier frequency~$n$ and on the scale~$\ep$. 

Given an integer $j_1 \geq 1$, for each Fourier mode we argue as in  Subsection~\ref{S.keyestODE}, decomposing $G_n$ as
$$
G_n(z) =: \sum_{k=0}^{j_1-2} \frac{z^k}{k!}G_n^{(k)}(0) \widetilde{\chi}((1+\al_n)z) + \cR_{G_n}(z)\,.
$$
Thus the function $v_n \in z^{1-m}H_0^1((0,1))$ can be similarly written as 
\begin{equation} \label{E.vnDecomposition}
v_n = \sum_{k=0}^{j_1-2} \frac{1}{k!} \, v_{k,n} + v_{\cR,n}\,,
\end{equation}
where $v_{k,n},v_{\cR,n} \in z^{1-m} H_0^1((0,1))$ are the only solutions to the ODEs
\begin{align} \label{E.ODEvkn}
\cL_n v_{k,n} &= z^{k}G_n^{(k)}(0) \widetilde{\chi}((1+\al_n)z) \,,\\
\label{E.ODEvRn}
\cL_n v_{\cR,n} &= \cR_{G_n}
\end{align}
in $(0,1)$. Note that the index~$k$ will eventually range from~0 to~$j_1-2$.

The case where $G(z)= \ep z H(z)$ with $H,z\pd_zH \in H^{j-2}((0,1) \times \TT)$ is handled the same way. We write~$H$ as a Fourier series
 $$
 H(z,\te) = \sum_{n \in \ZZ} H_n(z) e^{in\te}\,,\qquad H_n(z) :=  \frac{1}{2\pi} \int_0^{2\pi} H(z,\te) e^{-in\te} d\te\,,
 $$
and note that the Fourier coefficients in~\eqref{E.seriesv} can be similarly written as~\eqref{E.vnDecomposition}, where now $\tv_{k,n}, \tv_{\cR,n} \in z^{1-m} H_0^1((0,1))$ are the only solutions to
\begin{align} \label{E.ODEvknHn}
\cL_n \tv_{k,n} &= \ep z^{k+1}H_n^{(k)}(0) \widetilde{\chi}((1+\al_n)z) \,,\\
\label{E.ODEvRnHn}
\cL_n \tv_{\cR,n} &= \ep z \cR_{H_n}\,.
\end{align}
Here, $\cR_{H_n}$ is defined as
$$
H_n(z) =: \sum_{k=0}^{j_1-2} \frac{z^k}{k!}H_n^{(k)}(0) \widetilde{\chi}((1+\al_n)z) + \cR_{H_n}(z)\,,
$$
and
$$
v_n = \sum_{k=0}^{j_1-2} \frac{1}{k!} \tv_{k,n} + \tv_{\cR,n}\,.
$$
 
\noindent \textbf{Step 2: Sum over Fourier modes.} 
Lemma \ref{L.keyestReformulated} now follows from Lemma \ref{L.keyLemmaODE} (with $\al_n := \ep |n|/ a_-$) using the Parseval identity.

First of all, observe that Lemma \ref{L.keyLemmaODE} (i.b) applied with $\ell:=0$ implies that
\begin{equation} \label{E.keyestRC1}
\begin{aligned}
& \|v\|_{L^2((0,1)\times \TT)}^2 + \|z\pd_z v\|_{L^2((0,1)\times \TT)}^2  \\ 
& \quad = 2\pi \sum_{n \in \ZZ} \Big( \|v_n\|_{L^2((0,1))}^2 + \|z \pd_z v_n\|_{L^2((0,1))}^2 \Big) \lesssim  2\pi \sum_{n \in \ZZ} \|G_n\|_{L^2((0,1))}^2 = \|G\|_{L^2((0,1)\times \TT)}^2\,,
\end{aligned}
\end{equation}
and that, for all $j \geq 1$,
$$
\begin{aligned}
& \|\pd_\te^{j}v\|_{L^2((0,1)\times \TT)}^2 + \|z \pd_z \pd_\te^{j}v\|_{L^2((0,1)\times \TT)}^2 \\[0.3cm]
& \quad  = 2\pi \sum_{n \in \ZZ} n^{2j} \Big( \|v_n\|_{L^2((0,1))}^2 + \|z \pd_z v_n\|_{L^2((0,1))}^2 \Big) \\
& \quad \lesssim  \frac{2\pi}{\ep^2} \sum_{n \in \ZZ}n^{2(j-1)} \|G_n\|_{L^2((0,1))}^2 = \frac{1}{\ep^2}\, \|\pd_\te^{j-1}G\|_{L^2((0,1)\times \TT)}^2\,.
\end{aligned}
$$
Likewise, Lemma \ref{L.keyLemmaODE} (iii.b) applied with $\ell :=0$ implies that, for all $j \geq 2$,
$$
\|\pd_\te^{j}v\|_{L^2((0,1)\times \TT)}^2 + \|z \pd_z \pd_\te^{j}v\|_{L^2((0,1)\times \TT)}^2 \lesssim \frac{1}{\ep^2}\, \|\pd_\te^{j-2} H\|_{L^2((0,1) \times \TT)}^2\,.
$$

Once we have proved \eqref{E.keyestR1} and \eqref{E.keyestR3}, we focus on the more involved estimates \eqref{E.keyestR2} and \eqref{E.keyestR4}. Applying Lemma \ref{L.keyLemmaODE} (ii.b) with $\ell := 0$, we obtain that, for all nonnegative integer $j_2$ with $j_2 \leq j-1$
\begin{equation} \label{E.keyestRC2}
\begin{aligned}
& \|\pd_z \pd_\te^{j_2} v\|_{L^2((0,1)\times \TT)}^2 + \|z\pd_z^2 \pd_z^{j_2} v\|_{L^2((0,1)\times \TT)}^2 \\[0.3cm]
& \quad = \sum_{n \in \ZZ} n^{2j_2} \Big( \|\pd_z v_n\|_{L^2((0,1))}^2 + \|z \pd_z^2 v_n\|_{L^2((0,1))}^2 \Big)\\
& \quad \lesssim 2\pi \sum_{n \in \ZZ} n^{2j_2} \|G_n\|_{L^2((0,1))}^2 = \|\pd_\te^{j_2} G\|_{L^2((0,1)\times \TT)}^2\,.
\end{aligned}
\end{equation}
Likewise, Lemma \ref{L.keyLemmaODE} (iv.b) applied $\ell := 0$ implies that, for all nonnegative integer $j_2$ with $j_2 \leq j-1$, it follows that
\begin{equation} \label{E.keyestRC4}
\|\pd_z^2 \pd_\te^{j_2} v \|_{L^2((0,1)\times \TT)}^2 + \|z \pd_z^3 \pd_\te^{j_2} v\|_{L^2((0,1)\times \TT)}^2 \lesssim \|\pd_\te^{j_2} H\|_{L^2((0,1)\times \TT)}^2 + \|z \pd_z \pd_\te^{j_2} H\|_{L^2((0,1)\times \TT)}^2\,.
\end{equation}
 
We finally deal with the case where $j_1$ and $j_2$ are nonnegative integers with $j_1 \geq 2$  and $0 \leq j_2 \leq j-j_1$. By direct computations, we get  
\begin{equation} \label{E.keyestRCTranslation}
\begin{aligned}
& \|\pd_z^{j_1} \pd_\te^{j_2} v\|_{L^2((0,1)\times \TT)}^2 + \|z\pd_z^{j_1+1} \pd_\te^{j_2} v \|_{L^2((0,1)\times \TT)}^2 \\[0.3cm]
& \quad = 2\pi \sum_{n \in \ZZ} n^{2j_2}  \Big( \|v_n^{(j_1)}\|_{L^2((0,1))}^2 + \|z v_n^{(j_1+1)}\|_{L^2((0,1))}^2 \Big) \\
& \quad \lesssim \sum_{n \in \ZZ} n^{2j_2} \bigg( \sum_{k=0}^{j_1-2} \frac{1}{k!} \Big( \|v_{k,n}^{(j_1)}\|_{L^2((0,1))}^2 + \|z v_{k,n}^{(j_1+1)}\|_{L^2((0,1))}^2 \Big) \\
& \hspace{7cm}  + \|v_{\cR,n}^{(j_1)}\|_{L^2((0,1))}^2 + \|z v_{\cR,n}^{(j_1+1)}\|_{L^2((0,1))}^2 \bigg)\,. 
\end{aligned}
\end{equation}
On one hand, by Lemma \ref{L.keyLemmaODE} (ii.a) applied with $\ell := j_1-1$, we infer that 
\begin{align*}
& \sum_{n \in \ZZ} n^{2j_2} \bigg( \sum_{k=0}^{j_1-2} \frac{1}{k!} \Big( \|v_{k,n}^{(j_1)}\|_{L^2((0,1))}^2 + \|z v_{k,n}^{(j_1+1)}\|_{L^2((0,1))}^2 \Big) \\
& \quad  \lesssim \sum_{k=0}^{j_1-2} \frac{1}{k!} (1+n^2)^{j_1-1+j_2-k-\frac12} |G_n^{(k)}(0)|^2 \lesssim \|\pd_z^k G(0,\cdot)\|_{H^{j_1+j_2-k-\frac32}(\TT)}^2\,.
\end{align*}
On the other hand, by Lemma \ref{L.keyLemmaODE} (ii.b) applied with $\ell := j_1-1$, we get that 
\begin{equation*}
\begin{aligned}
\sum_{n \in \ZZ} n^{2j_2} \big( \|v_{\cR,n}^{(j_1)}\|_{L^2((0,1))}^2 + \|z v_{\cR,n}^{(j_1+1)}\|_{L^2((0,1))}^2 \big) \lesssim \sum_{n \in \ZZ} n^{2j_2} \sum_{\nu=0}^{j_1-1} \bigg( (1+n^2)^{j_1-1-\nu}  \|\cR_{G,n}^{(\nu)}\|_{L^2((0,1))}^2 \bigg)
\end{aligned}
\end{equation*}
Thus, it follows that
\begin{equation} \label{E.keyestRC3}
\|\pd_z^{j_1} \pd_\te^{j_2} v\|_{L^2((0,1)\times \TT)} + \|z\pd_z^{j_1+1} \pd_\te^{j_2} v \|_{L^2((0,1)\times \TT)} \lesssim \|G\|_{H^{j-1}((0,1) \times \TT)}\,.
\end{equation}
The bound \eqref{E.keyestR1} then follows from \eqref{E.keyestRC1}, \eqref{E.keyestRC2} and \eqref{E.keyestRC3}.

Let us now assume that $j_1 \geq 3$. On one hand, by Lemma \ref{L.keyLemmaODE} (iv.a) applied with $\ell := j_1-2$, we get that 
\begin{align*}
& \sum_{n \in \ZZ} n^{2j_2} \bigg( \sum_{k=0}^{j_1-2} \frac{1}{k!} \Big( \|\tv_{k,n}^{(j_1)}\|_{L^2((0,1))}^2 + \|z \tv_{k,n}^{(j_1+1)}\|_{L^2((0,1))}^2 \Big) \\
& \quad  \lesssim \sum_{k=0}^{j_1-2} \frac{1}{k!} (1+n^2)^{j_1-2+j_2-k-\frac12} |H_n^{(k)}(0)|^2 \lesssim \|\pd_z^k H(0,\cdot)\|_{H^{j_1+j_2-k-\frac52}(\TT)}^2\,.
\end{align*}
On the other hand, by Lemma \ref{L.keyLemmaODE} (iv.b) applied with $\ell := j_1-2$, we get  
\begin{align*}
& \sum_{n \in \ZZ} n^{2j_2} \Big( \|\tv_{\cR,n}^{(j_1)}\|_{L^2((0,1))}^2 + \|z \tv_{\cR,n}^{(j_1+1)}\|_{L^2((0,1))}^2 \Big) \\
& \qquad \lesssim \sum_{n \in \ZZ} n^{2j_2} \sum_{\nu=0}^{j_1-2} \bigg( (1+n^2)^{j_1-2-\nu} \Big( \|\cR_{H,n}^{(\nu)}\|_{L^2((0,1))}^2  + \|z \cR_{H,n}^{(\nu+1)}\|_{L^2((0,1))}^2 \Big) \bigg) \,.
\end{align*}
Thus, substituting into \eqref{E.keyestRCTranslation} with $v_{k,n}$ replaced by $\tv_{k,n}$, and $v_{\cR,n}$ replaced by $\tv_{\cR,n}$, we conclude that
\begin{equation} \label{E.keyestRC5}
\|\pd_z^{j_1} \pd_\te^{j_2} v\|_{L^2((0,1) \times \TT)} + \|z \pd_z^{j_1+1} \pd_\te^{j_2}\|_{L^2((0,1) \times \TT)} \lesssim \|H\|_{H^{j-2}((0,1) \times \TT)} + \|z \pd_z H\|_{H^{j-2}((0,1) \times \TT)}\,.
\end{equation}
Combining \eqref{E.keyestRC1}, \eqref{E.keyestRC4} and \eqref{E.keyestRC5}, we thus obtain \eqref{E.keyestR3}. The lemma then follows.
\end{proof}

\section{Spectral properties of $\fL_a$}
\label{S.bifurcation}

In this section we study the spectral properties of the operator $\fL_a$, depending on the parameters $a$ and $\ell$. Specifically, we prove Propositions \ref{eigen-crossing} and \ref{P.nonradial}. These results show a phenomenon of crossing of eigenvalues, which is essential for our bifurcation result.

The proof of Proposition \ref{eigen-crossing} will follow from several lemmas. Let us recall the positive symmetric bilinear form $\cB: H_0^1(\Om_a) \times H_0^1(\Om_a) \to \RR$ associated to the linear operator $\fL_a$, namely
\begin{align*}
	\cB(v,w)& :=  -\int_{\Omega} \big( L_a^{0,0} v + f_a'(\Phi_{a,1}, \widetilde{\psi}_a) v \big) w\, (R+a-4)\, d R \, d \theta \\
	& =  \int_{\Om} \left[ \pd_R v\, \pd_R w+ \frac{\pd_\theta v\,\pd_\theta w}{(\erre)^2} -  f_a'(\Phi_{a,1}, \widetilde{\psi}_a) v  w \right](\erre) \, dR \, d\theta\,.
\end{align*}
Also, we recall that there exists a sequence of eigenvalues of~$-\fL_a$ with finite multiplicity, which we denote by $\lambda_k \equiv \lambda_k(a)$, such that
$$ 
\lambda_1 < \lambda_2 \leq \lambda_3 \leq \cdots 
$$
Of course, $\la_k(a)$ tends to infinity as $k \to \infty$, and depends continuously on $a$.

As the potential is radial, the radial eigenfunctions of this operator, which define an orthonormal basis of $L^2((1,7),(\erre)\, dR)$, can be equivalently obtained from the symmetric bilinear form $ \cB^{\rm rad}: H_0^1(1,7))\times H_0^1(1,7))\to\RR$ given by
\[
\cB^{\rm rad}(v,w):=\int_{1}^7 \left[ \pd_R v\, \pd_R w -  f_a'(\Phi_{a,1}, \widetilde{\psi}_a) v  w \right](\erre) \, dR\,.
\]
The corresponding quadratic form is $\cQ^{\rm rad}(v):=\cB^{\rm rad}(v,v)$.

We denote the eigenvalues corresponding to the radial eigenfunctions by 
$$ 
\lambda_1^{\rm rad}(a) < \lambda_2^{\rm rad}(a) < \cdots < \lambda_k^{\rm rad}(a) < \cdots 
$$
When there is no risk of confusion, we simply write $\lambda_k^{\rm rad}\equiv\lambda_k^{\rm rad}(a)$. It is standard that $\lambda_k^{\rm rad}(a)\to\infty$ as $k\to\infty$, and that $\lambda_k^{\rm rad}(a)\neq \lambda_j^{\rm rad}(a)$ whenever $k\neq j$. To see this, note that, the eigenvalue equation for a radial eigenfunction $\phi$ of~$\fL_a$ (which, by Proposition~\ref{P.regDGa}, is in $\rho^{m-1} \cX^j$) reduces to the ODE. The classical argument showing that radial Dirichlet eigenvalues have multiplicity~1 then follows from  Frobenius' theory for ODEs.

We start the proof of Proposition \ref{eigen-crossing} with the following lemma:

\begin{lemma} \label{l1} For all $a \geq 4$, $\lambda_1^{\rm rad}(a)< \la_2^{\rm rad}(a)<0$.
\end{lemma}

\begin{proof} By composing with the diffeomorphism $\Phi_a$, it suffices to show that the quadratic form 
	\begin{equation}
	\tilde{\cQ}(v):= \int_{a_-}^{a^+} \left [ (\pd_r v)^2-f_a'(r, \psi_a) v^2  \right ] r\, dr\,,
	\end{equation}
is negative definite on a two dimensional subspace of $H_0^1((a_-, a_+))$. 	Recall that $\psi_a$ solves the equation
$$ 
\pd_{r}^2 \psi_a + \frac{\pd_r\psi_a}{r} +f_a(r, \psi_a(r))=0 \quad \textup{ in } (a_-,a_+)\,.
$$
Moreover, taking the derivative with respect to $r$, we get that $\xi(r) :=  \pd_r \psi_a(r)$ solves
$$
\pd_{r}^2\xi + \frac{\pd_r \xi}{r} - \frac{\xi}{r^2} +f_a'(r, \psi_a(r))\xi=0 \quad \textup{ in } (a_-,a_+)\,.
$$
It is worth stressing that, in the above expression, we are using that $\pd_rf_a(r, \psi_a(r)) \equiv 0$ in $(a_-,a_+)$.

Next, we recall that $m_a$ is the unique maximum of the function $\psi_a$, and define
$$
\xi_1(r):= \left\{
\begin{aligned}
& \pd_r \psi_a(r)\,, \ && r \in [a_-,m_a)\,,\\
& 0 && r \in [m_a, a_+)\,,
\end{aligned}
\right.
\quad
\textup{ and }
\quad 
\xi_2(r):= \left\{
\begin{aligned}
& 0\,, \ && r \in [a_-,m_a]\,,\\
& \pd_r \psi_a(r) && r \in (m_a, a_+]\,.
\end{aligned}
\right.
$$	
Clearly, $\xi_1$, $\xi_2 \in H_0^1((a_-, a_+))$ and
$$
\begin{aligned}
\tilde{\cQ}(\xi_1) & := \int_{a_-}^{m_a} \left [ (\pd_r \xi_1)^2\, - f_a'(r, \psi_a) \xi_1^2  \right ] r\, dr= \int_{a_-}^{m_a} \left( r (\pd_r \xi_1)^2\, +  r \xi_1 \pd_{r}^2 \xi_1  + \xi_1 \pd_{r} \xi_1   - \frac{\xi_1^2}{r} \right) dr  \\
& =\int_{a_-}^{m_a} \left( \pd_r(r \xi_1 \pd_r \xi_1) -\frac{\xi_1^2}{r} \right) dr  = -\int_{a_-}^{m_a} \frac{\xi_1^2}{r} \, dr<0\,.
\end{aligned}
$$
Analogously, it follows that
	$$ \tilde{\cQ}(\xi_2) <0\,.$$

Since $\xi_1$ and $\xi_2$ have disjoint support, we conclude that
		$ \tilde{\cQ}(\xi)<0\,, $ for all $\xi \in {\rm span} \{\xi_1, \ \xi_2\} \setminus \{0\}\,.$
This implies that $\lambda_2^{\rm rad}(a)<0$ and concludes the proof of the lemma.
\end{proof}

We are now interested in the behavior of the eigenvalues $\lambda_k^{\rm rad}$ as $a \to +\infty$. On this purpose, let us recall (see Lemma \ref{L.f2}) that the limit function $\overline{\psi}$ is a solution to the limit problem 
$$ 
\opsi'' + \of(\opsi) =0  \quad \textup{ in } (1,7)\,.
$$
We denote the bilinear and quadratic forms associated with the linearized operator, and defined on functions belonging to $ H_0^1((1,7))$ by
$$
\overline{\cB}(v,w) := \int_1^7 \left( v'(R) w'(R) - \of'(\opsi(R)) v(R)w(R) \, \right) dR\,, \qquad \overline{\cQ}(v):= \overline{\cB}(v,v)\,.
$$
Just as in the case of the radial eigenvalues of~$\fL_a$, it is standard that there is an orthonormal basis of $L^2((1,7))$ consisting of eigenfunctions, whose eigenvalues we label as
$$ 
\overline{\la}_1 < \overline{\la}_2 < \cdots < \overline{\la}_k < \cdots
$$
Furthemore, the eigenvalues tend to infinity and are simple: $\overline{\la}_j\neq \overline{\la}_k$ for $j\neq k$.

We have the following important convergence result:

\begin{lemma} \label{l2} As $a \to \infty$, the following limits hold true:
	$$ \la_1^{\rm rad}(a) \to \overline{\la}_1<0, \qquad  \la_2^{\rm rad}(a) \to \overline{\la}_2=0, \qquad \la_3^{\rm rad}(a) \to \overline{\la}_3>0 \,. $$
\end{lemma}

\begin{proof}
	
	We first observe that, for any $v\in H^1_0((1,7))$,
	\begin{align*}
		\cQ^{\rm rad}(v) & =  \int_1^7 \left[ (\pd_R v)^2 - f_a'(R+a-4, \psi_a(R+a-4)) v^2  \right](\erre) \, dR  \\ & = a \left[ \int_1^7 \left[ (\pd_R v)^2- f_a'(R+a-4, \psi_a(R+a-4)) v^2 \right] \, dR + o(1) \| v \|_{H^1(1,7)}^2\right].
	\end{align*}
Moreover, by Lemma \ref{L.f2} (iv), we have that 

\begin{align*} 
& \left | \int_1^7 \left[ (\pd_R v)^2- f_a'(R+a-4, \psi_a(R+a-4)) v^2 \right] \, dR - \int_1^7 \left[ (\pd_R v)^2 - \of'(\opsi(R)) v^2 \right] \, dR \, \right | \\
& \quad =   \bigg| \int_1^7 \Big (- f_a'(R+a-4, \psi_a(R+a-4))+ \of'(\psi_a(R+a-4)) \Big ) v^2 \, dR \\
& \qquad\ + \int_1^7 \Big( - \of'(\psi_a(R+a-4)) + \of'(\opsi(R)) \Big )  v^2 \, dR\, \bigg|  \\ 
& \quad \leq  \frac{m-1}{a}\int_1^7 \frac{v^2}{\rho} \, dR + o(1)\, \| v \|_{H^1(1,7)}^2 =   o(1) \| v \|_{H^1((1,7))}^2\,, \quad \textup{ as } a \to +\infty\,, 
\end{align*}
and so that
$$
\cQ^{\rm rad}(v) = a \, \big[\,\overline{\cQ}(v) +  o(1)\|v\|_{H^1((1,7))}^2\,\big]  \quad \textup{ as } a \to +\infty\,.
$$
On the other hand, it is easy to see that 
$$ 
 \int_1^7 v(R)^2 (R+a-4)  \, dR  = a\,[1+o(1)]  \int_1^7 v(R)^2  \, dR   \quad \textup{ as } a \to +\infty\,.
$$
This implies the convergence of the corresponding Rayleigh quotients, in the sense that
$$
\frac{\cQ^{\rm rad}(v)}{\|v\|^2_{L^2((1,7),\,(R+a-4)  \, dR)}}\, \xrightarrow[]{a \to +\infty}\, \frac{\overline{\cQ}(v)}{\ \|v\|^2_{L^2((1,7))}}\,, \quad  \mbox{ for any fixed radial function } v\,.
$$
	
	By the min-max characterization of the eigenvalues we infer that 
	 $\la_k^{\rm rad} \to \overline{\la}_k$, as $a \to + \infty$, for any fixed $k$. Thus, to conclude the proof we only need to verify that $\overline{\la}_2=0$. On that purpose, observe that $\opsi'\in H_0^1((1,7))$ is a solution of the linearized problem
$$ \xi'' + \of'(\opsi)\xi =0\,.$$
Hence, $\opsi'$ is an eigenfunction with associated eigenvalue $0$. Since $\opsi'$ changes sign exactly once, we conclude that $\overline{\la}_2=0$.
\end{proof}

As we can see in the previous result, the first radial eigenvalue is negative and remains bounded away from $0$ as $a \to + \infty$. This is the reason for the appearance of nonradial degeneracies, as we shall see. Our analysis has to be precise enough to avoid interferences with the second radial eigenvalue, which is converging to $0$ from below. In what follows we consider general eigenvalues with $\ell$-symmetry of the operator $-\fL_a$, as defined in \eqref{symmetry} and below. In next lemma we fix the values of $a_0$ and $\ell$ which will be used throughout this section.

\begin{lemma} \label{l3} Let $c := \frac{1}{3} \min \{-\overline{\la}_1, \overline{\la}_3 \}$. According to Lemma \ref{l2} we take $a_0 \geq 4$ so that   $|\lambda_1^{\rm rad} -\overline{\lambda}_1|<c$ and $|\lambda_3^{\rm rad} -\overline{\lambda}_3|<c$.
	Then, there exists $\ell \in \NN$ such that $\lambda_3^{\ell}(a_0)>0$.
\end{lemma}

\begin{proof}
	If $\lambda_3^\ell(a_0) =\lambda_3^{\rm rad}(a_0)$ then it is positive and we are done. If not, the corresponding eigenfunction $\phi_3$ can be written in Fourier as
	$$ 
	\phi_3(R,\te) = \sum_{k \in \NN} \varphi_k(R) \cos ( k \ell\theta)\,.
	$$
Hence, it suffices to show that for all functions of the form $v(R, \theta) = \varphi(R) \cos ( k \ell \theta)$ the quadratic form $\cQ(v)$ is strictly positive for $\ell \in \NN$ sufficiently large. It turns out that, for $v$ as above,
\begin{align*}
		\cQ(v) & =  \int_1^7 \int_0^{2\pi} \left[ (\pd_R v)^2 +\frac{(\pd_\theta v)^2}{(\erre)^2} - f_a'(R+a-4, \psi_a(R+a-4)) v^2  \right](\erre) \, dR \, d\theta  \\ & =  \pi \int_1^7 \left[ (\pd_R \varphi)^2 + \frac{\ell^2k^2 }{(\erre)^2} \varphi^2 - f_a'(R+a-4, \psi_a(R+a-4)) \varphi^2 \right] (\erre)\, dR.
	\end{align*}
Thus, if we choose $\ell \in \NN$ such that $\frac{\ell^2 }{(a_0-3)^2} + \overline{\lambda}_1  -2c >0$, we conclude that, for $v$ as above,
		\begin{align*}
		\cQ(v) & =  \pi \int_1^7 \left[ (\pd_R \varphi)^2 + \frac{\ell^2k^2 }{(\erre)^2} \varphi^2 - f_a'(R+a-4, \psi_a(R+a-4)) \varphi^2 \right](\erre) \, dR \\ & \geq \pi \int_1^7 \left[ (\pd_R \varphi)^2 + (c-\la_1^{\rm rad}) \varphi^2 - f_a'(R+a-4, \psi_a(R+a-4)) \varphi^2 \right] (\erre) \, dR \\ & \geq \pi c \int_1^7 \varphi^2 \, (\erre) \, dR >0\,.\end{align*}
The lemma then follows.
\end{proof}

In the next lemma we show that $\lambda_3^\ell(a)$ becomes negative if $a > a_0$ is large enough.

\begin{lemma} Let $a_0$, $\ell$ and $c$ as in Lemma \ref{l3}. Then, there exists $a_2>a_0$ such that $\lambda_3^\ell(a_2)<0$.
\end{lemma} 

\begin{proof}
Let $a_2>a_0$ be such that $ \frac{\ell^2}{(a_2-3)^2} +\bar{\la}_1 + 2c < 0$, and let $\phi_1$ be the eigenfunction associated to $\la_1(a_2)= \la_1^{\rm rad}(a_2)$. Then, we choose and fix the test function
$$
v(R, \theta) := \phi_1(R) \cos (\ell \theta)\,.
$$
By construction, $v$ is $L^2$-orthogonal to the two radial eigenfunctions $\phi_1$, $\phi_2$ associated to the eigenvalues $\la_1^{\rm rad}(a_2)$, $\la_2^{\rm rad}(a_2)$. Hence, it suffices to show that $\cQ(v) <0$. Arguing as in the proof of  Lemma \ref{l3}, we then conclude that	
\begin{align*}
		\cQ(v) & =  \pi \int_1^7 \left[ (\pd_R \phi_1)^2 + \frac{\ell^2 }{(\erre)^2} \phi_1^2 - f_a'(R+a-4, \psi_a(R+a-4)) \phi_1^2 \right](\erre) \, dR \\ & \leq \pi \int_1^7 \left[ (\pd_R \phi_1)^2 + (-c-\la_1^{\rm rad}) \phi_1^2 - f_a'(R+a-4, \psi_a(R+a-4)) \phi_1^2 \right] (\erre) \, dR \\ & = - \pi c \int_1^7 \phi_1^2 \, (\erre) \, dR <0\,,
\end{align*}
and the lemma follows.
\end{proof}

We can now prove Proposition \ref{eigen-crossing}.

\begin{proof}[Proof of Proposition \ref{eigen-crossing}]
	
	We fix $a_0$, $c$, $\ell$ and $a_2$ as in the previous lemmas and take $\ep_0>0$ such that $\ep_0 < \min \{ c, - \la^{\rm rad}_2(a)\}$ for all $a \in [a_0, a_2]$. This can be ensured thanks to Lemma \ref{l1}. Then, we set
$$
A:= \big\{ a \in \RR, \ a \geq a_0: \ \lambda_3^\ell(a) <0\,\big\}\,.
$$
Clearly, $a_2 \in A$, which is then a non empty set, and so we can define 
$$
\alpha := \inf A\,.
$$
By definition, $\lambda_3^\ell(\alpha)=0$ and $\lambda_3^\ell(a) \geq 0$ for all $a \in [a_0,\alpha]$.
	
Now, given $\ep \in (0, \ep_0)$, we can take $a_1>\alpha$ sufficiently close to $\alpha$ so that
	\begin{enumerate}
		\item[(a)] $\lambda_3^\ell(a_1)<0$.
		\item[(b)] $\la_3^\ell(a) > -\ep$ for all $a \in [a_0, a_1]$.
	\end{enumerate}
	
We have then proved assertions (i) and (ii) of Proposition \ref{eigen-crossing}. We now turn our attention to (iii). Observe that, by our choice of $\ep_0$ and $c$ in Lemma \ref{l3}, we have $\lambda_3^{\rm rad}(a)>0$ for any $a \in [a_0,a_1]$. Hence, if $\lambda_3^\ell(a) \leq 0$, it corresponds to a nonradial eigenfunction $\phi_3$. Being $\phi_3$ the first nonradial eigenfunction, it must be written as $\phi_3(R, \theta) = \phi(R) \cos (\ell \theta)$, where $\phi(R)$ is the eigenfunction corresponding to the first eigenvalue (which is equal to $0$) associated to the bilineal form $\cB^{\ell}:H_0^1((1,7)) \times H_0^1((1,7))  \to \RR$, given by
\begin{align*}
	\cB^{\ell}(v,w)& :=  \int_{1}^7 \left[ \pd_R v\, \pd_R w+ \frac{\ell^2 }{(\erre)^2} v w -  f_a'(\Phi_{a,1}, \widetilde{\psi}_a) v  w \right](\erre) \, dR\,.
\end{align*}

Again, this is a one-dimensional problem and its eigenfunction $\phi$ must be unique. We thus conclude that $\lambda_3^\ell(a)$ is simple. Also, the function $\phi$ must be positive, since it corresponds to the first eigenvalue. On the other hand, by Proposition \ref{P.regDGa}, we know that $\phi_3 = \rho^{m-1} w$ for some $w \in \cX^j$. Moreover, $w$ is continuous and cannot vanish on the boundary of $\Omega$. This implies the desired estimate $\phi \geq C \rho^{m-1}$ for some $C>0$.

We finally show the validity of  Proposition \ref{eigen-crossing} (iv). We first prove that $\la_4^\ell(a)>0$ for any $a \in [a_0, \alpha]$. Recall that for any $a \in [a_0, \alpha]$, $\la_3^\ell(a) \geq 0$. The claims follows immediately if $\la_3^\ell(a)>0$. Instead, if $\la_3^\ell(a)=0,$ we conclude by its simplicity, which has been proved above, and the strict inequality holds. It suffices now to take $a_1$ closer to $\alpha$, if necessary, so that (iv) is satisfied, and the proof is concluded.
\end{proof}

\bigbreak

We conclude this section with the proof of Proposition \ref{P.nonradial}.

\begin{proof}[Proof of Proposition \ref{P.nonradial}]
Let $(w_n)_{n=1}^{\infty} \subset \cX_\ell^j$ be as in \eqref{E.sequencewn}. First of all, observe that, for all $n \in \NN$,
\begin{equation} \label{E.meanvaluewn}
\begin{aligned}
0 & = \cG_{a_n}(w_n)\\
& = \cG_{a_n}(0) + \int_0^1 D\cG_{a_n}(tw_n) w_n\, dt  = D\cG_{a_n}(0) w_n + \int_0^1 \big( D\cG_{a_n}(tw_n) - D\cG_{a_n}(0) \big) w_n \, dt\,. 
\end{aligned}
\end{equation}
Moreover, by Lemma \ref{L.DGa}, we have that
$$
\left\|\int_0^1 \big( D\cG_{a_n}(tw_n) - D\cG_{a_n}(0) \big) w_n \, dt\, \right\|_{\cY^j} = o(\|w_n\|_{\cX^j})\,, \quad \textup{as } n \to \infty\,.
$$
Hence, by Proposition \ref{P.regDGa}, it follows that, for all $n \in \NN$ sufficiently large,
\begin{equation} \label{E.equivalencewn}
\|w_n\|_{\cY^j} \lesssim \|w_n\|_{\cX^j} \lesssim \|\rho^{m-1} w_n\|_{H^1(\Om)} \lesssim \|w_n\|_{\cY^j}\,.
\end{equation}

Now, we set $u_n := w_n/\|w_n\|_{\cX^j}$ for all $n \in \NN$. Since $(u_n)_{n=1}^\infty$ is a bounded sequence, up to a subsequence if necessary, we have
$$
u_n \rightharpoonup u_0 \quad \textup{in } \cX^j\,, \quad \textup{and} \quad u_n \to u_0 \quad \textup{in } \cY^j\,,
$$
for some $u_0 \in \cX^j$. Note that the strong convergence in $\cY^j$ follows from the compact embedding proved in Lemma \ref{L.CompactEmbedding}. Furthermore, taking into account \eqref{E.equivalencewn}, we infer that $u_0 \not \equiv 0$. 

On the other hand, arguing as in \eqref{E.meanvaluewn}, and using the continuous dependence of $D\cG_{a}(0)$ with respect to $a$, and the convergence $a_n \to a^*$, we get that
$$
-D\cG_{a^*}(0) u_0 = 0 = \lambda_3^{\ell}(a^*) \rho\, u_0\,.
$$
Hence, we conclude that $\rho^{m-1} u_0 \equiv \phi_3$, with $\phi_3$ as in Proposition \ref{eigen-crossing}, and so that $u_0 \not \equiv 0$ on $\pd \Om$. Even more, we infer that $u_0(R,\te) = \widetilde{\phi}(R) \cos(\ell\te)$ for some radial function $\widetilde{\phi}$ with $\inf_{1 < R < 7}\, \widetilde{\phi}(R) > 0$. This allows us to conclude that, for all $n \in \NN$ sufficiently large, the functions $b_{w_n}$ and $B_{w_n}$ given in \eqref{E.wto} with $w := w_n$ are nonconstant. 
\end{proof}





\section*{Acknowledgements}
This work has received funding from the European Research Council (ERC) under the European Union's Horizon 2020 research and innovation programme through the grant agreement~862342 (A.E.). A.E. is also supported by the grant PID2022-136795NB-I00 of the Spanish Science Agency and the ICMAT--Severo Ochoa grant CEX2019-000904-S. A.J.F. is partially funded by Proyecto de Consolidación Investigadora 2022, CNS2022-135640, MICINN. D.R. has been supported by:
the Grant PID2021-122122NB-I00 of the MICIN/AEI, the \emph{IMAG-Maria de Maeztu} Excellence Grant CEX2020-001105-M funded by MICIN/AEI, and the Research Group FQM-116 funded by J. Andaluc\'\i a.

\appendix

\section{Smooth compactly supported solutions given by \\ elliptic equations with autonomous nonlinearities are locally radial}
\label{A.nogo}

Our objective in this short appendix is to show that one cannot obtain non-radial smooth stationary Euler flows with compact support using the usual formulation in terms of (autonomous) semilinear elliptic equations.

To see this, suppose that $v\in C^1(\RR^2)$ is a classical solution of the stationary Euler equations~\eqref{E.Euler} on~$\RR^2$, and that there exists a region with $C^1$~boundary $\Om\subset\RR^2$ such that
\begin{equation}\label{E.vA}
v=\begin{cases}
	\nabla^\perp \bar\psi &\text{in }\Om,\\
	0 &\text{in }\RR^2\backslash\Om.
\end{cases}
\end{equation}
Here $\bar\psi\in C^2(\BOm)$ is a solution of a semilinear elliptic equation of the form
\begin{equation}\label{E.semiA}
\Delta \psi + f(\psi)=0
\end{equation}
in the domain~$\Om$, for some continuous function $f\in C(\RR)$. We assume that $\Om$ is bounded, and that $\pd\Om$ consists of $J\geq1$ connected components~$\Ga_j$.

The first observation is that one can equivalently assume that $v$ is globally given by the perpendicular gradient of a solution of the semilinear equation~\eqref{E.semiA}:

\begin{proposition}\label{P.support}
The vector field~\eqref{E.vA} can be equivalently written as $v=\nabla^\perp\psi$, where $\psi\in C^2(\RR^2)$ satisfies Equation~\eqref{E.semiA} in all of~$\RR^2$.
\end{proposition}

\begin{proof}
Since $v\in C^1(\RR^2)$, it follows that $\nabla\bar\psi|_{\pd\Om}=0$ and $\nabla^2\bar\psi|_{\pd\Om}=0$, which in particular implies that $\Delta\bar\psi|_{\pd\Om}=0$, and that there exist constants $c_j$ such that $\bar\psi|_{\Gamma_j}=c_j$. Here, $\{\Gamma_j\}_{j=1}^J$ denote the connected components of $\pd \Om$. By Equation~\eqref{E.semiA} and the continuity of~$f$, one must then have $f(c_j)=0$ for all $1 \leq j \leq J$. Let us then denote by $\Om_1,\ldots, \Om_J$ the different connected components of $\RR^2 \setminus \overline{\Om}$, relabeling the boundary components if necessary so that $\pd \Om_j = \Gamma_j$, and define the $C^2$ function
	\begin{equation*} 
\psi:=\begin{cases}
	 \bar\psi &\text{in }\Om\,,\\
	c_j &\text{in }\Om_j \text{ for each } 1\leq j\leq J\,.
\end{cases}
\end{equation*}
The functions $\Delta\psi$ and $f(\psi)$ are continuous on~$\RR^2$, identically zero in~$\RR^2\backslash\overline{\Om}$, and coincide with $\Delta\bar\psi$ and $f(\bar\psi)$, respectively, in~$\Om$, where they satisfy~\eqref{E.semiA} by hypothesis. This immediately implies that~$\psi$ is a solution to~\eqref{E.semiA} in the whole of~$\RR^2$ as claimed.
\end{proof}

\begin{theorem}\label{T.nogoA}
Let $\psi\in {C^3(\RR^2)}$ satisfy a semilinear equation of the form~\eqref{E.semiA} in~$\RR^2$ with $f\in C(\RR)$. If $\nabla\psi$ is compactly supported, then $\psi$ is locally radial. In other words, define: $$D = \{ x \in \RR^2: \ \nabla \psi(x) \neq 0\}.$$ Then,
$$ 
D = \bigcup_{i \in I} A_i\,,
$$
where $I$ is a countable set and where $\{A_i\}_{i\in I}$ are disjoint annuli or disks. Moreover, $\psi$ is radially symmetric in each set~$A_i$.
\end{theorem}

\begin{proof} Since $\nabla \psi$ has compact support, we have that $\psi$ is constant outside a bounded set. By adding a constant if necessary, we can assume that $\psi$ itself has compact support. Let us also point out that $f(0)=0$. We now consider separately two cases.

\medbreak \noindent \textit{Case 1: $\psi$ does not change sign.} The case where $\psi \geq 0$ immediately follows from \cite{brock}. Obviously so does the case $\psi \leq 0$, simply by considering the function $-\psi$. 

\medbreak \noindent \textit{Case 2: $\psi$ changes sign.} We start by proving that $\nabla \psi= 0$ on the zero set $\psi^{-1}(0)$. For this, we argue by contradiction, and assume that $\psi(p)=0$ and $\nabla \psi(p) \neq 0$ for some $p \in \RR^2$. Then, we follow the ideas of \cite{Hamel1, Hamel2}: for some $\delta >0$, let $ \sigma: (-\delta, \delta) \to \RR^2$ be the solution to the ODE
$$ 
\left \{ \begin{aligned}  \sigma'(t) & = \nabla \psi (\sigma(t))\,, \quad t \in (-\delta, \delta)\,, \\ \sigma(0)& =p\,.  \end{aligned}  \right.
$$
As $\nabla \psi(p) \neq 0$, we have that $g'(0)>0$, where $g= \psi \circ \sigma$. By taking $\delta>0$ smaller if necessary and suitable $\ep_i >0$, we have that $g: (-\delta, \delta) \to (-\ep_1, \ep_2 )$ is a diffeomorphism. In this way, we can write $f$ as
$$f|_{(-\ep_1, \ep_2)}= - \Delta \psi \circ \sigma \circ g^{-1}.$$
Since $\psi$ is $C^3$ in $\Omega$, then $f|_{(-\ep_1, \ep_2)}$ is a $C^1$ function. But this ensures that the problem $\Delta \psi + f(\psi)=0$ has a unique continuation property on the zero level set, which implies that $\psi\equiv 0$, which is a contradiction. Hence, we conclude that $\nabla \psi(x) = 0$ at every point satisfying $\psi(x) = 0$. 

Next, let us define 
$$\psi_+(x) := \max \{ \psi(x), 0\} = \left \{\begin{aligned} & \psi(x) && \text{if }x \in \Omega_+\,, \\ & 0  && \text{if }x \notin \Omega_+\,, \end{aligned} \right.
$$
where $\Omega_+ :=\{x \in \Omega: \ \psi(x)>0\}$. Clearly, $\psi=0$ on $\partial \Omega_+$, so it follows that $\nabla \psi=0$ on $\partial \Omega_+$. This implies that $\psi_+ \in C^1(\RR^2)$, and then it follows that it is a classical nonnegative solution to \eqref{E.vA}. As in \textit{Case 1}, we then conclude that $\psi_+$ is locally radial. 

The same argument can be applied to $\psi_- (x):= \max \{-\psi(x), 0\}$, so the theorem follows.
\end{proof}

\bibliographystyle{amsplain}

\end{document}